\documentclass[12pt]{amsart}
\usepackage{mathrsfs}
\usepackage[dvipsnames]{xcolor}
\usepackage[latin1]{inputenc}
\usepackage{geometry}
\usepackage{amsmath}
\usepackage{amscd}
\usepackage{amstext}
\usepackage{amsbsy}
\usepackage{amsopn}
\usepackage{amssymb}
\usepackage{amsthm}
\usepackage{amsxtra}
\usepackage{verbatim}

\title[Analytic functions relative to a
covariance map $\eta$]{Analytic functions relative to a
covariance map $\eta$: \\I. Generalized Haagerup products and analytic relations}
\author[Yoann Dabrowski]{Yoann Dabrowski}
\thanks {Research partially supported by ANR Grant NEUMANN}

\def\R{\mbox{I\hspace{-.15em}R} }

\def\C{\mathbb{C}}

\def\I{\mathscr{J}}

\theoremstyle{plain}
      \newtheorem{theorem}{Theorem}
      \newtheorem{lemma}[theorem]{Lemma}
      \newtheorem{corollary}[theorem]{Corollary}
      \newtheorem{proposition}[theorem]{Proposition}
     
      \theoremstyle{definition}
      \newtheorem{definition}[theorem]{Definition}
      \theoremstyle{hypothesis}

     \theoremstyle{remark}
     \newtheorem{remark}[theorem]{Remark}
     \theoremstyle{Fact}
     
     \theoremstyle{Example}
     \newtheorem{exemple}[theorem]{Example}
\theoremstyle{notation}

\begin{document}
\maketitle
\begin{abstract}
We generalize module weak-* Haagerup tensor products to obtain complete quotients of normal Haagerup tensor product included in canonical Hilbert spaces associated to completely positive normal (covariance) maps $\eta$ on a  finite von Neumann algebra $B$. We construct in this way dual operator spaces, providing new examples even in the case of module extended Haagerup tensor products. This is the basis for defining a matrix normed algebra of analytic functions that captures the relations of free semicircular variables with covariance $\eta$. We prove that a class of non-commutative random variables having finite Fisher information relative to $\eta$ have also no analytic relations among our class of analytic functions.
\end{abstract}

\section*{Introduction}

In a fundamental series of papers, Voiculescu introduced analogs of entropy and Fisher information in the context of free probability theory. A first microstate free entropy $\chi(X_{1},...,X_{n})$ is defined as a normalized limit of the volume of sets of microstate i.e. matricial approximations (in moments) of the n-tuple of self-adjoints $X_{i}$ living in a (tracial) $W^{*}$-probability space $M$. Starting from a definition of a free Fisher information \cite{Vo5}, Voiculescu also defined a non-microstate free entropy $\chi^{*}(X_{1},...,X_{n})$. For more details, we refer the reader to the survey \cite{VoS} for a list of properties as well as applications of free entropies in the theory of von Neumann algebras.

Morally speaking, finite entropy is a (strengthened) substitute for ``absolute continuity" with respect to free semicircular variables, which are the reference variables. Most assumptions of results applying free probability to von Neumann algebras are related to free entropy or free Fisher information. Especially, the most recent isomorphism results using monotone transport  \cite{alice-shlyakhtenko-transport} uses analytic assumptions on conjugate variables (the free analogue of score function) and obtains analytic transport maps.

Moreover, starting from \cite{Sp98}, semicircular variables relative to a subalgebra have been studied in the framework of conditional free probability. Those variables have been further investigated from a von Neumann algebra viewpoint, for instance in \cite{S99}. In this context, for a family $(S_i)_{i\in I}$ of semicircular variables over a von Neumann subalgebra $B$, the covariance map on $B$ defined by $\eta_{ij}(b)=E_B(S_ibS_j)$ is the basic parameter. In general, any $\eta:B\to B\otimes B(\ell^2(I))$ normal completely positive map can be obtained in such a way. However, we will only consider the tracial case where $B$ has faithful normal tracial state and where $\eta$ is $\tau$-symmetric, i.e. $\tau(\eta_{ij}(b)c)=\tau(b\eta_{ji}(c)),b,c\in B$ which is a necessary and sufficient condition so that the algebra generated by $B$ and $S_i$ to be tracial (see \cite[Prop 2.20]{S99}). 

The corresponding relative entropy has been introduced in \cite{Shl00}, see also section 3.3 bellow for a reminder.

Contrary to the setting with $\eta=\tau$ corresponding to ordinary free entropy, that prevents analytic relations \cite[lemma 37]{dabrowski:SPDE} (see also more recent developments in \cite{MS14,S14}), the general case has to deal with lots of relations. The covariance map $\eta$ encodes relations of $B$ and $S_i$. For instance, when $\eta:B\to B$ is a (trace preserving) conditional expectation $\eta=E_D,D\subset B$, then the variable $S$ commutes with $D$, and even $B$ and $S$ are free with amalgamation over $D$ so that the commutation relation is basically the only relation. 

Trying to generalize transport for semicircular variables with covariance thus implies to deal with these relations in a analytic function setting. The goal of this  paper and the following series is to develop the analytic tools needed in this respect. Motivated by an ongoing joint work with Guionnet and Shlyakhtenko on transport, our analysis will be based on Haagerup tensor products (see sections 1.1, 1.2 for reminders).

Even though the nice commutation relations of projective tensor product with $\ell^1$ direct sum would maybe suggest its use as a first choice for analytic functions, its lack of injectivity and its complicated kernel, when mapped to minimal tensor product in absence of approximation properties, makes it hard to use to capture relations.
On the other hand, being both projective and injective, having canonical maps $C\otimes_h C\to C$ for any operator algebra, the Haagerup tensor product will be much easier to use. The projective feature will be useful for the algebraic properties required by analytic functions while the injective feature will help taking care of relations. We will still need to capture our relations in a new variant of this tensor product in the general covariance case $\eta.$

The case of the already studied module Haagerup tensor product that captures the commutation relation, the case $\eta=E_D$ for us, will be our starting point and an important tool for our generalization. For obtaining weak-* compactness, we look for dual operator spaces. The study of module extended Haagerup tensor product with this feature has been started by Magajna in an impressive series of papers culminating in \cite{M97,M05} with a convenient duality theory. However, contrary to the non-module case, we don't have in general equality but only  $X\otimes_{eh D}Y\subset (X_\natural\otimes_{h D'} Y_\natural)^\natural$ for an appropriate notion of module predual $X_\natural$ and module dual $X^\natural$. Our first main result expresses how this issue does not appear in the situation we are interested in where $X,Y$ are finite von Neumann algebras (Theorem \ref{MainHaagerupModule}). At this stage, we will know $X\otimes_{eh D}Y$ is both a complete quotient of a huger non module Haagerup tensor product, the normal Haagerup tensor product $X\otimes_{\sigma h} Y$ and included, in the case $X=Y=M$ finite von Neumann algebra, in a canonical Hilbert space $\overline{BSB}^{L^2(W^*(S,B))}$ related to the semicircular $S$ of covariance $id_B$. 

Those two features will give our guiding line in the general case to get our Haagerup tensor product relative to $\eta$. We will look for a complete quotient included in a Hilbert space related to semicircular variable of covariance $\eta$ . Moreover the weak-* topology will be induced on bounded sets by the one of the Hilbert space and we will indeed obtain dual operator spaces. The extended Haagerup tensor product will also reappear to obtain a substitute of associativity in this context in Theorem \ref{SubmoduleEtaTHM}.

Finally, we will be able to introduce a class of analytic functions based on these Haagerup tensor products of subsection 2.2. Section 3 will explain the most general results, and obtain evaluation maps and free difference quotients. Then Theorem \ref{MainRelation} will show that the relations captured by these techniques are indeed related to the relative free entropy setting of \cite{Shl00}, generalizing partially \cite[lemma 37]{dabrowski:SPDE} and \cite{MS14,S14}. We should emphasize that our results only says that finite Fisher information prevents having more analytic relations than semicircular variables, it does not give any answer about the interesting conjecture that they should exactly the same relations, maybe under more restrictive assumptions on conjugate variables. A second paper in this series will deal with more technical results on analytic functions for applications to transport.

Let now describe the content in detail. The paper contains 3 sections after this introduction. Section 1 contains preliminary material, mostly using operator space techniques. Section 1.1 and 1.2 give background on Haagerup and module Haagerup tensor products. Section 1.3 starts the specialization to finite von Neumann algebras and describe module preduals in that case. Since analytic functions will need an appropriate type of $\ell^1$ direct sums, we describe an obvious candidate in section 1.4 keeping stable the class of operator modules. Section 1.5 describes various shuffle maps between various tensor products, mostly of projective type (or various duals of projective or nuclear tensor products) and Haagerup type (including normal and extended), that will be useful to define evaluations of our analytic functions. We also need in that respect a projective product adapted to operator modules. Technically functoriality often reduces some spaces to be taken only a matrix space. We also have to keep functoriality of the construction in order to be able to use several times iteratively on analytic functions. Subsection 1.6 and 1.7 are essentially technical. They use some matricially normed space structure (which are not operator spaces) on duals to obtain a density result enabling to prove the automatic normality result in section 2.1. To give an intuition, we need to put matrix norms natural as algebras of $CB$ maps (for compostion) on our finite von Neumann algebra instead algebras of bounded operators on a Hilbert space. This enables us to exploit reflexivity of $L^2$ in a case where considering only bounded operators on a Hilbert space make arrive annoyingly in $L^1$. Hopefully, the original paper of Blecher and Paulsen on Haagerup tensor product gave a guiding line for the use of such more general matricially normed spaces.

The result in section 2.1 is then easily obtained and we can, in that way, identify various module extended Haagerup products with duals of module Haagerup products without extra normality conditions. This will be crucial to exploit a dual operator space structure on them later. Section 2.2 then defines the generalized Haagerup products relative to covariance maps, and  expresses $n$-ary tensor products in terms of ordinary module extended Haagerup product of $2$-ary products. This will be what will serve as substitute of associativity since extended Haagerup product is associative. In the same way as multiplication with module products make appear a notion of commutant, we need to develop the appropriate substitute in section 3.3 for the case with covariance maps.

We are then ready to introduce analytic functions. We start by a universal class using normal Haagerup tensor products in section 3.1. It enables to keep weak-* continuity of evaluation. Then we introduce a natural complete quotient associated to our Haagerup tensor product with a covariance map. We can build in section 3.2  evaluation maps and free difference quotients while keeping some weak-* density of ordinary polynomials and some weak-* continuity properties. Finally, we recall free entropy relative to our covariance map in section 3.3 and explain our absence of analytic relation result. Most of the preliminary material has been developed in the universal context of section 3.1 and for building evaluation maps. The proof is thus quite similar to the case $\eta=\tau$, once the appropriate analytic machinery has been developed.

\section{Preliminaries}

\subsection{Haagerup tensor products}

We will use variants of Haagerup tensor product  $A\otimes_hA$ of some $C^*$ algebra $A$ as a building block for analytic functions.
Recall, there is always a multiplication $m:A\otimes_h A\to A$ and this characterizes operator algebras (see e.g \cite[Th 6.2]{PisierBook}).
 The Haagerup tensor product is well known to be associative, non-commutative, projective and injective.
We recall the following well known properties of this tensor product that will make it well suited for our purposes. We mostly refer to \cite{PisierBook} or to \cite{EffrosRuan} for details.

We now recall more sophisticated results about dual Haagerup tensor products. Even if we won't recall definitions, let us remind there are mostly three dual Haagerup tensor products : the weak-* Haagerup tensor product \cite{BlecherSmith} of dual operator spaces such that $X^*\otimes_{w^*\text{h}}Y^*=(X\otimes_h Y)^*$, its extension by injectivity the extended Haagerup tensor product $X\otimes_{eh}Y$ (cf e.g. \cite{EffrosRuanDual}) and finally a dual version of this last tensor product the normal Haagerup tensor product \cite{EffrosKishimoto} again defined only for dual operator spaces $X^*\otimes_{\sigma\text{h}}Y^*=(X\otimes_{eh} Y)^*.$

\begin{theorem}\label{DualHaagerup}
\begin{enumerate}
\item \cite{BlecherSmith,EffrosRuanDual} For any operator spaces $X_1,...,X_n$ there are completely isometric inclusions  $X_1\otimes_h...\otimes_hX_n\hookrightarrow  X_1\otimes_{eh}...\otimes_{eh}X_n$ and  $$(X_1\otimes_h...\otimes_hX_n)^*\simeq X_1^*\otimes_{eh}...\otimes_{eh}X_n^*\hookrightarrow  (X_1^*\otimes_{\sigma h}...\otimes_{\sigma h}X_n^*)\simeq (X_1\otimes_{eh}...\otimes_{eh}X_n)^*$$ the second inclusion having as retract the natural weak-* continuous projection dual of the first. Moreover $\otimes_{eh}$ and $\otimes_{\sigma h}$ are associative, define contractive tensor products of maps and if $X_i\subset Y_i$ completely isometrically, then $X_1\otimes_{eh}...\otimes_{eh}X_n\subset Y_1\otimes_{eh}...\otimes_{eh}Y_n,$ completely isometrically (i.e. $\otimes_{eh}$ is injective, but it is not projective)
\item \cite{EffrosKishimoto} For $M\subset B(H),N\subset B(K)$ von Neumann subalgebras, the lateral multiplication map $(x\otimes y)\mapsto (\psi_0(x\otimes y):T\mapsto xTy)$ extends uniquely to a weak-* $\sigma$-weakly homeomorphic isometric $(M,N)$-bimodule isomorphism of $M\otimes_{\sigma h} N$ with $(M',N')$ bimodular completely bounded maps on $B(K,H)$ : $CB_{M',N'}(B(K,H),B(K,H)).$

As a consequence, $(x\otimes y)\#(x'\otimes y')=(xx'\otimes y'y)$ extends to a Banach algebra multiplication on  $M\otimes_{\sigma h} N$  corresponding to composition in the isomorphim above.
\item \cite{BlecherSmith} For $M\subset B(H),N\subset B(K)$ von Neumann subalgebras, the lateral multiplication map $(x\otimes y)\mapsto (\psi_0(x\otimes y):T\mapsto xTy)$ extends uniquely to a complete isometric weak-* homeomorphic $(M,N)$-bimodule isomorphism of $M\otimes_{eh} N\simeq M\otimes_{w^*h} N$ with $(M',N')$ bimodular completely bounded maps from compacts $\mathcal{K}(K,H)$ to $B(K,H)$ : $CB_{M',N'}(\mathcal{K}(K,H),B(K,H)).$

Especially, there is a extension of the multiplication map $m:M\otimes_{\sigma h}M\to M.$

Moreover $(x\otimes y)\#(x'\otimes y')=(xx'\otimes y'y)$ extends to a Banach algebra multiplication on  $M\otimes_{eh} N$ weak-* continuous in the first but not in the second variable. 
\item \cite{BlecherSmith} For any von Neumann algebras $M,N$ the canonical map $M\otimes_{eh}N\to M\overline{\otimes}N$ to the von Neumann tensor product is an injection.

 As a consequence, if $M_0\subset M,N_0\subset N$ finite, we have inclusions $M\otimes_{h}N\subset M\otimes_{eh}N \subset M\overline{\otimes}N\subset L^2(M\otimes N,\tau)$ such that the projection $P_{L^2(M_0\otimes N_0)}$ restricted to $M\otimes_{eh}N$ equals $E_{M_0}\otimes_{eh}E_{N_0}$ and restricted to $M\otimes_{h}N$ equals $E_{M_0}\otimes_{h}E_{N_0}$.
 \end{enumerate}
\end{theorem}

The reader should note that from the proof by duality of the last statement, the crucial injection to the von Neumann tensor product of $M\otimes_{eh} N$ is not valid for $M\otimes_{\sigma h} N$ in general. Looking for such type of inclusion will be crucial for us to generalize those various Haagerup tensor products.

\subsection{Haagerup tensor products of $D$-modules.}\label{HaagD}
There are mainly 3 kinds of Haagerup tensor products of operator modules over a $C^*$ algebra $D$ which will be usually a von Neumann algebra. 

We follow \cite{M05} for notation.
For $X$ a right $D$ operator module (written $X\in OM_D$)  and $Y$ a left $D$ operator module, $X\otimes_{hD} Y$ is the quotient of $X\otimes_h Y$ by the closed subspace generated by elements of the form $xd\otimes y-x\otimes d y$, $x\in X,y\in Y, d\in D,$ cf. \cite[section 3.4]{BLM}. If $D$ is a von Neumann algebra, there is also a notion of extended Haagerup tensor product. This is the only one not having an obvious quotient description even though there is a less obvious one. This is not obvious because, even though the extended Haagerup product is often a dual operator space, the module variant is not a quotient by a weak-* closed subspace. The two other products are as follows, for $X$ a strong normal right $D$ operator module and $Y$ a strong normal left $D$ operator module (see \cite{M97}, written $X\in SOM_D$) $X\otimes_{ehD}Y$ and, for normal dual operator modules (see \cite{M05} especially for the terminology above, written $X\in NDOM_D, Y\in {}_DNDOM$) a normal Haagerup tensor product $X\otimes_{\sigma\text{h}D}Y$ which coincides with the obvious quotient by the weak-* closure of the module generated by $d\otimes 1- 1\otimes d$. We first recall a few known general results.
The reader should remember the specific class where each product is well behaved and defined, respectively operator modules, strong operator modules and normal dual operator modules which is more and more restrictive.

\begin{theorem}\label{HaagerupModule}
\begin{enumerate}
\item (Associativity)
\cite[Th 3.4.10]{BLM} For $X\in OM_D, Y\in {}_DOM_D,Z\in {}_DOM$, we have $X\otimes_{h D}Y\in OM_D,Y\otimes_{h D}Z\in{}_DOM$ and :
$$(X\otimes_{h D}Y)\otimes_{h D}Z=X\otimes_{h D}(Y\otimes_{h D}Z).$$
\cite[Prop 4.1, Th 4.3]{M97} For $X\in SOM_D, Y\in {}_DSOM_D,Z\in {}_DSOM$, we have $X\otimes_{eh D}Y\in SOM_D,Y\otimes_{eh D}Z\in{}_DSOM$ and :
$$(X\otimes_{eh D}Y)\otimes_{eh D}Z=X\otimes_{eh D}(Y\otimes_{eh D}Z).$$
\item (Fonctoriality,Injectivity) \cite[lemma 3.4.5]{BLM} For any completely bounded D-module maps between operator spaces $$u_1:A_1\to B_1,u_2:A_2\to B_2,A_1,B_1\in OM_D,A_2,B_2\in {}_DOM$$ the maps $$u_1\otimes_D u_2:A_1\otimes_{h D}A_2\to B_1\otimes_{hD}B_2$$ is completely bounded, and \cite[section 7]{AP} it is completely isometric if $u_1,u_2$ are. 
\cite[Proof of Prop 3.3]{M97} (see  also \cite[Th 2.3]{B97b}) If $D$ is a von Neumann algebra, for any completely bounded D-module maps between strong operator spaces $u_1:A_1\to B_1,u_2:A_2\to B_2,A_1,B_1\in SOM_D,A_2,B_2\in {}_DSOM$ the maps $u_1\otimes_D u_2:A_1\otimes_{eh D}A_2\to B_1\otimes_{eh D}B_2$ is completely bounded, and \cite[Proof of Prop 3.12]{M05} it is injective if $u_1,u_2$ are. 
\item (Module duals) \cite[lemma 2.4]{M97,M95} If $X\in SOM_D, Y\in {}_DSOM$ there are completely isometric inclusions  $X\otimes_{h D}Y\hookrightarrow  X\otimes_{eh D} Y$ 
\cite[Th 3.2,Th 4.2,]{M05}  if $M,D,N$ are finite von Neumann algebras in standard form on their $L^2$ space, $X\in  {}_MSOM_D, Y\in {}_DSOM_N$, if we call $X^\natural:=CB_M(X,B(L^2(D),L^2(M)))_D,$

\noindent $Y^\natural:=CB_D(X,B(L^2(N),L^2(D)))_N$ the (proper) bimodule duals, and $(X\otimes_{h D}Y)^{\natural D norm}$ the subspace of $(X\otimes_{h D}Y)^{\natural}:=CB_M(X\otimes_{h D}Y,B(L^2(N),L^2(M)))_N$ of maps normal in $D$, then if $D'\simeq D^op$ is the commutant of $D$ in its action on $L^2(D)$
$$(X\otimes_{h D}Y)^{\natural D norm}\simeq X^\natural\otimes_{eh D'}Y^\natural\supset X^\natural\otimes_{h D'}Y^\natural,$$
 $$
X^\natural\otimes_{\sigma\text{h} D'}Y^\natural\simeq (X\otimes_{eh D}Y)^{\natural}\simeq X^\natural\otimes_{\sigma\text{h}}Y^\natural/Q(D') .$$

with $Q(D)$ the weak-* closed subspace of $X^\natural\otimes_{\sigma\text{h}}Y^\natural$ generated by all elements of the form $xb\otimes y - x\otimes by, x\in X^\natural, y\in Y^\natural, b\in D'.$

\item (Multiplication, adjoint) \cite[Th 4.4]{M05} For $M\subset B(H),N\subset B(K)$ von Neumann subalgebras containing $D$ as a common sub-von Neumann algebra the lateral multiplication map $(x\otimes y)\mapsto (\psi_0(x\otimes y):T\mapsto xTy)$ extends uniquely to a weak-*  homeomorphic completely isometric $(M,N)$-bimodule isomorphism of $M\otimes_{\sigma \text{h} D} N$ with $(M',N')$ bimodular completely bounded maps from $B_D(K,H)=D'\cap B(K,H)$ to  $B(K,H)$ : 
$$CB_{M',N'}(B_D(K,H),B(K,H))=M\otimes_{\sigma \text{h} D} N.$$
Especially by definition $M\otimes_{\text{eh} D} N\hookrightarrow M\otimes_{\sigma \text{h} D} N$ completely isometrically.
As a consequence, $(x\otimes y)\#\sum (x'\otimes y')=\sum(xx'\otimes y'y)$ extends to a Banach space module map on  $M\otimes_{\sigma h D} N\times (D'\cap (M\otimes_{\sigma h D} N))$ or a Banach algebra multiplication on $D'\cap (M\otimes_{\sigma h D} N)\simeq CB_{M',N'}(B_D(K,H),B_D(K,H))$  corresponding to composition in the isomorphim above (thus weak-* continuous in each variable) and these maps restricts to the corresponding extended Haagerup tensor products and Haagerup tensor product.

Likewise, as a consequence, if $M=N$ the map $U\in  (M\otimes_{\sigma h D} M\mapsto U^{\star}$ defined by $U^{\star}(B)=(U\#B^*)^*$ for $B\in B_D(H)$ is an isometric antilinear involution that restricts to $(M\otimes_{eh D} M,(D'\cap (M\otimes_{eh D} M),D'\cap(M\otimes_{\sigma h D} M $.



 \end{enumerate}
\end{theorem}
For $X\in  {}_MOM_N$, we used the notation of \cite{M05} for the (proper) module dual
$X^{\natural}:=CB_M(X,B(L^2(N),L^2(M)))_N$.

Because of the normality issue in Theorem \ref{HaagerupModule}(3) above, it seems natural to define for $X,Y$ normal dual operator modules (over $D'$) with a canonical inclusion

$$I:X\otimes_{eh D'}Y\simeq(X_\natural\otimes_{h D}Y_\natural)^{\natural D norm}\subset (X_\natural\otimes_{h D}Y_\natural)^{\natural}=: X\otimes_{w^*h D'}Y.$$

In this way there is automatically a completely contractive projection $$P:X\otimes_{\sigma\text{h} D'}Y \to X\otimes_{w^*h D'}Y$$ dual to the inclusion between Haagerup and extended Haagerup tensor products.
However, contrary to the case $D=\C$ there is only a complete isometric embedding $$J:X^\natural\otimes_{eh D'}Y^\natural \hookrightarrow X^\natural\otimes_{\sigma\text{h} D'}Y^\natural$$ from \cite[Prop 3.10]{M97} (but in general no $X^\natural\otimes_{w^*h D'}Y^\natural \hookrightarrow X^\natural\otimes_{\sigma\text{h} D'}Y^\natural$) and of course one sees from the definitions that $P\circ J= I.$ 
However, it is convenient to use the weak-* topology on $X^\natural\otimes_{w^*h D'}Y^\natural$ and consider the topology induced on $X^\natural\otimes_{eh D'}Y^\natural$
as will be motivated in the next proposition in the case we will be interested in. We can also consider on $X^\natural\otimes_{eh D'}Y^\natural$ the topology generated by the weak-* topology and the supplementary seminorms indexed by $x\in X, y\in Y$ and given for $\Omega\in (X\otimes_{h D}Y)^{\natural D norm}$ by :
$$p_{x,y}(\Omega)=\sup_{d\in D_1}|\Omega(x\otimes_Ddy)|.$$
We will call this the \textit{normal weak-* topology} on $X^\natural\otimes_{eh D'}Y^\natural$ and this of course insures preservation of normality in the limit.

Since we will use often several tensor product variants of \cite[Theorem 3.2]{M05} and this will be crucial for us, we write it explicitly and will use obvious variants of the normal weak-* topology. We may still refer to it as \cite[Theorem 3.2]{M05} except for the supplementary results. The reader should note that the dual setting is crucial here  for the complete quotient map since the extended Haagerup  product is known not to be projective by \cite{EffrosRuanDual} even in the operator space case.

\begin{lemma}\label{IteratedModule}
Let $A_1,...A_{n+1}$ von Neumann algebras and $X_i,Y_i\in {}_{A_i}NOM_{A_{i+1}}$, then $(X_1\otimes_{h A_2}...\otimes_{hA_n}X_n)^{\natural_{A_2,...,A_n normal}}= X_1^\natural\otimes_{eh A_2'}...\otimes_{eh A_n'}X_n^\natural$ completely isometrically as $A_1-A_{n+1}$ bimodules. If $I_i:X_i\subset Y_i$ is a complete isometry between strong modules, then the module dual of their tensor product induces a complete quotient map : $$(I_1\otimes...\otimes I_{n})^\natural :Y_1^\natural\otimes_{eh A_2'}...\otimes_{eh A_n'}Y_n^\natural\to X_1^\natural\otimes_{eh A_2'}...\otimes_{eh A_n'}X_n^\natural.$$

 If moreover $X_i\in {}_{A_i}NDOM_{A_{i+1}},$ then $$(X_1\otimes_{\sigma h A_2}...\otimes_{\sigma hA_n}X_n):=((X_1)_\natural\otimes_{eh A_2'}...\otimes_{ehA_n'}(X_n)_\natural)^\natural= (...(X_1\otimes_{\sigma h A_2}X_2)...\otimes_{\sigma hA_n}X_n).$$
 
 Finally, for $X_i\in {}_{A_i}SOM_{A_{i+1}},$ if $$I:X_1^\natural\otimes_{eh A_2'}...\otimes_{eh A_n'}X_n^\natural\subset (X_1\otimes_{h A_2}...\otimes_{hA_n}X_n)^{\natural}=:X_1^\natural\otimes_{w^* h A_2'}...\otimes_{w^* h A_n'}X_n^\natural,$$ is the canonical bimodular complete isometry and $$P:(X_1^\natural\otimes_{\sigma h A_2'}...\otimes_{\sigma hA_n'}X_n^\natural)\to X_1^\natural\otimes_{w^* h A_2'}...\otimes_{w^* h A_n'}X_n^\natural$$ the canonical bimodular weak-* continuous complete quotient map, then there is a bimodular complete isometry $$J:X_1^\natural\otimes_{eh A_2'}...\otimes_{eh A_n'}X_n^\natural\to (X_1^\natural\otimes_{\sigma h A_2'}...\otimes_{\sigma hA_n'}X_n^\natural) $$ with commutative diagram $P\circ J=I,$ and uniquely determined by the relation : $$J(v)(\sum_{i_1\in I_1,...i_n\in I_n}x^{(1)}_{i_1}\otimes  x^{(2)}_{i_1,i_2}\otimes ...\otimes x^{(2)}_{i_n})=\sum_{i_1\in I_1,...i_n\in I_n}v(x^{(1)}_{i_1}\otimes  x^{(2)}_{i_1,i_2}\otimes ...\otimes x^{(n)}_{i_n}).$$
\end{lemma}
\begin{proof}
For the first part, the case $n=2$ is \cite[Theorem 3.2]{M05} and by induction, we have to check that  $(X_1\otimes_{h A_2}...\otimes_{hA_{n-1}}X_{n-1})^{\natural_{A_2,...,A_{n-1} normal}}\otimes_{eh A_n'}X_n^\natural=(X_1\otimes_{h A_2}...\otimes_{hA_n}X_n)^{\natural_{A_2,...,A_n normal}}.$
Recall that if $H_n$ is a fixed proper $A_n$ (Hilbert) module, $(X_1\otimes_{h A_2}...\otimes_{hA_n}X_n)^{\natural}=CB_{A_1}(X_1\otimes_{h A_2}...\otimes_{hA_n}X_n, B(H_{n+1},H_1))_{A_{n+1}}$. What we know by associativity and the case $n=2$ is that $(X_1\otimes_{h A_2}...\otimes_{hA_n}X_n)^{\natural_{A_n normal}}=(X_1\otimes_{h A_2}...\otimes_{hA_{n-1}}X_{n-1})^{\natural}\otimes_{eh A_n'}X_n^\natural.$ By injectivity of the module extended Haagerup tensor product, we know that the spaces we want to identify are subspaces. The complete isometry and modularity will we obvious once identified the subspaces. Note that \begin{align*}(X_1\otimes_{h A_2}...\otimes_{hA_{n-1}}X_{n-1})^{\natural_{A_2,...,A_{n-1} normal}}\otimes_{eh A_n'}X_n^\natural&=X_1^\natural\otimes_{eh A_2'}...\otimes_{eh A_n'}X_n^\natural\\&\subset(X_1\otimes_{h A_2}...\otimes_{hA_n}X_n)^{\natural_{A_2,...,A_n normal}},\end{align*}
is obvious with the same  explicit pairing that in (the easy part of) \cite[Theorem 3.2]{M05} we thus check the converse. Thus take $u\in M_{1,I}(X_1\otimes_{h A_2}...\otimes_{hA_n}X_n)^{\natural}=CB_{A_1}(X_1\otimes_{h A_2}...\otimes_{hA_{n-1}}X_{n-1}, B(\ell^2(I)\otimes H_{n},H_1))_{A_{n}}$, $v \in M_{I,1}(X_n^\natural)=CB_{A_n}(X_n, B(H_{n+1},\ell^2(I)\otimes H_{n})_{A_{n+1}}$ so that $u\odot v$ is a typical element of $(X_1\otimes_{h A_2}...\otimes_{hA_{n-1}}X_{n-1})^{\natural}\otimes_{eh A_n'}X_n^\natural$ and assume it lies in $(X_1\otimes_{h A_2}...\otimes_{hA_n}X_n)^{\natural_{A_2,...,A_n normal}}.$ Then define $H=Span\{v(x)[h], x\in X_n, h\in H_{n+1}\}\subset \ell^2(I)\otimes H_{n}.$ It is invariant by action of $A_n$ since for $a\in A_n$ $a[v(x)[h]]=v(ax)[h]$ and thus the projection on its closure $p\in A_n'\cap B(\ell^2(I)\otimes H_{n})= M_I(A_n').$ (by default $A_n'=A_n'\cap B(H_n)$).
Thus by results of Magajna $u\odot v=u\odot pv=up\odot pv.$ Now it suffices to show that $up\in M_{1,I}(X_1\otimes_{h A_2}...\otimes_{hA_n}X_n)^{\natural_{A_2,...,A_{n-1} normal}},$ i.e for any $e_i,$ $i\in I$ for the canonical basis of $\ell^2(I)$, $up(e_i\otimes .)\in X_1\otimes_{h A_2}...\otimes_{hA_n}X_n)^{\natural_{A_2,...,A_{n-1} normal}}.$ Take $x_m\in H$ tending to $p(e_i\otimes \xi)$ for $\xi\in H_n$, it is easy to see that $u(x_m)$ has the right normality by the assumption on $u\odot v$, thus by a norm limit  $up(e_i\otimes \xi)$has it too or said otherwise $up(e_i\otimes .)$ is normal for the weak topology at the target and again by normwise density (of elementary tensors in trace class) this is enough. 

For the complete quotient map, first note that one has a complete contraction $(I_1\otimes...\otimes I_{n})^\natural :(Y_1\otimes_{h A_2}...\otimes_{hA_{n}}Y_{n})^{\natural}\to (X_1\otimes_{h A_2}...\otimes_{hA_{n}}X_{n})^{\natural},$ that induces on the normal part the map we want. Let us check it is a quotient map. Start with the case $n=2$, thus take $\Phi\in M_n((X_1\otimes_{h A_2}X_{2})^{\natural_{A_2 normal}})$ of norm $<1$ and find, by the version of Christensen-sinclair representation Theorem in \cite[Th 3.9]{M97}, an Hilbert space $H$ with a normal representation of $A_2$, $\Phi_1:X_1\to B(H,H_1^n),\Phi_2:X_2\to B(H_3^n,H)$ bimodular completely (strictly) contractive such that $\Phi(x_1\otimes x_2)=\Phi_1(x_1)\Phi_2(x_2)$. Since $\Phi_i\in X_i^\natural$ for some (maybe non proper) dual, one can use, since $X_i,Y_i$ are strong \cite[Prop 3.12 (ii)]{M05} to get complete quotient maps $Y_i^\natural\to X_i^\natural$ thus giving $\Psi_1:Y_1\to B(H,H_1^n),\Psi_2:Y_2\to B(H_3^n,H)$ bimodular, completely contractive extensions of $\Phi_i$'s. We thus define $\Psi(y_1\otimes y_2)=\Psi_1(y_1)\Psi_2(y_2)$ as extension of $\Phi$ and from the normality of $H$ it is easy to see $\Psi\in  M_n((Y_1\otimes_{h A_2}Y_{2})^{\natural_{A_2 normal}})$ as expected of norm $||\Psi||\leq 1$. This concludes to the complete quotient map in the case $n=2$, the general case is left to the reader.

For the last part, we consider only the case $n=3.$ Then by definition in \cite[section 4]{M05} $(X_1\otimes_{\sigma h A_2}X_2)\otimes_{\sigma hA_3}X_3=((X_1\otimes_{\sigma h A_2}X_2)_\natural\otimes_{e hA_3'}(X_3)_\natural)^\natural$. And from \cite[Th 2.17]{M05} $(X_1\otimes_{\sigma h A_2}X_2)_\natural=(X_1)_\natural\otimes_{eh A_2'}(X_2)_\natural$ if and only if this last module is a strong operator module which is the case by \cite[Prop 4.1]{M97}. The result, as well as the general $n$ case by an easy induction, follows from associativity of module extended Haagerup product of strong modules  \cite[Theorem 4.3]{M97}.

$J$ is obtained by iteratively applying \cite[Prop 3.10]{M97} and the uniqueness condition expressing the extension on standard tensors there implies the relation $P\circ J=I.$
\end{proof}
\subsection{Concrete examples of module duals in the finite case.}

From now on $B,D,M,N$ will be \textbf{finite von Neumann algebras with a fixed faithful normal tracial state $\tau$}. We may still call them only finite von Neumann algebras and sometimes more precisely $W^*$ probability space. $L^2(N)=L^2(N,\tau)$ will then be the associated Hilbert space realizing the standard form Hilbert space for $N$. All inclusions $B\subset N$ will be with agreeing trace.

 To describe the proper module dual in the cases we will be interested in, we will
need other results of the paper \cite{M05}. We equip in general the Hilbert space $L^2(N)$ with its left module structure and its column Hilbert space operator space structure as in this paper if not otherwise specified. If not otherwise specified, we write $L^2(N)^*$ for its dual, the right module with row Hilbert space structure \cite{BLM,M05}. Then, \cite[Corol 3.5]{M05}, for $X$ an $M-N$ operator bimodule, $X^{\natural}=(L^2(M)\otimes_{h M} X\otimes_{h N} L^2(N))^*\in {}_{M'}NDOM_{N'}$ (with operator space dual structure, the commutant actions referring to commutants on $L^2(M)$, $L^2(N)$).

For $X\in  {}_MNDOM_N$, there is a predual $X_{\natural}:=NCB_M(X,B(L^2(N),L^2(M)))_N.$ Then, \cite[Th 2.17, 3.7]{M05}, $X\in  {}_MOM_N$ is strong if and only if $(X^{\natural})_{\natural}=X$, and for any $X\in  {}_MNDOM_N$, $X\simeq (X_{\natural})^{\natural}$ as a completely isometric weak-* homeomorphic isomorphism. 

From \cite[Prop 4.3]{M05}, if $A,B\subset M$ von Neumann subalgebras and we consider $M\in {}_ANDOM_B$, then \begin{equation}\label{bimodpredualM}({}_AM_B)_{\natural}\simeq B_A(L^2(M),L^2(A))\otimes_{eh M'}B_B(L^2(B),L^2(M)),\end{equation}
by identification with the side  multiplication map.

We will relate these with natural objects studied by von Neumann algebraist, for $D\subset M,$ finite von Neumann algebras: the space of right bounded vectors is  $${}_ML^2(M)_{L^2(D)}:=\{\xi\in L^2(M)\ : \ \exists C>0\forall d\in D, ||\xi d||_2\leq C||d||_2\},$$
normed with $||\xi||=\sup_{||d||_2\leq 1,d\in D}||\xi d||_2=||E_D(\xi^*\xi)||^{1/2}.$ This space is thus an $D$-module in the sense of $C^*$-modules.
We recall the classical norm on $M_n({}_ML^2(M)_{L^2(D)})$ by $||(y_{ij})||^2=||[\sum_k E_D(y_{ki}^*y_{kj})]_{ij}||_{M_n(D)}$ as the canonical operator space structure of the (right)
$C^*$-modules studied in \cite{B97}. {It is shown there that ${}_ML^2(M)_{L^2(D)}\otimes_{hD}L^2(D)_c$ is completely isometric to a  column Hilbert space, namely $L^2(M)_c$.}

Likewise, we define 
\begin{align*}{}_ML^1(M)_{L^2(D)}:=\{\xi\in L^1(M)\ : &\ \exists C>0\forall n\forall d\in M_n(D)\subset[M_n(L^2(D))]^*,\\& ||\xi d||_{\mathscr{TC}(\R^n)\hat{\otimes} L^1(M)}\leq C||d||_2\},\end{align*}
normed with $||\xi||=\sup_{||d||_2\leq 1,d\in M_n(D)\subset[M_n(L^2(D))]^*}||\xi d||_{\mathscr{TC}(\R^n)\hat{\otimes} L^1(M)},$ similarly :
\begin{align*}{}_{L^2(D)}L^1(M)_M:=\{\xi\in L^1(M)\ &: \ \exists C>0\forall n\forall d\in M_n(D)\subset [M_n(L^2(D)^*)]^*,\\& ||d\xi ||_{\mathscr{TC}(\R^n)\hat{\otimes} L^1(M)}\leq C||d||_2\},\end{align*}
(recall that here the meaning of $[M_n(L^2(D)^*)]^*$ is the Banach space dual of $M_n(L^2(D))$ with row norm.)
 and 
\begin{align*}{}_{L^2(D)}L^1(M)_{L^2(D)}:=\{\xi\in L^1(M)\ &: \ \exists C>0\forall n\forall d_1,d_2\in D^n\subset (L^2(D))^n, \\&||(d_{1,i}\xi d_{2,j})_{i,j}||_{\mathscr{TC}(\R^n)\hat{\otimes} L^1(M)}\leq C||d_1||_2||d_2||_2\},\end{align*}
normed with $||\xi||=\sup_{||d_i||_2,\leq 1,d_i\in D^n}||(d_{1,i}\xi d_{2,j})_{i,j}||_{\mathscr{TC}(\R^n)\hat{\otimes} L^1(M)},$
 


 \begin{proposition}\label{ModulePreDual}Let $D\subset M$ finite von Neumann algebras.
   We have completely isometric module isomorphisms :
 $${}_{M'}({}_{L^2(D)}L^2(M)_M^{ op})_{D'}\simeq ({}_MM_D)_{\natural},\ \ \  {}_{D'}({}_ML^2(M)_{L^2(D)}^{ op})_{M'}\simeq ({}_DM_M)_{\natural}.$$
 the first  with the canonical operator space structure when seen as ${}_{M'}({}_{L^2(D)}L^2(M)_M^{ op})_{D'}= {}_{M'}L^2(M')_{L^2(D')}$, namely as noted, the opposite structure as defined before. 
  Moreover, we have isometric isomorphisms :
 $${}_ML^1(M)_{L^2(D)}\simeq ({}_ML^2(M)_{L^2(D)}^{ op})\otimes_{h M'}L^2(M)=({}_DM_{\C})_{\natural} ,$$ $$ {}_{L^2(D)}L^1(M)_M\simeq (L^2(M))^*\otimes_{h M'}({}_{L^2(D)}L^2(M)_M^{ op})= ({}_{\C}M_D)_{\natural},$$
and we equip the left hand sides with the operator space structure of the right hand sides. Likewise, we have :
 $${}_{L^2(D)}L^1(M)_{L^2(D)}\simeq ({}_ML^2(M)_{L^2(D)}^{ op})\otimes_{eh M'}{}_{L^2(D)}L^2(M)_M^{ op}=({}_DM_{D})_{\natural} .$$
  Moreover $({}_ML^2(M)_{L^2(D)}^{ op})$ is dense in $({}_DM_{D})_{\natural}=NCB(M,B(L^2(D)))$ for the topology of pointwise strong operator topology convergence and ${}_{L^2(D)}L^2(M)_M^{ op}$ is dense for the topology of pointwise *-strong operator topology convergence. 
 \end{proposition}

\begin{proof}
 From formula \eqref{bimodpredualM}, we have  $$({}_MM_D)_{\natural}=M'\otimes_{eh M'}B_D(L^2(D),L^2(M))=B_D(L^2(D),L^2(M)),$$ $$({}_DM_M)_{\natural}=B_D(L^2(M),L^2(D)).$$

Let us show that the maps $\psi_1: {}_{L^2(D)}L^2(M)_M\to B_D(L^2(D),L^2(M)), \psi_2: {}_ML^2(M)_{L^2(D)}\to B_D(L^2(M),L^2(D))$ defined by $$\psi_1(\xi)(\eta)=\eta\xi,\ \ \ \ \psi_2(\xi)(\eta)=E_D(\eta\xi),$$
are isometric module isomorphisms, the last one is well defined since for $\xi\in {}_ML^2(M)_{L^2(D)}, \eta \in L^2(M), \eta\xi\in {}_ML^1(M)_{L^2(D)}$ and $E_D:{}_ML^1(M)_{L^2(D)}\to L^2(D)$ bounded. To check the module structure, take $x\in M',y\in D'$
$$\psi_1(x.\xi.y)(\eta)=\eta x\xi y=\eta y^o \xi x^o=\psi_1(\xi)(\eta y^o)x^o,$$ $$\psi_2(y.\xi. x)(\eta)=E_D(\eta x^o \xi y^o)=\psi_2(\xi)(\eta x^o)y^o,$$
The injectivity of $\psi_i$ is obvious, and the isometry also by duality, surjectivity can be proved in noting that   $B_D(L^2(M),L^2(D)),B_D(L^2(D),L^2(M))\subset B_D(L^2(M),L^2(M))=\langle M',e_B\rangle$ (the later inclusion using composition with $e_D$) and using the usual weakly-* dense subset of the basic construction.

To see that $\psi_1$ gives a complete isometry, let us make explicit the norm on the spaces involved : $M_n(B_D(L^2(D),L^2(M)))=B_D((L^2(D))^n,(L^2(M))^n)$ and for $[\xi_{ij}]\in {}_{M'}L^2(M')_{L^2(D')}$ we deduce the formula from the case $\xi_{ij}=(x_{ij})^o\in M^o$
Then $$||[\xi_{ij}]||^2=||E_{D^o}(\sum_{k}((x_{ki})^o)^*(x_{kj})^o)||_{M_n(D^o)}=||([E_{D}(\sum_{k}x_{kj}x_{ki}^*)]^o)||_{M_n(D^o)}.$$
But the norm of $(\psi_1(x_{ij}))$ is exactly given by \begin{align*}||(\psi_1(x_{ij}))||&=\sup_{\eta\in (L^2(D))^n}\left[\sum_{ijk}\tau((\eta_jx_{ij})^*\eta_kx_{ik})\right]^{1/2}\\&=\sup_{\eta\in (L^2(D))^n}\left[\sum_{ijk}\tau(\eta_kE_D(x_{ik}x_{ij}^*)\eta_j^*)\right]^{1/2}\\&=||E_D(\sum_i x_{ik}x_{ij}^*)||^{1/2},\end{align*}
proving the desired complete isometry.

Similarly, $\langle\psi_2((x_{ij}))((y_k)),(z_k)\rangle=\sum_{jk} \tau(z_k^*( y_jx_{kj})=\tau((z_kx_{kj}^*)^*y_j)$ so that :$\psi_2(X)^*=\psi_1(X^*)$ and thus for $X=(x_{ij})$ we have $$||(\psi_2(X))||=||(\psi_1(X^*)||=||E_D(\sum_i x_{ki}^*x_{ji})||^{1/2}=||X^T||_{M_n({}_ML^2(M)_{L^2(D)})}=||X||_{M_n({}_ML^2(M)_{L^2(D)}^{ op})}.$$

For $L^1$ spaces, as before formula \eqref{bimodpredualM} gives $$({}_DM_{\C})_{\natural}\simeq B_D(L^2(M),L^2(D))\otimes_{eh M'}L^2(M), ({}_{\C}M_D)_{\natural}\simeq (L^2(M))^*\otimes_{eh M'}B_D(L^2(D),L^2(M)).$$

This gives the second isomorphisms since for Hilbert spaces, the extended-Haagerup and haagerup products coincide \cite[Rmk 2.18]{M05}.

 The first isomorphism is given by coming back to the definition of the predual 
 $({}_DM_{\C})_{\natural}=NCB_D(M,L^2(D)), ({}_{\C}M_D)_{\natural}=NCB(M,(L^2(D))^*)_D.$. Let us show that the isomorphisms is thus given by $\psi_3: {}_{L^2(D)}L^1(M)_M\to NCB(M,L^2(D)^*)_D, \psi_2: {}_ML^1(M)_{L^2(D)}\to NCB_D(M,L^2(D)),$ with the same formula for $\psi_2$ and for $m\in M, d\in L^2(D)$:
 $$\psi_3(\xi)(m)(d)=\tau(\xi md).$$
 Bimodularity and isometry are obvious. The inverse being given by evaluation on $L^2(D)$ at $1$ and via identification of $L^1(M)$ with the predual, this concludes the proof of the second isomorphisms. Finally the computation of $({}_DM_{D})_{\natural}$ is similar, one of them is given by 
 $\psi_4: {}_{L^2(D)}L^1(M)_{L^2(D)}\to NCB_D(M,B(L^2(D)))_D,$
$$\psi_4(\xi)(m)(d)= E_D(md\xi).$$

For the last density result, write $\xi=x\otimes_{M'}y\in L^2(M)_{L^2(D)}^{ op}\otimes_{eh M'}{}_{L^2(D)}L^2(M)_M^{ op}$ in its canonical writing and look at $y_{\alpha}=\alpha(\alpha+yy^*)^{-1}y=y\alpha(\alpha+y^*y)^{-1}$ (written with opposite products and  $\xi_{\alpha}=x\otimes_{M'}y_{\alpha}$ get the corresponding multiplication in $xy_{\alpha}\in({}_ML^2(M)_{L^2(D)})$ and, with all products in $M^o$: $$\tau(d'E_{D^{op}}(x(y-y_{\alpha})dm))|^2\leq \tau(d'E_{D^{op}}(xx^*)d'^{*})||(y-y_{\alpha})dm||_2\leq ||d'||_2^2||x||^2||(y-y_{\alpha})dm||_2\to_{\alpha\to \infty} 0$$ even uniformly in $d'$, this concludes. The second case is similar.
\end{proof}
We will explain more on module extended Haagerup products of these spaces in subsection \ref{HaaModuleSup}.

\subsection{Operator space direct sums and operator modules}

To define analytic functions, we will need to take $\ell^1$ direct sums, but unfortunately, there is no reason for an operator space $\ell^1$ direct sum (see e.g. \cite[section 2.6]{PisierBook}) to be an operator module (this is however true for a matrix normed module, i.e. a module for the projective opertor tensor product by commutation of projective tensor product and $\ell^1$ direct sums, cf \cite[Rmk 3.1.9]{BLM}). However, the $\ell^\infty$ (also written merely $\oplus$, and $c_0$) direct sums obviously remain operator modules, we will have by duality and universal property, two (agreeing) candidates to be 
operator modules $\ell^1$ direct sums. We summarize this in the next~:

\begin{lemma}\label{Ell1Sum}
Let $(E_i)_{i\in I}$ a family of D-D normal operator modules. Consider $\oplus^{fin}_{i\in I}E_i\subset (\oplus_{i\in I}E_i^{\natural})^{\natural}$, then the completion $\ell_{D,D}^1(E_i;i\in I)\in {}_DNOM_D,$ with complete isometric injection $\epsilon_i:E_i\to \ell_{D,D}^1(E_i;i\in I)$ and complete quotient projection $\pi_i:\ell_{D,D}^1(E_i;i\in I)\to E_i$  and the duality relations :
$$c_0(E_i;i\in I)^\natural =\ell_{D,D}^1(E_i^\natural;i\in I), \ell_{D,D}^1(E_i;i\in I)^\natural=\oplus_{i\in I}E_i^\natural,$$
Moreover, if $E_i\in {}_DSOM_D,$ (resp. ${}_DNDOM_D$) then we have $\ell_{D,D}^1(E_i;i\in I)\in {}_DSOM_D.$ (resp. ${}_DNDOM_D$), and $(\oplus_{i\in I}E_i^\natural)_\natural=\ell_{D,D}^1(E_i;i\in I)$ (resp. in the normal operator dual case $c_0((E_i)_\natural;i\in I) =(\ell_{D,D}^1(E_i;i\in I))_\natural$).

Finally, $\ell_{D,D}^1(E_i;i\in I)$ has the following universal property: if $Z\in {}_DOM_D$ and $u_i:E_i\to Z$ are $D-D$ module completely contractive maps, there is a unique completely contractive $D-D$ module map $u:\ell_{D,D}^1(E_i;i\in I)\to Z$, $u\epsilon_i=u_i$ and $$CB_{D,D}(\ell_{D,D}^1(E_i;i\in I),Z)\simeq \oplus_{i\in I}CB_{D,D}(E_i,Z).$$
Thus if $E_i\subset F_i$ completely contractively as submodule, then 
$\ell_{D,D}^1(E_i;i\in I)\subset \ell_{D,D}^1(F_i;i\in I)$, and if $C\subset D$, there is a canonical contractive injection $L_{C\to D}:\ell_{C,C}^1(E_i;i\in I)\to \ell_{D,D}^1(E_i;i\in I),$ with dense image. For general $E_i,F_j$ normal operator modules, we also have the complete isometry: $$\ell_{D,D}^1(E_i;i\in I)\widehat{\otimes}_D\ell_{D,D}^1(F_j;j\in J)\simeq \ell_{D,D}^1(E_i\widehat{\otimes}_D F_j;(i,j)\in I\times J) .$$

\end{lemma}
Note that we will write $E_1\oplus^1_DE_2$ for $\ell_{D,D}^1(E_i;i\in I)$ when $|I|=2,$ and $\ell_{D,D}^1=\ell_{D}^1.$

\begin{proof}
The first statement is obvious since module duality gives module as duals, $\epsilon_i,\pi_i$ comes from dualization of those for $\ell^\infty$ sums. Note that in order to have $\epsilon_i$ complete isometric, we use the normality of $E_i$ to identify $E_i\subset (E_i)^{\natural\natural}$ completely isometrically (cf. the characterization of the normal parts $(E_i)_n$ of an operator module \cite[Prop 5.7]{M05} which is $E_i$ itself when it is normal). Since $c_0(E_i;i\in I)\hookrightarrow \oplus_{i\in I}E_i^{\natural\natural}$ completely isometrically, one gets a complete isometric $D-D$ module map  $\ell_{D,D}^1(E_i^\natural;i\in I)\subset(\oplus_{i\in I}E_i^{\natural\natural})^\natural\to c_0(E_i;i\in I)^\natural.$ Since $\oplus^{fin}_{i\in I}E_i^\natural$ is obviously dense in this last space, this concludes.  
The second duality result is a specific case of the universal property, but we will use it in its proof. 

Likewise we have the canonical complete isometric bimodule map $\oplus_{i\in I}E_i^\natural\to (\oplus_{i\in I}E_i^\natural)^{\natural\natural}\to  \ell_{D,D}^1(E_i;i\in I)^\natural$ dualizing the defining inclusion, and the surjectivity is obvious since composing $\psi\in \ell_{D,D}^1(E_i;i\in I)^\natural$ with $\epsilon_i$ gives a family $(\psi\epsilon_i)\in \oplus_{i\in I}E_i^\natural$ that has to agree with $\psi$ by density of finite sums in $\ell_{D,D}^1(E_i;i\in I).$

Let us now check the stability of normal dual operator modules.
By \cite[Th 3.7]{M05} and our previous duality result :$c_0((E_i)_\natural;i\in I)^\natural=\ell_{D,D}^1(E_i;i\in I)$, thus the $\ell^1$ sum is a dual operator module, and it remains to check normality. Since a $(E_i)_\natural$ is a strong module and a $c_0$ or $\ell^\infty$ sum obviously remain strong, the predual computation then follows from \cite[Th 2.17]{M05}.
By \cite[Rmk 2.10]{M05}, we are done since $\ell_{D,D}^1(E_i;i\in I)$ is always a normal module as complete isometric subspace of a normal dual module.



For the stability of the property of being a strong module, it suffices to check the predual computation by characterization of a strong module \cite[Th 2.17]{M05}.
Thus take $\psi\in (\oplus_{i\in I}E_i^\natural)_\natural$ a normal completely bounded map. Obviously the inclusion $\epsilon_i:E_i^\natural\to (\oplus_{i\in I}E_i^\natural)$ is normal, thus $\psi\epsilon_i\in (E_i^\natural)_\natural=E_i$ by characterization of a strong module. Thus by definition $\psi'=\oplus_i \psi\epsilon_i\in \ell_{D,D}^1(E_i,i\in I),$ it remains to check this is indeed $\psi$. But for $x\in \oplus_{i\in I}E_i^\natural,\oplus_{i\in F}\pi_i(x)$ weak-* converge to $x$, and $\psi(\oplus_{i\in F}\pi_i(x))=\psi(\oplus_{i\in F}\epsilon_i\pi_i(x))=\psi'(\oplus_{i\in F}\pi_i(x))$ we indeed deduce by weak-* continuity $\psi=\psi'$.

The only non obvious part of the universal property, is to build $u$ from $u_i$. This comes from the fact $(\oplus u_i^\natural)^\natural$ agree with the finite sum of $u_i$ on finite sums. The consequences of the universal property are obvious. To check $L_{C\to D}$ is injective, one uses that $\pi_i^D\circ L_{C\to D}=\pi_i^C$ (where the superscript emphasizes the type of modules the projection is referring to). This identity is obvious on finite sums by definition and then goes to the completion. Thus if $L_{C\to D}(x)=0$ one has $\pi_i^C(x)=0$ for all $i$, thus it is obvious that $x$ is 0 in duality with $\oplus_{i\in I}^{fin}E_i^{\natural_C}$ thus evaluated against its norm closure $c_0(E_i^{\natural_C},i\in I)$ thus also against $(\oplus_{i\in I}E_i^{\natural_C})=[\ell_{C,C}^1(E_i;i\in I)]^{\natural_C}=(c_0(E_i^{\natural_C},i\in I))^{\natural\natural}=(c_0(E_i^{\natural_C},i\in I))^{{}_{C'}*_{C'}*}.$ Indeed, by Golstine lemma, $c_0(E_i^{\natural_C},i\in I)$ is weak-* dense in its bidual, thus also from the proof of \cite[Corol 3.6]{M05} in $(c_0(E_i^{\natural_C},i\in I))^{{}_{C'}*_{C'}*}\subset (c_0(E_i^{\natural_C},i\in I))^{**}$ (recall we write $*$ for operator space duality written $\#$ there) and from the isomorphism in \cite[Corol 3.5]{M05} with the dual bimodule, this weak-* topology  in $(c_0(E_i^{\natural_C},i\in I))^{{}_{C'}*_{C'}*} $ is the pointwise convergence on $\ell_{C,C}^1(E_i;i\in I)$ in the weak-* topology of operators. Thus this density is indeed enough to get vanishing on $x\in  \ell_{C,C}^1(E_i;i\in I)$.
 Thus $x$ is $0$ in 
$\ell_{C,C}^1(E_i;i\in I)\subset(\oplus_{i\in I}E_i^{\natural_C})^{\natural_{C'}}.$

For commutation with projective operator space tensor product, one uses \cite[Prop 3.5.9]{BLM}.
\end{proof}
\subsection{Shuffle maps}
A shuffle map between tensor products is an extension to completion of a map on algebraic tensor product $S:B\otimes C_1\otimes A\otimes C_2\to C_1\otimes A\otimes B\otimes C_2$ with $S(b\otimes c\otimes a\otimes d)=(c\otimes a \otimes b\otimes d)$ or of the inverse of $S$. Two results are available in the literature, commutation of nuclear and extended Haagerup product and its dual between normal Haagerup and von Neumann tensor product \cite[Th 6.1]{EffrosRuanDual}. There is also the similar relation between minimal and Haagerup tensor product from \cite[Th 5.15 ]{PisierBook}. We will need variants of these results for other tensor products. Recall that the normal Fubini tensor product $(B^*{{\otimes_F}}A^*)=CB(B,A^*)$ from \cite[Th 7.2.3]{EffrosRuan}. The von Neumann (or normal spatial) tensor product $B^*{\overline{\otimes}}A^*$ is the weak-* closed subspace generated by the algebraic tensor product. 
We start by a module variant. We gave background on modules in  subsection \ref{HaagD}, but we refer mostly for normal dual operator bimodules ${}_DNDOM_D$ (or strong operator modules) to \cite{M05} and \cite[section 3.8]{BLM} but we use the terminology of Magajna.

\begin{proposition}\label{ShuffleNormalModule} Let $D$ be a von Neumann algebra. 
For any operator space $B$ and any normal dual operator $D$-bimodule $A,C_1,C_2,$ then $$(B^*{{\otimes_F}}A)=CB(B,A)\in {}_DNDOM_D, (B^*{\overline{\otimes}}A)\in {}_DNDOM_D$$ for the multiplications induced pointwise by $A$ and the shuffle map extends to  weak-* continuous completely contractive maps  (the second extension being unique) :
$$S_{F \sigma h}=S_{F \sigma h}^{B,C_1,A,C_2 }:C_1\otimes_{\sigma h D}(B{{\otimes_F}}A)\otimes_{\sigma h D}C_2\to B{{\otimes_F}}[C_1\otimes_{\sigma h D}A\otimes_{\sigma h D}C_2],$$
$$C_1\otimes_{\sigma h D}(B{\overline{\otimes}}A)\otimes_{\sigma h D}C_2\to B{\overline{\otimes}}[C_1\otimes_{\sigma h D}A\otimes_{\sigma h D}C_2].$$
The first shuffle map makes commutative diagrams with the canonical maps coming from functoriality of tensor products from weak-* continuous completely bounded bimodule maps between normal bimodules $j:C_1\to C_3, k:C_2\to C_4, l:A\to A_2.$ More precisely, we have $[Id\otimes (j\otimes l\otimes k)]\circ S_{F \sigma h}^{B,C_1,A,C_2 }=S_{F \sigma h}^{B,C_3,A_2,C_4 }\circ [ j\otimes (Id\otimes l)\otimes k].$
\end{proposition}

\begin{proof}
Let us start by building the first shuffle map in the case without module structure (namely $D=\C$) we build a canonical contractive map  $I\in CB(D\otimes_{\sigma h }CB(B,A)\otimes_{\sigma h }D , CB(B,D\otimes_{\sigma h }A\otimes_{\sigma h }D))$. But if $b\in B,$ $b\otimes .:CB(A_*, B\hat{\otimes} A_*)$ thus defines a dual weak-* continuous map $(b\otimes .)^*:CB(B,A)\to A$ which is merely the evaluation map. It is easy to see that $b\mapsto (b\otimes .)^*\in CB(B,NCB(CB(B,A),A))$. From \cite[(5.22)]{EffrosRuanDual} one gets a map $$Id\otimes (b\otimes .)^*\otimes Id\in NCB(D\otimes_{\sigma h }CB(B,A)\otimes_{\sigma h }D ,D\otimes_{\sigma h }A\otimes_{\sigma h }D))),$$
and one thus defines $I(x)(b)=(Id\otimes (b\otimes .)^*\otimes Id)(x)$. To check it is completely contractive in $b,$ one uses that for $b\in M_n(B),$  then $b\otimes .\in CB(M_n(A), M_n(B)\hat{\otimes} M_n(A)), (b\otimes .)^*\in NCB( CB(M_n(B),M_n(A)),M_n(A))\to NCB( CB(B,A),M_n(A))$ and $(Id\otimes (b\otimes .)^*\otimes Id)(x)\in D\otimes_{\sigma h }M_n(A)\otimes_{\sigma h }D\to M_n(D\otimes_{\sigma h }A\otimes_{\sigma h }D))$ by a very special case of the shuffle map from \cite[Th 6.1]{EffrosRuanDual}. This concludes the case with $D=\C$. The commutativity with functorial maps is obvious from the definitions.

Since $A\in {}_DNDOM_D$, thus $A$ is an operator module so that we have a map $D\otimes_hA\otimes_h D\to A$ and from the representation, it is separately weak-* continuous, thus from \cite[Prop 5.9]{EffrosRuanDual} it extends to a weak-* continuous completely contractive map $D\otimes_{\sigma h}A\otimes_{\sigma h} D\to A.$

From the result without module structure, one gets a weak-* continuous completely contractive map $D\otimes_{\sigma h }(B^*{{\otimes_F}}A)\otimes_{\sigma h }D \to B^*{{\otimes_F}}(D\otimes_{\sigma h }A\otimes_{\sigma h }D) \to (B^*{{\otimes_F}}A)$ and from \cite[Th 6.1]{EffrosRuanDual} we have the von Neumann tensor product analogue  $D\otimes_{\sigma h }(B^*{\overline{\otimes}}A)\otimes_{\sigma h }D \to B^*{\overline{\otimes}}(D\otimes_{\sigma h }A\otimes_{\sigma h }D) \to (B^*{\overline{\otimes}}A)$. Thus using also Effros-Ruan's CES Theorem for dual modules \cite[Th 3.8.3]{BLM} and from the characterization of weak-* continuity \cite[(5.22)]{EffrosRuanDual} one gets $(B^*{{\otimes_F}}A),(B^*{\overline{\otimes}}A)\in {}_DNDOM_D.$

For the shuffle map in the module case, it suffices to treat the case $C_2=\C$ by symmetry and associativity of the normal Haagerup product (see lemma \ref{IteratedModule} bellow).
Now composing the weak-* continuous map we obtained to the weak-* continuous completely quotient map from \cite[Th 4.2]{M05}, one gets : $$C_1\otimes_{\sigma h }(B{{\otimes_F}}A)\to B{{\otimes_F}}[C_1\otimes_{\sigma h }A]\to B{{\otimes_F}}[C_1\otimes_{\sigma h D}A],$$
and from the same theorem, it suffices to see it induces a map to the quotient. Since the kernel is weak-* closed, it suffices to show that it contains for any $c\in C_1, d\in D, T\in CB(B, A)$ $x=cd\otimes T-c\otimes dT$ but for $b\in B$, the image of  $x$ is $(id\otimes (b\otimes1)^*)(x)=cd\otimes T(b)-c\otimes dT(b)$ which is $0$ when mapped by the quotient map of the normal Haagerup product. This concludes. An easy diagram chasing with the definition of quotients proves the commutativity with functorial maps in the module case from the non-module case. The von Neumann tensor product case is similar starting from the non-module case from \cite[Th 6.1]{EffrosRuanDual}.
\end{proof}
As for $\ell^1$ direct sums, projective tensor products are not well behaved with respect to operator modules, but this suggests a definition of a module projective product.

\begin{corollary} \label{ShuffleStrongModule}
Let $D$ be a von Neumann algebra. 
For any operator space $B$ and any strong operator $D$-bimodule $A$, $B\widehat{\otimes}^{D-D}A:=(B^*\otimes_F A^\natural)_\natural$ is a strong operator $D$-bimodule with a canonical completely contractive map $B\widehat{\otimes}A\to B\widehat{\otimes}^{D-D}A.$
It satisfies the following universal property, it is the unique strong operator bimodule such that for any other strong operator bimodule $X$: $$CB_D(B\widehat{\otimes}^{D-D}A,X)_D=CB(B, CB_{D}(A,X)_{D}).$$
As a consequence, for $B_1,B_2$ operator spaces, we have :
$$B_1\widehat{\otimes}^{D-D}(B_2\widehat{\otimes}^{D-D}A)\simeq (B_1\widehat{\otimes}B_2)\widehat{\otimes}^{D-D}A.$$

 Moreover, the shuffle map extends for $C_1,C_2$ strong operator $D$-bimodule to a completely contractive $D$ bimodule map : $$S_{P e h}=S_{P e h}^{B,C_1,A,C_2 }:B\widehat{\otimes}^{D-D}[C_1\otimes_{e h D}A\otimes_{e h D}C_2]\to C_1\otimes_{e h D}[B\widehat{\otimes}^{D-D}A]\otimes_{e h D}C_2 ,$$
whose dual module map is the extension of the previous proposition \ref{ShuffleNormalModule} $$C_1^\natural\otimes_{\sigma h D'}(B^*{{\otimes_F}}A^\natural)\otimes_{\sigma h D'}C_2^\natural\to B^*{{\otimes_F}}[C_1^\natural\otimes_{\sigma h D'}A^\natural\otimes_{\sigma h D'}C_2^\natural].$$
Moreover the shuffle map makes commutative diagrams with the canonical maps coming from functoriality of tensor products from  completely bounded bimodule maps between strong bimodules $j:C_1\to C_3, k:C_2\to C_4, l:A\to A_2.$ More precisely, we have $[ j\otimes_{eh D} (Id\otimes l)\otimes_{eh D} k]\circ S_{P e h}^{B,C_1,A,C_2 }=S_{P e h}^{B,C_3,A_2,C_4 }\circ [Id\otimes (j\otimes_{eh D} l\otimes_{eh D} k)].$
\end{corollary}
\begin{proof}
The strong module statement is obvious from the characterization of strong modules \cite{M05} from \cite[Prop 3.11, Th 3.7]{M05} if $X$ is a strong operator $D$-module $$CB_D((B^*\otimes_F A^\natural)_\natural, X)_D\simeq NCB_{D'}(X^\natural, B^*\otimes_F A^\natural)_{D'}\subset CB(B, CB_{D'}(X^\natural,A^\natural)_{D'})$$ and since both spaces we take modular $CB$ maps on the right hand side are strong, this isomorphism is even completely isometric from \cite[Prop 3.12]{M05}. Moreover, it is easy to see that $NCB_{D'}(X^\natural, B^*\otimes_F A^\natural)_{D'}= CB(B, NCB_{D'}(X^\natural,A^\natural)_{D'})\simeq CB(B, CB_{D}(A,X)_{D})$ maps that are pointwise valued in normal $CB$ maps. But $CB(B, CB_{D}(A,X)_{D})\subset CB(B\widehat{\otimes} A,X)$. Thus taking $X=(B^*\otimes_F A^\natural)_\natural$ one gets from the identity map the expected canonical map $B\widehat{\otimes} A\to (B^*\otimes_F A^\natural)_\natural$ which is $A-A$ modular for the obvious module structure. 
We have also checked the universal property. Associativity follows straightforwardly.

For the shuffle map, we already recalled the two spaces for which we want an extension are strong, thus from \cite[Prop 3.11, 3.12]{M05} again, the unique module predual map $S$ of the previous shuffle map exists and have same $CB$ norm thus is also contractive and it suffices to check it extends the shuffle map of the algebraic product. Indeed, for $c_1\in C_1, c_2\in C_2, b\in B, a\in A, M\in CB(B, A^\natural), L\in C_1^\natural, N\in C_2^\natural :$
\begin{align*}\langle S(b\otimes c_1\otimes a\otimes c_2),L\otimes M\otimes N\rangle&=\langle ( c_1\otimes a\otimes c_2),S^*(L\otimes M\otimes N)[b]\rangle\\&=\langle ( c_1\otimes a\otimes c_2),(L\otimes M(b)\otimes N)\rangle\\&=\langle ( c_1\otimes (b\otimes a)\otimes c_2),(L\otimes M\otimes N)\rangle\end{align*}
and by weak-* density of the  elements of the form $L\otimes M\otimes N$ one concludes to $S(b\otimes c_1\otimes a\otimes c_2)=c_1\otimes (b\otimes a)\otimes c_2.$
The functoriality follows from the dual one in the previous proposition.
\end{proof}

We conclude by a technical shuffle lemma we did not find in the literature :

\begin{lemma}\label{Shuffle}
For any operator spaces $A,B,C_1,C_2,$ the shuffle map extends to a completely contractive maps in the following situations (the last two maps are unique extensions) :
$$B{\widehat{\otimes}}[C_1\otimes_{eh}A\otimes_{eh}C_2]\to C_1\otimes_{e h}(A\widehat{\otimes} B)\otimes_{e h}C_2,$$
$$B{\widehat{\otimes}}[C_1\otimes_{h}A\otimes_{h}C_2]\to C_1\otimes_{ h}(A\widehat{\otimes} B)\otimes_{ h}C_2,$$
$$B{{\otimes}_{nuc}}[C_1\otimes_{h}A\otimes_{h}C_2]\to C_1\otimes_{ h}(A{\otimes}_{nuc} B)\otimes_{ h}C_2,$$
$$C_1\otimes_{e h}CB(B,A)\otimes_{e h}C_2\to CB(B,C_1\otimes_{eh}A\otimes_{eh}C_2),$$
$$C_1\otimes_{ h}CB(B,A)\otimes_{ h}C_2\to CB(B,C_1\otimes_{h}A\otimes_{h}C_2).$$

Moreover, if $C_1$ (resp. $C_2$) is a D-E (resp. E-D) normal dual operator module over  von Neumann algebra $D,E$ and assume $A$ is now a normal dual operator E bimodule,  there is a completely contractive $D$ bimodular map extending the shuffle map, weak-* continuous on the normal Haagerup tensor product argument :
$$S_{P \sigma h}^{D,E}:B\widehat{\otimes}^{D-D}[C_1\otimes_{\sigma h E}A\otimes_{\sigma h E}C_2]\to C_1\otimes_{\sigma h E}(CB(B, A_\natural))^{\natural}\otimes_{\sigma h E}C_2,$$
where $CB(B, A_\natural)$ is given the structure of strong Haageruop E' module induced by the canonical map $E'\otimes_{eh}CB(B, A_\natural)\otimes_{eh}E'\to CB(B,E'\otimes_{eh}A_\natural\otimes_{eh}E')\to CB(B, A_\natural).$
Finally, when $E=D$, if $k=k_{B,A}^D:[B\widehat{\otimes}^{D-D}A]\to(CB(B, A_\natural))^{\natural}$ is the canonical map, $J:[C_1\otimes_{e h E}A\otimes_{e h E}C_2]\to [C_1\otimes_{\sigma h E}A\otimes_{\sigma h E}C_2]$ , $I: C_1\otimes_{e h E}(CB(B, A_\natural))^{\natural}\otimes_{e h E}C_2\to C_1\otimes_{\sigma h E}(CB(B, A_\natural))^{\natural}\otimes_{\sigma h E}C_2$ the contractive inclusions and $S_{P e h}:B\widehat{\otimes}^{D-D}[C_1\otimes_{e h D}A\otimes_{e h D}C_2]\to C_1\otimes_{e h D}[B\widehat{\otimes}^{D-D}A]\otimes_{e h D}C_2 $ the shuffle map of proposition \ref{ShuffleStrongModule}, then we have the commutative diagram : $I\circ (1\otimes k\otimes 1)\circ S_{P e h}=S_{P \sigma h}^{D,D}\circ [1\otimes J]$.

\end{lemma}
\begin{proof}

The first five maps are easy. To get the first, it suffices to take the predual map of the first map built in proposition \ref{ShuffleNormalModule} in the case $D=\C$ and reason as in corollary \ref{ShuffleStrongModule}.

The second map follows from the first since the second range is included in the first and working on elementary tensors enables to restrict to this new range.
The third map is induced from the one  $$S': B{{\otimes}_{nuc}}[C_1\otimes_{eh}A\otimes_{eh}C_2]\to C_1\otimes_{ eh}(A{\otimes}_{nuc} B)\otimes_{ eh}C_2$$ from \cite[Th 6.1]{EffrosRuanDual} on extended Haagerup products using that $C_1\otimes_{ h}(A\widehat{\otimes} B)\otimes_{ h}C_2\subset  C_1\otimes_{ eh}(A\widehat{\otimes} B)\otimes_{ eh}C_2$ completely isometrically and $[C_1\otimes_{h}A\otimes_{h}C_2]\subset [C_1\otimes_{eh}A\otimes_{eh}C_2]$
 so that the canonical map $[C_1\otimes_{h}A\otimes_{h}C_2]\widehat{\otimes}B\to [C_1\otimes_{eh}A\otimes_{eh}C_2]\widehat{\otimes}B$
induces a map by quotient  $Q:[C_1\otimes_{h}A\otimes_{h}C_2]{\otimes}_{nuc}B\to [C_1\otimes_{eh}A\otimes_{eh}C_2]{\otimes}_{nuc}B,$ (using functoriality of the injective tensor product) so that $S'\circ Q$ has actually a range on the closed subspace  $C_1\otimes_{ h}(A{\otimes}_{nuc} B)\otimes_{ h}C_2$ as seen on elementary tensors and is thus our third map.

The fourth and fifth maps are obtained as follows in seeing 
\begin{align*}CB(C_1\otimes_{(e)h}&CB(B, A)\otimes_{(e)h}C_2, CB(B,C_1\otimes_{(e)h}A\otimes_{(e)h}C_2))\\&=CB(B,CB(C_1\otimes_{(e)h}CB(B, A)\otimes_{(e)h}C_2, C_1\otimes_{(e)h}A\otimes_{(e)h}C_2))\end{align*}
and in taking  in the right hand side the map  corresponding to $b\mapsto 1\otimes ev(b)\otimes 1$ and using the functoriality from \cite[(5.11)]{EffrosRuanDual} (and the well-known one in the Haaggerup case) and complete boundeness of $ev$ in $b$ using also usual completely contractive maps $C_1\otimes_{h}M_n(A)\otimes_{h}C_2\to M_n(C_1\otimes_{h}A\otimes_{h}C_2)$ from \cite[Thm 5.15]{PisierBook} and~: \begin{align*}C_1\otimes_{eh}M_n(A)\otimes_{eh}C_2&\subset [C_1^*\otimes_h (T_n\widehat{\otimes}A^*)\otimes_h C_2^*]^*\\&\to (T_n\widehat{\otimes}[C_1^*\otimes_h A^*\otimes_h C_2^*])^*\\&=M_n (C_1^{**}\otimes_{eh}A^{**}\otimes_{eh}C_2^{**})\supset M_n (C_1\otimes_{eh}A\otimes_{eh}C_2) \end{align*} the first inclusion using injectivity of extended Haagerup product and its relation to Haagerup product, the second map in the middle being dual to the second shuffle map built above and the final map arriving in the smaller target subset  without biduals (cf detail later in this proof for the normal Haagerup case.)

The canonical map to get the strong module structure is obtained in using the fourth map just built above.
 This concludes to the strong operator module structure on $CB(B,A_\natural),$ using $A_\natural$ is strong from \cite[Th 2.17]{M05}.

To build the $D$ bimodular map extension of  the shuffle map, one can use the universal property proved in proposition \ref{ShuffleStrongModule}. Indeed the target space is a strong $D$-bimodule as any dual operator module since the module map extends to normal Haagerup thus to the included extended Haagerup product  from which we deduce the defining stability for strong modules. Thus,  it suffices to build a map : $$S=S_{P \sigma h}^{D,E}\in CB(B,NCB_D([C_1\otimes_{\sigma h E}A^{*}\otimes_{\sigma h E}C_2], C_1\otimes_{\sigma hE}(CB(B, A_\natural))^{\natural}\otimes_{\sigma h E}C_2)_D)$$ corresponding to the shuffle map.

But for any $b\in B$, we already noted that $ev(b)\in CB_{E'}(CB(B,A_\natural),A_\natural)_{E'}$ with $ev\in CB(B,CB(CB(B,A)_\natural,A_\natural)\simeq CB(CB(B,A_\natural),CB(B,A_\natural))$ corresponds to identity.  Thus since $A_\natural, CB(B,A)_\natural$ are strong modules, by \cite[Prop 3.12,Th 3.7]{M05} the map induces a map $(ev(b))^{\natural}\in NCB_E(A,(CB(B,A_\natural))^{\natural})_E$ with same cb norm, so that $$S(b)=Id\otimes ((ev(b))^{\natural}\otimes Id\in NCB_D([C_1\otimes_{\sigma h E}A\otimes_{\sigma h E}C_2], C_1\otimes_{\sigma h E}(CB(B,A_\natural))^{\natural}\otimes_{\sigma h E}C_2)_D.$$ The $D$-modularity is obvious, so is normality, and we check the map is indeed completely bounded. Indeed for $b\in M_n(B)$, $$(ev(b))^{\natural}\in NCB_E(A,(CB(M_n(B),A_{\natural})^{\natural})_E\to NCB_E(A,M_n[(CB(B,A_\natural))^{\natural})])_E$$ the last map being obtained by  dualising the canonical map (obtained using evaluation maps and commutativity of projective product. We also  use \cite[Th 2.5.1]{PisierBook} to recall $M_n(A)\subset M_n(A^{**})\simeq CB(A^*,M_n(\C))=CB(A^*,CB(T_n,\C))=CB(T_n,A^{**})\supset CB(T_n,A)$ all inclusions being completely isometric so that  $CB(T_n,A)=M_n(A)$ with $T_n=M_n(\C)^*$) and for $K$ a proper  left Hilbert $E'$ module (the third map comes from the induced modular quotient of the fifth map above) : \begin{align*}T_n\widehat{\otimes}(K^*\otimes_{h E'}CB(B,A_\natural)\otimes_{h E'} K)&\to T_n\widehat{\otimes}(K^*\otimes_{h E'}CB(M_n(B),M_n(A_\natural))\otimes_{h E'} K)\\&=T_n\widehat{\otimes}(K^*\otimes_{h E'}CB(T_n,CB(M_n(B),A_\natural))\otimes_{h E'} K) \\&\to T_n\widehat{\otimes}CB(T_n,[(K^*\otimes_{h E'}CB(M_n(B),A_\natural))\otimes_{h E'} K])\\& \to  K^*\otimes_{h E'}CB(M_n(B),A_\natural)\otimes_{h E'} K.\end{align*}  Thus it suffices to note there is a completely contractive map $$C_1\otimes_{\sigma h E}M_n((CB(B,A_\natural))^{\natural})\otimes_{\sigma h E}C_2\to M_n(C_1\otimes_{\sigma h E}(CB(B,A_\natural))^{\natural}\otimes_{\sigma h E}C_2)$$ which is again a special case of (the module variant from proposition \ref{ShuffleNormalModule} of )\cite[Th 6.1]{EffrosRuanDual} to deduce the expected complete boundedness for $S.$
 
 It is also easy to see that $S$ restricts to the expected shuffle map on elementary tensors.

For the functoriality statement, since the maps are build by the universal property of the tensor product, it suffices to prove the relation for $b\in B$ fixed. Let also $S'=S_{P e h}$ for short.

Of course, one can define $$T(b)=Id\otimes ((ev(b))^{\natural}\otimes Id\in CB_D([C_1\otimes_{e h E}A\otimes_{e h E}C_2], C_1\otimes_{e h E}(CB(B,A_\natural))^{\natural}\otimes_{e h E}C_2)_D,$$ and $S(b)J=IT(b)$
is easy from the definitions.

Then recall $[B\widehat{\otimes}^{D-D}A]=[CB(B,A^\natural)]_\natural$ and that $k=k_1^\natural\circ k_2$ with $k_2:[CB(B,A^\natural)]_\natural\to [CB(B,A^\natural)]^\natural$ the injection and $k_1^\natural:[CB(B,A^\natural)]^\natural\to [CB(B,A_\natural)]^\natural$ the dual of the complete isometry $k_1:CB(B,A_\natural)\to CB(B,A^\natural).$

We have to check $(1\otimes k\otimes 1)\circ S'(b)=T(b)$ and since the source and target space are strong modules, from \cite[Prop 3.11]{M05} there is a unique predual map so that it suffices to check $$T(b)^\natural=S'(b)^\natural \circ (1\otimes k^\natural\otimes 1)\in NCB(C_1^\natural\otimes_{\sigma h E'}(CB(B,A_\natural))^{\natural\natural}\otimes_{\sigma h E'}C_2^\natural, [C_1\otimes_{e h E}A\otimes_{e h E}C_2]^\natural).$$ Both map are weak-* continuous thus it suffices to check they agree on the weak-* dense set $C_1^\natural\otimes_{h E'}(CB(B,A_\natural))\otimes_{ h E'}C_2^\natural$ (from \cite[Th 4.2, Corol 3.6]{M05}
 and \cite[lemma 5.8]{EffrosRuanDual}) where $T(b)^\natural=1\otimes ev(b)\otimes 1.$ But on this subspace $(1\otimes k^\natural\otimes 1)=k_1$ and by definition $S'(b)^\natural=1\otimes ev(b)\otimes 1$ on $C_1^\natural\otimes_{\sigma h E'}(CB(B,A^\natural))\otimes_{\sigma h E'}C_2^\natural$ one indeed obtain the expected equality $S'(b)^\natural k_1=T(b)^\natural.$
 
 Of course one can deal with the case $E\neq D$ if we change the previous corollary to deal with this case.\end{proof}

\subsection{Extended Haagerup products with matricially normed spaces}
We will need to leave sometimes the class of operator spaces to consider more general matricially normed space in the sense \cite{BlecherPaulsen}, i.e. satisfying axioms (i),(ii) p 264. For $X$  such a matricially normed space, we call column dual $X'_c=(X'_l)^{op}=(B_l(X,\C))^{op}$ the opposite of the left dual of their example 2.7, alternatively described in proposition 2.9 using example 2.8. We recall that a  matrix cross normed space in the sense of \cite[p 1764]{MathesPaulsen} is a matricially normed space such that the norm on $M_n(X)$ is a cross normed in Banach space sense (also called reasonnable tensor norm, namely a norm between projective and injective tensor norm).

To stick to matricially normed spaces in the sense above, we will use a variant of $B_l(X,Y)$ as in \cite[ex 2.8]{BlecherPaulsen}, agreeing in the case $Y=\C$. We won't use the variant $\Pi_l(X,Y)$ they suggested, but rather put a alternative matricially normed structure on completely bounded maps.

  \begin{exemple} Let $X,Y$ be matricially normed spaces.  We write $CB_l(X,Y)$ the set of completely bounded maps with matricially normed space structure defined (as for $\Pi_l$) by the isometry $M_n(CB_l(X,Y))\simeq CB(C_n\otimes_h X,C_n\otimes_h Y)$
 with $\Phi=(\phi_{ij})\in M_n(CB_l(X,Y))$ mapped to $\Phi((x_1,...,x_n)^T)=(\sum_{i}\phi_{1i}(x_i),...,\sum_{i}\phi_{ni}(x_i))^T.$
 Similarly the space of completely bounded maps is given the matricially normed space structure $CB_r(X,Y)=(CB_l(X^{op},Y^{op}))^{op}$. Obviously $CB_r(X,.), CB_l(X,.)$ are (covariant) functors.
\end{exemple}

We will even have to use an extended Haagerup tensor product with one argument being a matricially normed space and not an operator space. This will be used mostly to handle infinite matrix spaces over matricially normed spaces and corresponding operations. For this we use Pisier's factorization by an Hilbert space as a definition. 

If $X_1,...,X_n$ are operator spaces and $Y$ is a matricially normed space, we define 
$$X_1\otimes_{eh}...\otimes_{eh}X_n\otimes_{eh}Y:=\Gamma_R^{\sigma}(X_1^*\otimes_{h}...\otimes_{h}X_n^*,Y),\ Y\otimes_{eh}X_1\otimes_{eh}...\otimes_{eh}X_n:=
\Gamma_C^{\sigma}(X_1^*\otimes_{h}...\otimes_{h}X_n^*,Y).$$
where $\Gamma_R(X,Y)$ (resp. $\Gamma_C(X,Y)=\Gamma_R(X^{op},Y^{op})^{op}$) is as in \cite[p 39-40]{PisierOH} is the matricially normed space of maps factorizing through a row Hilbert space (resp. a column Hilbert space) with matrix normed structure as in \cite{EffrosRuanFacto} where this space is studied in more detail. The index $\sigma$ indicates we take maps weak-* continuous in each arguments, namely with the notation of \cite{EffrosRuanDual} for normal maps $$\Gamma_R^{\sigma}(X_1^*\otimes_{h}...\otimes_{ h}X_n^*, Y):=\Gamma_R(X_1^*\otimes_{ h}...\otimes_{ h}X_n^*, Y)\cap CB_m^{\sigma}(X_1^*\times...\times X_n^*, Y^{**}).$$

Note that when $Y$ is an operator space we can gather results from the literature to get  the expected consistency statement :
\begin{lemma}\label{MatricialEHaagerup}
If $X_1,...,X_n,Y$ are operator spaces the usual extended Haagerup product satisfies $$X_1\otimes_{eh}...\otimes_{eh}X_n\otimes_{eh}Y=\Gamma_R^{\sigma}(X_1^*\otimes_{h}...\otimes_{h}X_n^*,Y),\ Y\otimes_{eh}X_1\otimes_{eh}...\otimes_{eh}X_n=
\Gamma_C^{\sigma}(X_1^*\otimes_{h}...\otimes_{h}X_n^*,Y).$$
Moreover, if $Y$ is only a matricially normed space, the new extended Haagerup result has the following associativity maps which are completely contractive :
$$X_1\otimes_{eh}...\otimes_{eh}X_n\otimes_{eh}Y\simeq (X_1\otimes_{eh}...\otimes_{eh}X_n)\otimes_{eh}Y, X_1\otimes_{eh}(Y\otimes_{eh}X_2)\simeq (X_1\otimes_{eh}Y)\otimes_{eh}X_2,$$ $$ X_1\otimes_{eh}X_2\otimes_{eh}Y\to X_1\otimes_{eh}(X_2\otimes_{eh}Y).$$
The last map is even a complete isometry when $X_2=H_c.$
We also have for any Hilbert space $H$ and $Y$ matricially normed spaces , a complete isometry $Y\otimes_{eh} H_r\simeq CB_l(H_c^*,Y).$

Finally, if $M_{1,n}(Y)$ is 2-column summing for all $n$, we have a completely isometric map :
$$X_1\otimes_{h}...\otimes_{h}X_n\otimes_{h}Y\to X_1\otimes_{eh}...\otimes_{eh}X_n\otimes_{eh}Y.$$
\end{lemma}

\begin{proof}
From \cite{EffrosRuanDual} we have in the operator space case $$X_1\otimes_{eh}...\otimes_{eh}X_n\otimes_{eh}Y=(X_1^*\otimes_{ h}...\otimes_{ h}X_n^*\otimes_h Y^*)^*_\sigma\subset (X_1^*\otimes_{h}...\otimes_{ h}X_n^*\otimes_h Y^*)^*$$

Moreover, from \cite[(5.8)]{EffrosRuanFacto} 
 $$(X_1^*\otimes_{h}...\otimes_{ h}X_n^*\otimes_h Y^*)^*=\Gamma_R(X_1^*\otimes_{ h}...\otimes_{ h}X_n^*, Y^{**})\simeq \Gamma_R(X_1^*,...,\Gamma_R(X_n^*,Y^{**})...).$$
 
 Thus adding the normality conditions and using the notation of \cite{EffrosRuanDual}, one gets $$X_1\otimes_{eh}...\otimes_{eh}X_n\otimes_{eh}Y= \Gamma_R(X_1^*\otimes_{ h}...\otimes_{ h}X_n^*, Y^{**})\cap  CB_m^{\sigma}(X_1^*\times...\times X_n^*, Y^{**})\cap CB((X_1^*\otimes_{ h}...\otimes_{ h}X_n^*), Y)$$
But from  \cite[Prop 5.2]{EffrosRuanFacto} and its proof, one gets a complete isometry (since the intersection norm is the factorization norm): $$\Gamma_R(X_1^*\otimes_{ h}...\otimes_{ h}X_n^*, Y^{**})\cap  CB((X_1^*\otimes_{ h}...\otimes_{ h}X_n^*), Y)\simeq  \Gamma_R(X_1^*\otimes_{ h}...\otimes_{ h}X_n^*,  Y).$$
 
 
We thus got the expected formula and the column case is similar.

 
Let now $Y$ be any matricially normed space. For simplicity we only consider the case $n=2$. By \cite[(5.22)]{EffrosRuanDual}, 
we have an obvious completely contractive restriction map to $ X_1^*\otimes_{h}X_2^*\subset (X_1\otimes_{eh}X_2)^*$ :   \begin{align*}(X_1\otimes_{eh}X_2)\otimes_{eh}Y&=\Gamma_R((X_1\otimes_{eh}X_2)^*,Y)\cap CB^{\sigma}((X_1\otimes_{eh}X_2)^*,Y^{**})\\&=\Gamma_R((X_1\otimes_{eh}X_2)^*,Y)\cap CB_m^{\sigma}(X_1^*\times X_2^*, Y^{**})
\to X_1\otimes_{eh}X_2\otimes_{eh}Y.\end{align*}
 
To show it is a completely isometric isomorphism, it suffices to show any map in the target has an extension with the same norm in the source (injectivity follows from \cite[lemma 5.8]{EffrosRuanDual}. But take a map $u\in \Gamma_R(X_1^*\otimes_{h}X_2^*,Y)\cap CB_m^{\sigma}(X_1^*\times X_2^*, Y^{**})$ with factorization $u=AB, A\in CB(H_r,Y), B\in CB(X_1^*\otimes_{h}X_2^*,H_r)$. Up to getting smaller norms we can replace $H$ by a space anti-isomorphic to $K^*=A^*(Y^*)\subset H^*$ which is a closed subspace and composing $B$ with the projection, but in this case, for $k=A^*(y^*)\in K^* k\circ B=y^*\circ u$ is normal in each argument thus so is $B\in CB_m^{\sigma}(X_1^*\times X_2^*,K_r)$ and by \cite[(5.22)]{EffrosRuanDual} again $B$ extends to $\tilde{B}\in CB^{\sigma}((X_1\otimes_{eh}X_2)^*,K_r)$ and thus $\tilde{u}=A\tilde{B}$ is the desired factorization for the unique extension (which thus exists with this factorization). We can reason similarly at the matrix level to see complete isometry.

The second associativity isomorphism easily comes from (the proof of) commutativity of projective tensor product. Indeed, if $U\in X_1\otimes_{eh}(Y\otimes_{eh}X_2)$, one has a factorzation $U=AB$, $B\in CB(X_1^*,H_r),A\in CB(H_r,\Gamma_C(X_2^*,Y)).$ For any $e_i\in H_r, i\in I$ in an orthonormal basis, one writes $A(e_i)=C_iD_i$, $D_i\in CB(X_2^*,(K_i)_c),C_i\in CB((K_i)_c,Y)$. Then define, $K=\oplus_{i\in I}K_,i$ $D(\sum\lambda_i e_i)(y)=\oplus_{i\in I}\lambda_i D_i(y)$ so that $D\in B(H_r,CB(X_2^*,K_c))$ and if $C(\oplus_i x_i)\sum_i C_i(x_i)$ so that if we assumed before that $D_i(X_2^*)$ is dense in $K_i$...

For $u\in  X_1\otimes_{eh}X_2\otimes_{eh}Y$ one  takes a decomposition $u=AB$ as before and we now see $B\in CB_m^{\sigma}(X_1^*\times X_2^*,K_r)\subset \Gamma_R(X_1^*\times X_2^*,K_r)\simeq \Gamma_R(X_1^*,\Gamma_R(X_2^*,K_r))$ by \cite[(5.8)]{EffrosRuanFacto} Thus take $B=CD, D\in \Gamma_R(X_1^*,H_r), C\in \Gamma_R(H_r,\Gamma_R(X_2^*,K_r))$, it is easy to see that taking $H$ as the closure of $D(X_1^*)$, one can get $C(\xi)\in \Gamma_R(X_2^*,K_r) $ normal and thus also, from the choice of $K$ above, $[[A\circ .]\circ C](\xi):=A\circ(C(\xi))\in X_2\otimes_{eh}Y$ and $||[[A\circ .]\circ C](\xi)||_{X_2\otimes_{eh}Y}\leq ||A||_{cb}||(C(\xi))||_{cb}\leq ||A||_{cb}||C||_{\Gamma_R(H_r,\Gamma_R(X_2^*,K_r))} ||\xi||$, i.e. $[[A\circ .]\circ C]\in CB(H_r,X_2\otimes_{eh}Y)$ and from a matrix variant one estimates the completely bounded norm by $||A||_{cb}||C||_{cb}$. We have thus written $u=[[A\circ .]\circ C]D\in \Gamma_R(X_1^*,X_2\otimes_{eh}Y)$ with the right norm estimate. It only remains to check normality, i.e. $u\in CB^{\sigma}(X_1^*,(X_2\otimes_{eh}Y)^{**}).$ But take a net $x_n\to x$ in $X^*$ weak-*. Arguing as in the previous paragraph, maybe replacing $H$, one can assume $D$ weak-* continuous and thus $D(x_n)\to D(x)$ weak-* in $H$ and let $\phi\in (X_2\otimes_{eh}Y)^*.$ Since $\phi(u(x_n))$ bounded, it suffices to get $\phi(u(x))$ as unique cluster point, thus take a subnet with $\phi(u(x_n))$ converging to $y$, $(D(x_n),\phi(u(x_n)))$ converging to $(D(x),y)$ in an Hilbert (thus reflexive) space, and by Hahn Banach there is a convex combination of points in this net converging normwise to $(D(x),y)$. Taking a corresponding convex combination $z_n$ of $x_n$, we have another net with $(D(z_n),\phi(u(z_n)))\to (D(x),y)$. But now $u(z_n)$ converges normwise to $u(x)$
thus $\phi(u(z_n))\to \phi(u(z))=y.$ 
Now, conversely,  if $X_2=H_c$, $U\in\Gamma_R^\sigma(X_1^*,\Gamma_R^\sigma(X_2^*,Y)),$ U factorizes as AB, $A\in CB(K_r,CB(H_c^*,Y)),B\in CB(X_1^*,K_r)$ completely isometrically.
But we have  $CB(K_r,CB(H_c^*,Y))\simeq CB(K_r\hat{\otimes} H_c^*,Y)\simeq CB(K_r\otimes_h H_c^*,Y),$ (the second complete isometry coming from a well known identity e.g. \cite[Prop 9.3.2]{EffrosRuan}, the first identity being similar to the case $Y$ operator space). Thus $A(B\otimes I)$ is a row Hilbert space factorization of $\Phi(U)$ seen in $CB^{\sigma}(X_1^*\times X_2^*,Y)$ (which is multiplicative as a consequence). This gives $$||\Phi(U)||_{X_1\otimes_{eh}(X_2\otimes_{eh}Y)}\leq ||U||_{X_1\otimes_{eh}X_2\otimes_{eh}Y},$$
using the separate normality is equivalent to the weak-* to weak continuity coming from $X_1\otimes_{eh}X_2$ by \cite[Prop 5.9]{EffrosRuanDual} (which does not use the target operator space and is mostly a Banach space result). The complete isometry is obvious since it boils down to replacing $Y$ by $M_n(Y).$

We can now prove the final statement relating Haagerup and extended Haagerup products. First from associativity of Haagerup and extended Haagerup products,injectivity of the Haagerup product in this setting \cite{BlecherPaulsen} and since it is known in the operator space case $X_1\otimes_h...\otimes_hX_n\to X_1\otimes_{eh}...\otimes_{eh}X_n$
is a complete isometry \cite{EffrosRuanDual}, we can be content with the case $n=1$. This is then a consequence of the reamrk after \cite[Th 3.11]{BlecherPaulsen} since the factorization norm is the norm taken in $X_1''\otimes_hY$, which is, since $X_1$ is an operator space and we are in position to apply the matricial case of injectivity of Haagerup tensor product, the same as the norm in $X_1\otimes_hY.$

Let us now prove the identity with $H_r$. It boils down to the study of the matrix norm given in \cite{EffrosRuanFacto} to $\Gamma_C(H^*_c,Y)$ since the isometry is obvious by definition. Thus take $\phi\in(\phi_{ij})\in M_n(\Gamma_C(H^*_c,Y))$ and write it as $\phi_{ij}=\tau_i\circ \sigma_j$ as in this paper with $\sigma:H_c\to M_{1,n}(K_c), \tau:K_c\to M_{n,1}(Y)$ for $K_c$ a column Hilbert space.
Then one induces $\sigma\otimes 1: H_c\otimes_h C_n\to K_c\otimes_h R_n\otimes_hC_n$
 and composing with the trace on trace class $Tr: R_n\otimes_hC_n\to \C$ for $v=(v_j)_{j=1...n}\in H_c\otimes_h C_n=M_{n,1}(H_c)$

$\sigma':v\mapsto (1\otimes Tr)\circ (\sigma\otimes 1)[v]=\sum_j\sigma_j(v_j)$ which has $CB$ norm less than $||\sigma||$. Thus the map $v\mapsto \sum_{j}\phi_{ij}(v_j)$ is $\tau\circ \sigma'$ and thus $||\phi||_{M_n(CB_l(H^*_c,Y))}\leq ||\tau|-||\sigma||$ and taking the infemum, $||\phi||_{M_n(CB_l(H^*_c,Y))}\leq ||\phi||_{\Gamma_C(H^*_c,Y)}.$

Conversely, we consider the following decomposition with $K_c=M_{n,1}(H_c)=C_n\otimes_hH_c$, $\sigma(x)=x\otimes Id_n=(e_1\otimes x,...,e_n\otimes x))\in M_{1,n}(K_c)$ so that of course $||\sigma||_{cb}=1$. Then we take $\tau(\phi):M_{n,1}(H_c)\to M_{n,1}(Y)$ the map with $\tau(\phi)(\sum_j e_j\otimes v_j)=(\sum_j \phi_{ij}(v_j))_i.$
Thus indeed $\tau(\phi)_i\circ \sigma_j(v)=\tau(\phi)_i(e_j\circ v)=\phi_{ij}(v)$ as expected and we thus deduce the reverse inequality :$||\phi||_{\Gamma_C(H^*_c,Y)}\leq ||\phi||_{M_n(CB_l(H^*_c,Y))}.$
\end{proof}


We take inspiration of infinite matrix representations of \cite{EffrosRuanHopf} and of the operator module case in \cite{M97}. We even push further the idea to avoid the use of operator space duality as in the first context and of concrete Hilbert space representations as in the second.

For $Z$ a matricially normed space, we write $$M_{I,J}(Z)=(\ell^2(I))_c\otimes_{eh}(Z\otimes_{eh}(\ell^2(J)^*)_r)\simeq ((\ell^2(I))_c\otimes_{eh}Z)\otimes_{eh}(\ell^2(J)^*)_r$$ for sets $I,J$.
Note that from lemma \ref{MatricialEHaagerup}, one deduces $M_{K,L}(M_{I,J}(Z))\simeq M_{I\times K,J\times L}(Z).$

For $B\in M_{K,I}(\C)=B(\ell^2(I),\ell^2(K))=CB(\ell^2(I)_c,\ell^2(K)_c)$ we have actions $M_{K,I}(\C)\times M_{I,J}(Z)\to M_{K,J}(Z),M_{I,J}(Z)\times M_{J,K}(Z)\to M_{I,K}(Z).$ For instance the first one $BU\in M_{K,J}(Z)\simeq CB((\ell^2(K)^*)_r,M_{1,J}(Z))$ is the composition $U\circ B^*$ with the same viewpoint for $U\in CB((\ell^2(I)^*)_r,M_{1,J}(Z)).$

We will need this in the next preparatory subsection.

\subsection{Left module duals and Haagerup tensor product}

As a technical tool for our next result, we have now to use a left module duality. We start by some reminders.

For $X$ a matricially normed spaces with A-B h-bimodule (i.e. with a map $A\otimes_{h}X\otimes_{h}B\to X$, with $A,B$  von Neumann algebras, really soon assumed finite) we will write $CB_{l}(X,B)_B\subset CB_{l}(X,B)$ completely isometrically the subspace of right $B$ modular maps i.e. with $\phi(.b)=\phi(.)b, b\in B$. It has a B-A bimodule structure given by $(b\phi a)(x)=b\phi(ax)$. More generally we have 

\begin{lemma}\label{ModuleCBL}Let $X$ a matrix cross normed space as well as all $M_{n}(X)$ and $Y,Z$ matricially normed spaces we have canonical complete contractions :
$$S^X_{Y,Z}:X\otimes_hCB_l(Y,Z)\to CB_l(Y,X\otimes_hZ),\ \ \  ev_{Y,Z}^l:CB_l(Y,Z)\otimes_h Y\to Z,$$
$$A_{X,Y}^Z:CB_l(X\otimes_hY,Z)\to CB_l(X,CB_l(Y,Z)).
$$
The last map is even a complete isometry when $X,Y$ operator spaces.

Let $A,B,C$ $C^*$ algebras. As a consequence, for $X$ an A-B h-bimodule, $Y$ a C-B h-bimodule, $CB_{l}(X,Y)_B$ is a $C-A$ h-bimodule, i.e. the module structure $(c\phi a)(x)=c\phi(ax), c\in C, a\in A$ extends to map $$C\otimes_{h}CB_{l}(X,Y)_B\otimes_{h}A\to CB_{l}(X,Y)_B.$$
If moreover $X,Y$ are dual Banach spaces and $A\otimes_hX\to X,C\otimes_hY\to Y$ are pointwise (in $A,C$) weak-* continuous (in $X,Y$), the space of weak-* continuous map $NCB_{l}(X,Y)_B$ is also a C-A h-bimodule.
\end{lemma}
\begin{proof}
The first map is induced from $S^X_{Y,Z}(x,\phi):b\mapsto x\otimes\phi(b)$ and for a decomposition $U=\sum_k A_k\odot B_k$, $A_k\in M_{n,n_k}(X), B_k\in M_{n_k,n}(CB_{l}(Y,Z))=CB(M_{n,1}(Y),M_{n_k,1}(Z))$ one wants to show $S^X_{Y,Z}(U)\in CB(M_{n,1}(Y),M_{n,1}(X\otimes_hZ))$ thus take $y=(y_{(ij),l})_{j\leq n,i,l\leq K}\in M_K(M_{n,1}(Y))$ so that $B_k(y)\in M_K(M_{n_k,1}(Z))$
\begin{align*}||S^X_{Y,Z}(U)(y)||_{M_K(M_n(X\otimes_hZ))}&=||\sum_k (A_k\otimes Id_K) \odot (B_k(y))||\\&\leq \sum_k||A_k\otimes Id_K||_{M_K(M_{n,n_k}(X))}||B_k(y)||_{ M_K(M_{n_k,1}(Z))}\\&\leq \sum_k||A_k||_{M_{n,n_k}(X)}||B_k||_{CB(M_{n,1}(Y),M_{n_k,1}(Z))}||y||_{M_K(M_{n,1}(Y))}\end{align*}
using $||A_k\otimes Id_K||_{M_K(M_{n,n_k}(X))}=||A_k||_{M_{n,n_k}(X)}$
since $X$ is a matrix cross normed space, and thus establishing the stated complete contraction.

Similarly for $U=\sum_k A_k\odot B_k$, $A_k\in M_{n,n_k}(CB_l(Y,Z))=CB(M_{n_k,1}(Y),M_{n,1}(Z)), B_k\in M_{n_k,n}(Y)$ we have the estimate giving the second complete contraction :
$$||ev_{Y,Z}^l(U)||_{M_n(Z)}\leq \sum_k||A_k(B_k)||_{M_n(Z)}\leq ||A_k||_{CB(M_{n_k,1}(Y),M_{n,1}(Z))}||B_k||_{M_{n_k,n}(Y)}$$

For $\phi\in CB(X\otimes_hY,Z)$ we have 
  \begin{align*}||A_{X,Y}^Z(\phi)&||_{CB(X,CB_l(Y,Z))}=\sup_{KL}\sup_{||(x_{kl})||_{M_{KL}(X)}\leq 1}||A_{X,Y}^Z(\phi)(x_{kl})||_{M_{KL}(CB_l(Y,Z))=CB(M_{L,1}(Y),M_{K1}(Z))}\\&=\sup_{KLn}\sup_{||(x_{kl})||_{M_{KL}(X)}\leq 1}\sup_{||(y_{(jl)i})||_{M_n(M_{L,1}(Y))}\leq 1}||\sum_l A_{X,Y}^Z(\phi)(x_{kl}\otimes Id_{jj})(y_{(jl)i})||_{M_n(M_{K1}(Z))}
  \\&=\sup_{KLn}\sup_{||(x_{kl})||_{M_{KL}(X)}\leq 1}\sup_{||(y_{(jl)i})||_{M_n(M_{L,1}(Y))}\leq 1}||\phi(\sum_{lj} (x_{kl}\otimes Id_{jj}) \otimes y_{(jl)i})||_{M_n(M_{K,1}(Z))}
 \\&\leq\sup_{Kn }\sup_{||(u_{kn})||_{M_n(M_{K,1}(X\otimes_h Y))}\leq 1}||\phi( u_{ki})||_{M_{Kn}(Z)} =||\phi||_{CB(X\otimes_hY,Z)}
  \end{align*}
since $||(x_{kl}\otimes Id_{jj})_{(kj),(lj)}||=||(x_{kl})||$ so that one gets the stated contraction and we have equality in the operator space case, using \cite[lemma 3.2]{BlecherPaulsen}
For the complete statement, one uses :
\begin{align*}&M_n(CB_l(X\otimes_hY,Z))\simeq CB(C_n\otimes_hX\otimes_hY,C_n\otimes_h Z)\\& CB(C_n\otimes_hX,CB_l(Y,C_n\otimes_hZ))\simeq CB(C_n\otimes_hX,C_n\otimes_hCB_l(Y,Z))\simeq M_n(CB_l(X,CB_l(Y,Z))).\end{align*}

One can now compose respectively the following contractions and complete contractions~:
$$.\otimes_hId_Z:CB(CB_{l}(X,Y),CB_l(Z\otimes_hX,Y)))\to CB(CB_{l}(X,Y)\otimes_hZ,CB_l(Z\otimes_hX,Y))\otimes_hZ),$$
$$A_{CB_l(Z\otimes_hX,Y))\otimes_hZ,X}^Y(ev_{Z\otimes_hX,Y}^l):CB_l(Z\otimes_hX,Y))\otimes_hZ\to CB_l(X,Y)$$
for the previously built $ev_{Z\otimes_hX,Y}^l\in CB(CB_l(Z\otimes_hX,Y)\otimes_h(Z\otimes_h X),Y)$ to define $\Psi_{X,Y}^Z(u)[x]=A_{CB_l(Z\otimes_hX,Y))\otimes_hZ,X}^Y(ev_{Z\otimes_hX,Y}^l) [(u\otimes_hId_Z)(x)]$ so that we get a contraction
$$\Psi_{X,Y}^Z:CB(CB_{l}(X,Y),CB_l(Z\otimes_hX,Y)))\to CB(CB_{l}(X,Y)\otimes_{h}Z, CB_{l}(X,Y))$$

Starting with the adjoint of a multiplication map $m:A\otimes_h X\to X$  one gets by functoriality of duality a map $m^*:CB_{l}(X,Y)\to CB_{l}(A\otimes_h X,Y).$ Applying the above map gives exactly the multiplication map we want: $\Psi_{X,Y}^A(m^*)\in CB(CB_{l}(X,Y)\otimes_{h}A, CB_{l}(X,Y))$



Moreover, for the $B-A$ h-bimodule structure, $B\otimes_{h}CB_{l}(X,B)_B\to CB_{l}(X,B)_B$ is induced from the composition of the canonical map and the multiplication map for h-modules $B\otimes_{h}CB_{l}(X,Y)\to CB_{l}(X,B\otimes_hY)\to CB_{l}(X,Y).$

The weak-* continuous case is obvious since then  we have a map $C\otimes_{h}NCB_{l}(X,Y)_B\otimes_{h}A\to CB_{l}(X,Y)_B,$ and since  $c\otimes_h\phi \otimes_hb \mapsto (x\mapsto c\phi(ax))$ has range in $NCB_{l}(X,Y)_B$, so is the norm convergent extension.

\end{proof}
Thus for $X$ a matricially normed spaces with A-B h-bimodule
$CB_{l}(X,B)_B=Y$ is B-A h-bimodule and thus $CB_l(A,Y)_A$ is matricially normed spaces with B-A h-bimodule structure.
We define the module dual as $$X_l^\natural:=CB_l(A, CB_{l}(X,B)_{B})_{A}.$$
 Note that if say $A,X$ are operator spaces  $CB_l(A, CB_{l}(X,B))=CB_l(A\otimes_hX,B).$
Thus using the multiplication map :$A\otimes_hX\to X$ and its right inverse $1\otimes. :X\to A\otimes_hX$ one gets  maps $CB_l(X,B)\to CB_l(A\otimes_hX,B)\to CB_l(X,B)$ which restricts to a complete isomorphism :$CB_{l}(X,B)_{B}\simeq X_l^\natural.$ If moreover $X$ has a strong A-module structure, namely an extension of the multiplication map $A\otimes_{eh}X\to X$ this is a completely isometric subspace  $X_l^\natural\subset CB_l(A\otimes_{eh}X,B)_B.$
If $A,B$ are von Neumann algebras and $X$ is a normal dual operator module, we introduce a predual left module dual :$$X_{l\natural}=NCB_{l}(X,B)_{B}.$$ 

\begin{lemma}\label{matrixMultPredual}
Let $A,B$ von Neumann algebras, and $X$ a strong operator A-B bimodule, then  \begin{align*}M_{i,I}(X_l^\natural)&=CB_l({\ell^2(I)}_c,CB_{l}(X,M_{i,1}(B))_{B}))\simeq CB_l({\ell^2(I)}_c\otimes_h X,M_{i,1}(B))_{B})\\&\subset CB_l({\ell^2(I)}_c\otimes_{eh} X,M_{i,1}(B))_{B})=CB_l(M_{I,1}( X),M_{i,1}(B))_{B}).\end{align*} 

Moreover, there are multiplication maps for any sets $I,J,K$, and $k,k'$ finite, extending the operator module structure (resp. for $X\in {}_ANDOM_B$):
$$.[.]:M_{k,I}(X^\natural_l)\times M_{I,J}(A)\to M_{k,J}(X^\natural_l), (resp . \ \ .[.]:M_{k,I}(X_{l\natural})\times M_{I,J}(A)\to M_{k,J}(X_{l\natural})\ \ ),$$ $$\boxdot:M_{k,I}(X^\natural_l)\times M_{I,J}(X)\to M_{k,J}(B).$$
Finally, we have the relations $\forall  b\in M_{k',k}(B),\forall  u\in M_{k,I}(X^\natural_l),\forall  v\in M_{I,J}(X)$:
$$\forall x\in M_{I,J}(A), (bu)[x]=b(u[x]),\ \ \forall a\in M_{J,K}(A), u[x a]=(u[x])[a],\ \ \ b(u\boxdot v)= (bu)\boxdot v,$$
and if $X\in {}_ANDOM_B$ we also have $$\forall  u\in M_{k,I}(X_{l\natural}), \ \ \ \ \forall c\in M_{J,K}( B), u\boxdot (vc)= (u\boxdot v)c,\ \ \forall a\in M_{I,I}( A), (u[a])\boxdot v= u\boxdot (av).$$
\end{lemma} 

\begin{proof}
The first identity comes from lemma \ref{MatricialEHaagerup} and the second from lemma \ref{ModuleCBL}.
This last inclusion follows from \cite[Th 5.1]{EffrosRuanDual} as their Th 5.7 corresponding to $B=\C$ first in the case $B=B(H),$ the general case using strongly  $B$ von Neumann algebra thus closed for the strong operator topology in $B(H)$. 

Note that ${\ell^2(I)}_c\otimes_h X$ is a $[{\ell^2(I)}_c\otimes_h A\otimes_h (\ell^2(I)^*)_r]-B$ h-bimodule  from the map obtained from associativity and the canonical trace on trace class operators $[(\ell^2(I)^*)_r\otimes_h{\ell^2(I)}_c]\to \C$, and the product $[A\otimes_h X\otimes_h B]\to X$: $$[{\ell^2(I)}_c\otimes_h A\otimes_h (\ell^2(I)^*)_r]\otimes_h[{\ell^2(I)}_c\otimes_h X]\otimes_h B\simeq 
{\ell^2(I)}_c\otimes_h[A\otimes_h [(\ell^2(I)^*)_r\otimes_h{\ell^2(I)}_c]\otimes_h X\otimes_h B]\to {\ell^2(I)}_c\otimes_hX$$
and this is of course consistent with inclusions of set $I\subset J.$ From lemma \ref{ModuleCBL}, one deduces $M_{k,I}(X^\natural_l)$ is an $M_k(B)-[{\ell^2(I)}_c\otimes_h A\otimes_h (\ell^2(I)^*)_r]$ h-bimodule (consistent with inclusion in $k,I$. It suffices to extend it to an $M_k(B)-M_{I,I}(A)$ h-bimodule for the canonical inclusion $$[{\ell^2(I)}_c\otimes_h A\otimes_h (\ell^2(I)^*)_r]\subset M_{I,I}(A)\simeq M_I(\C)\overline{\otimes} A\simeq [{\ell^2(I)}_c\otimes_{\sigma h} A\otimes_{\sigma h} (\ell^2(I)^*)_r]$$
the weak-* continuous isomorphism coming from the predual identity \cite[(9.3.3-4-7)]{EffrosRuan}:
$$(M_{I}(\C))_*\widehat{\otimes}A_*\simeq[{\ell^2(I)^*}_r\otimes_h A_*\otimes_h ({\ell^2(I)})_c]\simeq[{\ell^2(I)^*}_r\otimes_{eh} A_*\otimes_{eh} ({\ell^2(I)})_c] .$$

Let us take as $M_k(B)-M_{I,I}(A)$ h-bimodule structure the unique pointwise (on ${\ell^2(I)}_c\otimes_h X$) weak-* (for $M_{i,1}(B)$ and $M_I(A)$) continuous extension on $CB_l({\ell^2(I)}_c\otimes_h X,M_{i,1}(B))_{B})$ i.e. with $a\mapsto u[a](x), x\in {\ell^2(I)}_c\otimes_h X$ weak-* continuous. Uniqueness is obvious, for existence, note that since ${\ell^2(I)}_c\otimes_{eh} X$ is an $M_I(A)-B$ h-bimodule (agreeing by restriction, $CB_l({\ell^2(I)}_c\otimes_{eh} X,M_{i,1}(B))_{B})$ is a $B-M_I(A)$ h bimodule and it is easy to see it is weak-* continuous in the $M_I(A)$ argument. Then 
$$CB_l({\ell^2(I)}_c\otimes_{eh} X,M_{i,1}(B))_{B})\otimes_h M_{I}(A)\to CB_l({\ell^2(I)}_c\otimes_{eh} X,M_{i,1}(B))_{B})\to CB_l({\ell^2(I)}_c\otimes_{h} X,M_{i,1}(B))_{B})$$
is also agreeing and weak-* continuous as expected (the last map being the adjoint of restriction), and also when restricted to $CB_l({\ell^2(I)}_c\otimes_{h} X,M_{i,1}(B))_{B}).$ All the identities for $.[.]$ for the module structure are extended by (separate) weak-* continuity (of product) (we extend in $x\in M_{I,J}(A)$ first then in $a\in M_{J,K}(A)$)

The case $X$ normal dual is easier. $$M_{k,I}(X_{l\natural})=CB_l^\sigma({\ell^2(I)}_c\otimes_hX,M_{i,1}(B))_{B}))=NCB_l({\ell^2(I)}_c\otimes_{\sigma h}X,M_{i,1}(B))_{B}))$$
the first $\sigma$ meaning meaning separately normal in each argument (the $\ell^2$ argument being automatic by reflexivity) and the normal Haagerup identity coming from \cite{EffrosRuanDual}. Since ${\ell^2(I)}_c\otimes_{\sigma h}X={\ell^2(I)}_c\otimes_{e h}X$ is a normal dual $M_I(A)-B$ bimodule, the normal case of lemma \ref{ModuleCBL} applies directly (and the map is of course the restriction of the previous map).

We reason similarly for the second map. One uses the identities noted before $u\in M_{k,I}(CB_{l}(X,B)_{B})\simeq CB_{l}(\ell^2(I)_c\otimes_hX,C_k\otimes_h B)_{B}\subset CB_{l}(\ell^2(I)_c\otimes_{eh}X,C_k\otimes_h B)_{B}$. Moreover we obtain for $v\in M_{I,J}(X)=CB((\ell^2(J))_c,CB((\ell^2(I)^*)_r,X))$ defines an element $I\otimes v\in CB((\ell^2(J))_c,CB((\ell^2(I))_c\otimes_{eh}(\ell^2(I)^*)_r,(\ell^2(I))_c\otimes_{eh}X))$ so that one can compose $u\circ (I\otimes v)\in CB((\ell^2(J))_c,CB((\ell^2(I))_c\otimes_{eh}(\ell^2(I))_r,C_k\otimes_h A_{1}))=CB((\ell^2(J))_c,CB(M_J(\C),C_k\otimes_h A_{1}))).$ Applying this to $Id\in M_J(\C)$ one defines our expected  product $u\boxdot v=[u\circ (I\otimes v)](Id)\in CB((\ell^2(J))_c,C_k\otimes_h A_{1})= M_{k,J}(A_1).$ 

The identities are easy to check for finite matrices  $a,b,c$ and it suffices to note the appropriate continuities to extend them by weak-* density in those von Neumann algebras. Obviously  $u\mapsto u\boxdot v$ is continuous from point weak-* topology convergence to (operator) weak topology convergence, and the point weak-* topology was exactly the target topology of continuity of $a\mapsto u[a]$ as expected. If moreover $u\in M_{k,I}(NCB_{l}(X,B)_{B})= NCB_{l}(\ell^2(I)_c\otimes_{eh}X,C_k\otimes_h B)_{B}$ and note that what we called $v\mapsto (I\otimes v)(Id)$ is nothing but the composition at target with the map $CB((\ell^2(I)^*)_r,X)\to(\ell^2(I))_c\otimes_{eh}X)$  dual of the canonical map  \cite[(9.3.4)]{EffrosRuan} $(\ell^2(I)^*)_r\otimes_{h}X_*\simeq (\ell^2(I)^*)_r\widehat{\otimes} X_*.$
Thus $v\mapsto (u\boxdot v)$ is continuous from the weak-* topology to the operator weak topology since a convergent net in $v$ implies pointwise on $\ell^2(J)$ weak-* convergence of $(I\otimes v)(Id)$ by the result above and by normality of $u$ convergence after application of $u$ weak-* in $M_{i,1}(B)$. The operator weak topology is weaker than this point weak-* $M_{i,1}(B)$ convergence. Since $M_{I,J}(X)$ is also normal dual operator module, one can use weak-* continuity of $c\mapsto vc$, $a\mapsto av$ to extend the identities.
\end{proof}


We are now ready to prove our module duality result and exploit it to obtain a useful density result.

\begin{proposition}\label{ModuleLeftDual} Let $A_i, i=0,...,n$ von Neumann algebras.
Let $X_i$ be normal dual $A_{i}-A_{i+1}$ operator bimodules , then their bimodule left predual $(X_i)_{l\natural}$ are matricially normed $A_{i+1}-A_{i}$ bimodules and, we have a completely contractive canonical map of $A_{n+1}-A_0$ bimodules :
$$(X_n)_{l\natural}\otimes_{h A_n}... \otimes_{h A_1}(X_0)_{l\natural}\to (X_0\otimes_{eh A_1}... \otimes_{eh A_n}X_n)^\natural_l.$$

As a consequence, for $D\subset (M,\tau)$ finite von Neumann algebras ${}_ML^2(M)_{L^2(D)}\subset ({}_DM_D)_r^\natural=({}_DM_{\C})_r^\natural$ completely isometrically and there is a canonical contractive map $$ L^2(M^{op})_r\otimes_{h D^{op}}({}_ML^2(M)_{L^2(D)}^{op})^{\otimes_{h D^{op}}n}\to ({}_{L^2(D)}L^1(M)_{L^2(D)})^{\otimes_{h D'} n}\otimes_{h D'}({}_ML^1(M)_{L^2(D)})$$ with normwise dense range.
\end{proposition}

\begin{proof}

To get the map from the non-module Haagerup tensor product, take $u=u_n\odot ...\odot u_0$ $u_i\in M_{k_i,k_{i-1}}((X_i)_{l\natural})$ ($k_{-1}=1=k_n$, $k_i$ finite). Let $x=x_0\odot ...\odot x_n,$ $x_i\in M_{I_i,I_{i+1}}(X_i)$ ($I_{0}=1,I_{n+1}$ finite, $I_i$ sets) then we define for $a\in A_0$ $u(a)[x]\in M_{1,I_{n+1}}(A_{n+1})$ in noting 
that by using our various maps built above in lemma \ref{matrixMultPredual} : $$u_0[a]\in M_{k_0,I_0}((X_0)_{l\natural}), u_0[a]\boxdot x_0\in M_{k_0,I_1}(A_1), $$ $$u_1[u_0[a]\boxdot x_0]\in M_{k_1,I_1}((X_1)_{l\natural}),u_1[u_0[a]\boxdot x_0]\boxdot x_1\in M_{k_1,I_2}(A_2)...$$
and finally one defines  $$u(a)[x]=u_n[\cdots [u_1[u_0[a]\boxdot x_0]\boxdot x_1]\cdots ]\boxdot x_n\in M_{k_n,I_{n+1}}(A_{n+1})=M_{1,I_{n+1}}(A_{n+1}).$$
From the modularity of the operations it is easy to see it only depends on $x$ in the extended Haagerup product and it is completely contractive in $x$ from \cite{M05} equation (2.6). Then, the modularity in $u$ variables shows $u\mapsto u(.)[.]\in (X_0\otimes_{eh A_1}... \otimes_{eh A_n}X_n)^\natural_l$ induces a map on the quotient which is the module Haagerup tensor product.

We now turn to statement for the inclusion $D\subset M$. We saw before the proof $({}_DM_D)_r^\natural=CB_{rD}(M,D)=({}_DM_{\C})_r^\natural.$
Now, take $(\xi_{ij})\in M_n({}_ML^2(M)_{L^2(D)})$ and consider for $(m_{(k,K)(j,l)})\in M_{k\times L,n\times L}(M)$, then $\sum_{j} E_D(m_{(k,K)(j,l)}\xi_{ji})\in M_{k\times L,n\times L}(D)$ and 
\begin{align*}||\sum_{j} E_D(m_{(k,K)(j,l)}\xi_{ji})||_{M_{k\times L,n\times L}(D)}^2&=||\sum_{j,j',k,K} E_D(\xi_{j'i'}^*m_{(k,K)(j',l')}^*)E_D(m_{(k,K)(j,l)}\xi_{ji})||_{M_{n\times L,n\times L}(D)}\\&\leq ||\sum_{j,j'} E_D(\xi_{j'i'}^*(m^*m)_{(j',l')(j,l)}\xi_{ji})||_{M_{n\times L,n\times L}(D)}
\\&\leq ||\sum_{j} E_D(\xi_{ji'}^*\xi_{ji})||_{M_{n\times L,n\times L}(D)}||m^*m||_{M_{n\times L,n\times L}(M)}\\&=||m||_{M_{k\times L,n\times L}(M)}^2||\xi||_{M_n({}_ML^2(M)_{L^2(D)})}^2\end{align*}
showing that $(\xi_{ij})$ defines a map in $M_n(CB_{rD}(M,D))=CB_{rD}(M_{1,n}(M),M_{1,n}(D))$. Since $[(1+\sum_{j} E_D(\xi_{j.}^*\xi_{j.}))^{-1}]_{lk}\xi_{ik}^*\in M_n(M)$
 if the image of $(\xi_{il})$ is zero, one deduces  $[(1+\sum_{j} E_D(\xi_{j.}^*\xi_{j.}))^{-1}]_{lk}\sum_iE_D(\xi_{ik}^*\xi_{il})=0$ and thus by functional calculus $\sum_iE_D(\xi_{ik}^*\xi_{il})=0$ i.e. $(\xi_{il})=0.$ We thus obtained the expected injection ${}_ML^2(M)_{L^2(D)}\subset ({}_DM_D)_r^\natural$.
 
 Of course $(({}_DM_D)_r^\natural)^{op}$ is a $D^{op}-D^{op}$ module, and we thus deduce from the first part of the proof a map  $$({}_ML^2(M)_{L^2(D)}^{op})^{\otimes_{h D^{op}}(n+1)}\subset [({}_DM_D)_r^\natural)^{op}]^{\otimes_{h D^{op}}n}\otimes_{h D^{op}}[({}_DM_{\C})_r^\natural)^{op}]\to ([{}_D(M^{\otimes_{eh D}n}\otimes_{eh D}M)_{\C}]_r^\natural)^{op}.$$
 
 Now, by \cite{B97} $L^2(M^{op})_r\otimes_{h D^{op}}({}_ML^2(M)_{L^2(D)}^{op})^{\otimes_{h D^{op}}n}=L^2(D^{op})_r\otimes_{h D^{op}}({}_ML^2(M)_{L^2(D)}^{op})^{\otimes_{h D^{op}}(n+1)}$
 
 and we have a canonical map from (a CB/modular variant of) lemma \ref{PilHaagerup} :\begin{align*}[{}_D(M^{\otimes_{eh D}n}\otimes_{eh D}M)_{\C}]_r^\natural\otimes_{hD} L^2(D)_c&=CB_{rD}(M^{\otimes_{eh D}n}\otimes_{eh D}M,D)\otimes_{hD} L^2(D)_c
 \\&\to CB_{rD}(M^{\otimes_{eh D}n}\otimes_{eh D}M,D\otimes_{hD} L^2(D)_c)\\&=CB_{rD}({}_D(M^{\otimes_{eh D(n+1)}})_{\C}, L^2(D)_c)\\&=({}_D(M^{\otimes_{eh D(n+1)}})_{\C})^\natural\end{align*}
 the last equality being only isometric and not competely isometric and comes from the definition in \cite{M05}. Now from the definition of duality there, the map from $ ({}_ML^2(M)_{L^2(D)}^{op})^{\otimes_{h D^{op}}(n+1)}  \to({}_{L^2(D)}L^1(M)_{L^2(D)})^{\otimes_{h D'} n}\otimes_{h D'}({}_ML^1(M)_{L^2(D)})
 \subset ({}_D(M^{\otimes_{eh D(n+1)}})_{\C})^\natural$\footnote{the last inclusion being isometric since ${}_D(M^{\otimes_{eh D(n+1)}})_{\C}\subset ({}_{L^2(D)}L^1(M)_{L^2(D)})^{\otimes_{h D'} n}\otimes_{h D'}({}_ML^1(M)_{L^2(D)}))^\natural$ completely isometrically, cf the next proposition  from \cite{M05}, so that \begin{align*}({}_{L^2(D)}L^1(M)_{L^2(D)})^{\otimes_{h D'} n}\otimes_{h D'}({}_ML^1(M)_{L^2(D)}))&\subset [{}_{D'}({}_{L^2(D)}L^1(M)_{L^2(D)})^{\otimes_{h D'} n}\otimes_{h D'}({}_ML^1(M)_{L^2(D)})_{\C}]^{\natural\natural}\\&\subset ({}_D(M^{\otimes_{eh D(n+1)}})_{\C})^\natural\end{align*}} can be identified with the inclusion in $ L^2(D^{op})_r\otimes_{h D^{op}}({}_ML^2(M)_{L^2(D)}^{op})^{\otimes_{h D^{op}}(n+1)}$ followed by the opposite map to $[{}_D(M^{\otimes_{eh D}n}\otimes_{eh D}M)_{\C}]_r^\natural\otimes_{hD} L^2(D)_c$ followed by the map above. Since the first inclusion is dense, one deduces our contraction by extension
$$ L^2(D^{op})_r\otimes_{h D^{op}}({}_ML^2(M)_{L^2(D)}^{op})^{\otimes_{h D^{op}}(n+1)}\to ({}_{L^2(D)}L^1(M)_{L^2(D)})^{\otimes_{h D'} n}\otimes_{h D'}({}_ML^1(M)_{L^2(D)}).$$

It remains to show it has dense range inductively on $n$ equivalently when replacing $L^2(D^{op})$ by $D^{op}.$ The case $n=0$ is obvious since $M$ normwise dense in $L^2(M)$ and by the description of $({}_ML^1(M)_{L^2(D)})$ in the last proposition. Assume recurrence hypothesis, take $\xi\otimes \eta, \xi \in M_{1,k}({}_{L^2(D)}L^1(M)_{L^2(D))},\eta \in M_{k,1}({}_{L^2(D)}L^1(M)_{L^2(D)})^{\otimes_{h D'} n-1}\otimes_{h D'}({}_ML^1(M)_{L^2(D)})$ in the dense subspace of elementary tensors and  get $\eta'\in ({}_ML^2(M)_{L^2(D)}^{op})^{\otimes_{h D^{op}}n}$ with $||\eta'-\eta||\leq \epsilon/2$. One can assume $||\xi||\leq 1.$
 take a bounded net $M_{1,k}({}_ML^2(M)_{L^2(D)}^{op})\ni\xi_n\to \xi$ converging in the sens of the last proposition and let us show that for $n$ large enough $||\xi_n\otimes \eta'-\xi\otimes\eta||_{({}_D(M^{\otimes_{eh D(n+1)}})_{\C})^\natural}\leq \epsilon$ which will be enough by the isometric embedding above. Of course it suffices $||(\xi_n-\xi)\otimes \eta'||_{({}_D(M^{\otimes_{eh D(n+1)}})_{\C})^\natural}\leq \epsilon/2.$

But we have by definition the inequality $$||(\xi_n-\xi)\otimes \eta'||_{({}_D(M^{\otimes_{eh D(n+1)}})_{\C})^\natural}\leq \sup_{m_1\in (M_{1,I}(M))_1, m_2\in (M_{I,1}(M^{\otimes_{eh D}n}))_1}\sum_j ||E_D[m_1E_D(m_2\eta'_{j1})(\xi_n-\xi)_{1,j}]||_2.$$
This indeed tends to 0 since $E_D(m_2\eta'_{j1})\in M_{I,1}(D)$ thus $m_1E_D(m_2\eta'_{j1})\in M$ and the type of convergence of the previous proposition is thus enough.
\end{proof}

\section{Generalized weak-$*$ Haagerup products}

\subsection{Supplementary results on Haagerup tensor products of $D$-modules.}\label{HaaModuleSup}

\begin{proposition}\label{Submodules}Let $D\subset M$ finite von Neumann algebras.
   We have completely isometric isomorphisms and embeddings (for $k\geq 0$):\begin{align*}({}_MM_{D})\otimes_{eh D}&({}_DM_{D})^{\otimes_{eh D}k}\otimes_{eh D}({}_DM_{M})\\&\simeq [({}_{L^2(D)}L^2(M)_M)\otimes_{h D'}({}_{L^2(D)}L^1(M)_{L^2(D)})^{\otimes_{h D'} k}\otimes_{h D'}({}_ML^2(M)_{L^2(D)})]^{\natural D'norm}\\&\subset [({}_{L^2(D)}L^2(M)_M)\otimes_{h D'}({}_{L^2(D)}L^1(M)_{L^2(D)})^{\otimes_{h D'} k}\otimes_{h D'}({}_ML^2(M)_{L^2(D)})]^{\natural}\\&=({}_MM_{D})\otimes_{w^*h D}({}_DM_{D})^{\otimes_{w^*h D}k}\otimes_{w^*h D}({}_DM_{M}),\end{align*}
 so that (with an ordinary duality compatible with $M-M$ module structure) 
  \begin{align*}({}_MM_{D})\otimes_{eh D}&({}_DM_{D})^{\otimes_{eh D}k}\otimes_{eh D}({}_DM_{M})=:M^{\otimes_{eh D}k+2}\\&=[({}_{L^2(D)}L^1(M)_M)\otimes_{h D'}({}_{L^2(D)}L^1(M)_{L^2(D)})^{\otimes_{h D'} k}\otimes_{h D'}({}_ML^1(M)_{L^2(D)})]^{*D'norm}\\&\subset[({}_{L^2(D)}L^1(M)_M)\otimes_{h D'}({}_{L^2(D)}L^1(M)_{L^2(D)})^{\otimes_{h D'} k}\otimes_{h D'}({}_ML^1(M)_{L^2(D)})]^*\\&=({}_MM_{D})\otimes_{w^*h D}({}_DM_{D})^{\otimes_{w^*h D}k}\otimes_{w^*h D}({}_DM_{M})=:M^{\otimes_{w^*h D}k+2},\end{align*} such that for instance for $m_1\otimes_D ...\otimes_D m_{k+2}\in M^{\otimes_{w^*h D}k+2},\xi_1\otimes_{D'} ...\otimes_{D'} \xi_{k+2}\in  {}_{L^2(D)}L^1(M)_M)\otimes_{h D'}({}_{L^2(D)}L^1(M)_{L^2(D)})^{\otimes_{h D'} k}\otimes_{h D'}({}_ML^1(M)_{L^2(D)})=:L^1(M)^{\otimes_{h D'} k+2}$:
\begin{align*}\langle  &\xi_1\otimes_{D'} ...\otimes_{D'} \xi_{k+2},m_1\otimes_D ...\otimes_D m_{k+2}\rangle\\&=\tau(\xi_1m_1E_D(m_2E_D(m_3...E_D(m_{k+1}E_D(m_{k+2}\xi_{k+2})\xi_{k+1})....\xi_3)\xi_2)).
 \end{align*}
 
   The inclusion $M^{\otimes_{eh D}k+2}\subset M^{\otimes_{w^*h D}k+2}$ is a weak-* closed subspace so that both spaces are dual operator spaces.
   
Moreover, there is a completely contractive weak-* continuous $M-M$ bimodule embeddings :$M^{\otimes_{w^*h D}k+2}\hookrightarrow (L^2(M)^{\otimes_{D}k+2})^*$, 
such that moreover the projection $E_{D'}$ on $L^2(M)^{\otimes_{D}k+2}$ restricts to contractive maps  $M^{\otimes_{w^*h D}k+2}\to D'\cap M^{\otimes_{w^*h D}k+2}\simeq (L^1(M)^{\otimes_{h D'} k+2}/\overline{[D,L^1(M)^{\otimes_{h D'} k+2}]})^*,$ weak-* continuous on bounded sets and from  $M^{\otimes_{eh D}k+2}\to D'\cap M^{\otimes_{eh D}k+2}.$

Finally, for any $x\in L^2(M),y\in M$, the map $T_{x,y}: M^{\otimes_{eh D}2}\to (L^2(M)^*\otimes_{h D} L^2(M))$ induced from $T:L^2(M)^*\otimes_{eh}( M^{\otimes_{ehh D}2})\otimes_{eh}L^2(M)\to (L^2(M)^*\otimes_{h D} L^2(M))$ (by $T_{x,y}(U)=T(x\otimes U\otimes y)$) is completely contractive weak-* continuous and an embedding when $x=y=1.$ 

 \end{proposition}
 \begin{proof} 
  The two first statements are consequences of general results stated above and the formula for the duality couplings follows easily from our previous computations. 
  The statement that the extended Haagerup product is a dual operator space is standard from the weak* closedness statement (cf. e.g. \cite[lemmas 1.4.6, 1.4.7]{BLM}).
  
The fact that the extended Haagerup tensor product is closed in the weak-* Haagerup tensor product follows from the statement that for a bounded (say by 1) convex sets $C$ in the extended space, its closure for the weak-* topology inside the  weak-* product is actually its closure for the normal weak-* topology (and thus is in the extended product). For, let a net $(\xi_i)_{i\in I}\in C\subset ({}_MM_{D})\otimes_{eh D}({}_DM_{D})^{\otimes_{eh D}k}\otimes_{eh D}({}_DM_{M})$ converges to $\xi$  (in the weak-* product) for the weak-* topology. Define $J=]0,1[\times L\times L'$ where $L$ is the set of finite sets $\eta=\{\eta_{1},...,\eta_{n}\}$ of any length such that $\eta_l=(\eta_{l,a}, \eta_{l,b})$, $\eta_{l,a}\in ({}_{L^2(D)}L^1(M)_M)\otimes_{h D'}({}_{L^2(D)}L^1(M)_{L^2(D)})^{\otimes_{h D'} m}, 0\leq m=m(l)\leq k$  
$\eta_{l,b}\in ({}_{L^2(D)}L^1(M)_{L^2(D)})^{\otimes_{h D'} k-m}\otimes_{h D'}({}_ML^1(M)_{L^2(D)})$ 
with $||\eta_{i,j}||\leq 1$ in their respective norms and $L'$ is the set of $\mu=\{\mu_{1},...,\mu_{n}\}, \mu_i\in ({}_{L^2(D)}L^1(M)_M)\otimes_{h D'}({}_{L^2(D)}L^1(M)_{L^2(D)})^{\otimes_{h D'} k}\otimes_{h D'}({}_ML^1(M)_{L^2(D)}).$ We will order $J$ by product order (inverse order in $]0,\epsilon[$, and inclusion for $L,M$) and build a net indexed by $J$ converging to $\xi$ in the weak-* topology and moreover converging in the normal  weak-* topology and  still in $C$. Thus fix $(\epsilon,\eta,\mu)\in J.$

{We know $({}_ML^2(M)_{L^2(D)}^{op})^{\otimes_{h D'} k-m(l)+1}$ is normwise dense in $({}_{L^2(D)}L^1(M)_{L^2(D)})^{\otimes_{h D'} k-m(l)}\otimes_{h D'}({}_ML^1(M)_{L^2(D)})$ by the previous lemma \ref{ModuleLeftDual}. Thus let $\eta_{l,c}\in ({}_ML^2(M)_{L^2(D)}^{op})^{\otimes_{h D'} k-m(l)+1}$ such that $$||\eta_{l,c}-\eta_{l,b}||\leq \frac{\epsilon}{4\min(1,||\eta_{l,a}||)}.$$

One  also gets from the same lemma : $$||d\eta_{l,c}||_{({}_{L^2(D)}L^1(M)_{L^2(D)})^{\otimes_{h D'} k-m(l)}\otimes_{h D'}({}_ML^1(M)_{L^2(D)})}\leq ||d||_2||\eta_{l,c}||_{({}_ML^2(M)_{L^2(D)}^{op})^{\otimes_{h D'} k-m(l)+1}}.$$

Thus $\langle \Psi(\eta_{l,a},\eta_{l,c})( m_1\otimes_D ...\otimes_D m_{k+2}), d\rangle=\langle  \eta_{l,a}\otimes_{D'} d\eta_{l,c},m_1\otimes_D ...\otimes_D m_{k+2}\rangle$ defines $\Psi(\eta_{l,a},\eta_{l,c})( m_1\otimes_D ...\otimes_D m_{k+2})\in L^2(D).$

But $\Psi(\eta_{l,a},\eta_{l,c})(\xi_i)$ weakly converges in $L^2(D)$ for all $l$. Indeed, we saw it is bounded in $L^2(D)$, and since $D\subset L^2(D)$ is normwise dense, we only have to see weak convergence against elements in $D$, which is part of the assumed weak convergence of $\xi_i$.  Thus  if we write its limit with a slight abuse  of notation $(\Psi(\eta_{l,a},\eta_{l,c})(\xi))_{l\in 1,...,n}\in\overline{((\Psi(\eta_{1,a},\eta_{1,c}),...,\Psi(\eta_{n,a},\eta_{n,c}))(C)}^{||.||_2}$ by a standard consequence of Hahn-Banach theorem. 
One can take $\Xi_{\epsilon,\eta,\mu}\in C$ such that for all $l\in 1,...,n$ $||\Psi(\eta_{l,a},\eta_{l,c})(\Xi_{\epsilon,\eta,\mu}-\xi)||_2\leq \frac{\epsilon}{2}$ and thus 
$||(\Psi(\eta_{l,a},\eta_{l,c})(\Xi_{\epsilon,\eta,\mu}-\xi)||_1\leq \epsilon.$ We can also get $|\langle \mu_i ,\Xi_{\epsilon,\eta}-\xi\rangle|\leq \epsilon.$ Thus $\Xi_{\epsilon,\eta,\mu}$ is a net of $C$ converging to $\xi$ for the weak-* topology and converging for the normal weak-* topology, as wanted.}

For the last embedding (say for $k=0$, the general case is similar using the last density result of the previous proposition), it suffices to get an essentially surjective completely contrative predual map :
$$(L^2(M)\otimes_{D}L^2(M))\to ({}_{L^2(D)}L^1(M)_M)\otimes_{h D'}({}_ML^1(M)_{L^2(D)}).$$ 
Let us see ${}_{\C}({}_{L^2(D)}L^2(M)_M)_{D'}= {}_{M'}L^2(M')_{L^2(D')}$ as a (right) $\C-D'$ correspondence as in the sense of \cite[section 3]{B97b} and $L^2(M)$ $D'-\C$ correspondence. Thus by \cite[Theorem 3.1]{B97b} (or even \cite[Theorem 4.3]{B97} since we don't use the weak-* topology)
$L^2(M)\otimes_{D}L^2(M)\simeq ({}_{L^2(D)}L^2(M)_M)\otimes_{eh D'}L^2(M)\simeq ({}_{L^2(D)}L^2(M)_M)\otimes_{h D'}L^2(M)$  completely isometrically (the last isomorphism is explained in  \cite{M05}, here recall that $L^2(M)\otimes_{D}L^2(M)$ and $L^2(M)$ have their column Hilbert space structures). 

Now we have essentially surjective (canonical) complete contractions $$\pi_1: {}_{\C}({}_{L^2(D)}L^2(M)_M)_{D'}\simeq M'\otimes_{h M'}({}_{L^2(D)}L^2(M)_M\to {}_{L^2(D)}L^1(M)_M\simeq (L^2(M))^*\otimes_{h M'}({}_{L^2(D)}L^2(M)_M)$$ and $$\pi_2: L^2(M)\simeq M'\otimes_{h M'}L^2(M)\to {}_ML^1(M)_{L^2(D)}\simeq ({}_ML^2(M)_{L^2(D)})\otimes_{h M'}L^2(M)$$
 using the complete isometries given in (1) for compatibility of the operator space structures. The reader should note that $\pi_1$ is normwise essentially surjective, while $\pi_2$ is only weakly essentially surjective with a proof similar to the last density statement in proposition \ref{ModulePreDual}. 
  Obviously, ${\pi_1}\otimes_{h D'}\pi_2$ is the map we wanted which is weakly essentially surjective. The case with more than two tensors is similar using also the last density statement in proposition \ref{ModulePreDual}. 

For the conditional expectations, note that on $L^2$, it is well known that $E_{D'}(x)$ is a limit of a net of convex combinations $\xi_i=\sum \lambda_uuxu^*, u\in\mathcal{U}(D).$ Since when $x$ is in the weak-* Haagerup tensor product; such sum has also bounded norm, it has a weak-* convergent limit point which has to coincide with $E_{D'}(x)$. The statement for extended Haagerup tensor product then follows from the normal weak-* closability statement of weak-* closed bounded sets. 
The weak-* continuity follows by duality and density in preduals.
 $E_{D'}$ induces the isomorphism between the commutant and the quotient. The identification of the predual of the commutant is obvious from the bimodule structures.
 
It remains to prove the final statement. For, we have to produce $T_{x,y}$ as the adjoint of  a map : $$S_{x,y}: [({}_{L^2(D)}L^2(M)_M)\otimes_{h D^{op}}(({}_ML^2(M)_{L^2(D)})]\to (({}_{L^2(D)}L^1(M)_M)\otimes_{h D^{op}}({}_ML^1(M)_{L^2(D)})).$$
We define $S_{x,y}(\xi\otimes_{h D^{op}}\xi')=x\otimes_{M'}\xi\otimes_{h D^{op}}\xi'\otimes_{M'}y\in [(L^2(M))^*\otimes_{h M'}({}_{L^2(D)}L^2(M)_M)]\otimes_{h D^{op}}[({}_ML^2(M)_{L^2(D)})\otimes_{h M'}L^2(M)]$  with the isomorphism of Proposition \ref{ModulePreDual}, or said otherwise $$S_{x,y}(\xi\otimes_{h D^{op}}\xi')=(\xi x\otimes_{h D^{op}}y\xi').$$
Now, we have the claimed adjoitness realtion :$$\langle S_{x,y}(\xi\otimes_{h D^{op}}\xi'),\eta\otimes_{h D}\eta'\rangle=\tau(E_D(\xi x\eta)\eta'y\xi')=\langle \xi\otimes_{h D^{op}}\xi',x\eta\otimes_{h D}\eta'y\rangle=\langle \xi\otimes_{h D^{op}}\xi',T_{x,y}(\eta\otimes_{h D}\eta')\rangle,$$
thus $T_{x,y}=S_{x,y}^*|_{M^{\otimes_{eh D}2}} $ is indeed completely contractive and weak-* continuous.
\end{proof}

We can now finish improving our last result to get our first main theorem.
\begin{theorem}\label{MainHaagerupModule}
Let $D\subset M$ finite von Neumann algebras.
   We have a completely isometric weak-* homeomorpic isomorphism
 (for $k\geq 0$):\begin{align*}({}_MM_{D})\otimes_{eh D}&({}_DM_{D})^{\otimes_{eh D}k}\otimes_{eh D}({}_DM_{M})\simeq ({}_MM_{D})\otimes_{w^*h D}({}_DM_{D})^{\otimes_{w^*h D}k}\otimes_{w^*h D}({}_DM_{M}).\end{align*}
\end{theorem}
\begin{proof}
First, note that (following the argument in \cite[Th 4.2]{M05} by injectivity of Haagerup tensor product and \begin{align*}{}_{L^2(D)}(L^{1}(M))_{M}&\otimes_{h D'}{}_{M}(L^{1}(M))_{L^2(D)}\\&\simeq {}_{L^2(D)}(L^{1}(M))_{M}\otimes_{h D'}D'\otimes_{h D'}{}_{M}(L^{1}(M))_{L^2(D)}\\&\subset {}_{L^2(D)}(L^{1}(M))_{M}\otimes_{h D'}B(L^2(D))\otimes_{h D'}{}_{M}(L^{1}(M))_{L^2(D)}\\&\simeq {}_{L^2(D)}(L^{1}(M))_{M}\otimes_{h D'}[L^2(D)\otimes_{eh}(L^2(D))^*]\otimes_{h D'}{}_{M}(L^{1}(M))_{L^2(D)}
\\&\subset {}_{L^2(D)}(L^{1}(M))_{M}\otimes_{eh D'}[L^2(D)\otimes_{eh}(L^2(D))^*]\otimes_{eh D'}{}_{M}(L^{1}(M))_{L^2(D)}\\&\simeq [{}_{L^2(D)}(L^{1}(M))_{M}\otimes_{eh D'}L^2(D)]\otimes_{eh}[(L^2(D))^*\otimes_{eh D'}{}_{M}(L^{1}(M))_{L^2(D)}\\&\simeq (L^{1}(M))_{M}\otimes_{eh}[{}_{M}(L^{1}(M))]\end{align*}
Thus the dual map gives a weak-* continuous complete quotient map $M\otimes_{\sigma h }M\to M\otimes_{w^*h D}M$ (this map is also the composition of known maps $M\otimes_{\sigma h }M\to M\otimes_{\sigma h  D}M\to M\otimes_{w^*h D}M$). Similarly, one gets weak-* continuous complete quotient map :$$ M^{\otimes_{\sigma h }(k+2)}\to ({}_MM_{D})\otimes_{w^*h D}({}_DM_{D})^{\otimes_{w^*h D}k}\otimes_{w^*h D}({}_DM_{M}).$$
But from \cite[lemma 5.8]{EffrosRuanDual}, the algebraic tensor product is weak-* dense in $M^{\otimes_{\sigma h }(k+2)}$ and obviously mapped to $({}_MM_{D})\otimes_{eh D}({}_DM_{D})^{\otimes_{eh D}k}\otimes_{eh D}({}_DM_{M})$, thus this space is also weak-* dense in $({}_MM_{D})\otimes_{w^*h D}({}_DM_{D})^{\otimes_{w^*h D}k}\otimes_{w^*h D}({}_DM_{M}).$ But we saw in the last proposition it is also weak-* closed, thus they are equal.
\end{proof}

\subsection{Haagerup tensor products relative to covariance maps}

We start by observing that $M\otimes_{w^*h D}M$ is weak-* continuously completely isometric  to a quotient of $M\otimes_{\sigma h }M$ as in the proof of theorem \ref{MainHaagerupModule}. 

And thus $M\otimes_{\sigma h }M\to M\otimes_{w^*h D}M\subset (L^2(M)\otimes_DL^2(M))^*$ and the adjoint map actually factorizes (with value in the predual) $$J:L^2(M)\otimes_DL^2(M)\to {}_{L^2(D)}(L^{1}(M))_{M}\otimes_{h D'}[{}_{M}(L^{1}(M))_{L^2(D)}]\subset (L^{1}(M))_{M}\otimes_{eh}[{}_{M}(L^{1}(M))].$$

We even have the following explicit description when we  see $(L^{1}(M))_{M}\otimes_{eh}[{}_{M}(L^{1}(M))]\subset(M\otimes_h M)^*$ that will be useful to compare with our general theory :

\begin{lemma}\label{IdentifED}
For tensors in the algebraic tensor product, $y\otimes x\in L^2(M)\otimes_DL^2(M) m\otimes n\in M\otimes_h M$ the duality pairing is given by $$\langle J(y\otimes x), m\otimes n\rangle= \tau(xm E_D(ny)).$$
\end{lemma} 
\begin{proof}
If $(e_i)$ is an orthonormal basis of $L^2(D)$ so that $y\otimes x$ is represented by $x\otimes 1_D\otimes y\in {}_{L^2(D)}(L^{1}(M))_{M}\otimes_{h D'}B(L^2(D))\otimes_{h D'}{}_{M}(L^{1}(M))_{L^2(D)}$ in the inclusion above and then, decomposing $1_{D'}=1=\sum_i e_i\otimes e_i^*\in [L^2(D)\otimes_{eh}(L^2(D))^*]$, by $ \sum_i e_ix\otimes ye_i^*.$
Thus by definition the pairing is given by \begin{align*}\langle J(y\otimes_D x), m\otimes n\rangle&= \sum_i  \langle e_ix\otimes ye_i^*,m\otimes n\rangle=\sum_i\tau(e_ixm)\tau( ye_i^*n)\\&=\sum_i\tau(e_iE_D(xm))\tau( E_D(ny)e_i^*)=\tau(xm E_D(ny)).\end{align*}
\end{proof}

Replacing $L^2(M,E_D)\simeq L^2(M)\otimes_DL^2(M)$ by $L^2(M,\eta\circ E_D)$ we will now generalize this situation to general covariance maps.

Recall that $\eta:D\to B(\ell^2(I))\otimes D$ is a normal $\tau$-symmetric completely positive map with $\sup_{i\in I}||\eta_i(1)||\leq C$. We let $(S_i)_{i\in I}$ a semicircular system of covariance $\eta$ over $D$. We define various "direct sums" of Haagerup tensor products to take care of our several variable setting. For $X,Y$ matricially normed spaces we define :

$$X\otimes_h^{(I,1)} Y=X\otimes_h \ell^{1}(I)\otimes_h Y, \ \ X\otimes_h^{(I,\infty)} Y=X\otimes_h c_0(I)\otimes_h Y$$

with $\ell^{1}(I)$ given its standard operator space structure as dual of $c_0(I),$ predual of $\ell^{\infty}(I)$ both being given their standard operator space structure as $C^*$-algebras.
Likewise, if $X,Y$ are operator spaces (respectively for $\sigma h$ products dual operator spaces), we define :

$$X\otimes_{eh}^{(I,1)} Y=X\otimes_{eh} \ell^{1}(I)\otimes_{eh} Y, \ \ X\otimes_{eh}^{(I,\infty)} Y=X\otimes_{eh} \ell^{\infty}(I)\otimes_{eh} Y,$$

$$X\otimes_{\sigma h}^{(I,1)} Y=X\otimes_{\sigma h} (\ell^{\infty}(I))^*\otimes_{\sigma h} Y, \ \ X\otimes_{\sigma h}^{(I,\infty)} Y=X\otimes_{\sigma h} \ell^{\infty}(I)\otimes_{\sigma h} Y.$$

We have the following result following from well-known properties of Haagerup tensor products : 

\begin{lemma}\label{IHaagerup}
For $X,Y$ matricially normed spaces, one has completely contractive maps~: $$(X^*\otimes_{eh}^{(I,\infty)} Y^*)\to (X\otimes_h^{(I,1)} Y)^*, \ \ (X^*\otimes_{eh}^{(I,1)} Y^*)\to (X\otimes_h^{(I,\infty)} Y)^*.$$
If moreover, $X,Y$ are operator spaces the maps above are complete isomorphisms and we have completely isometric maps :
$$(X^*\otimes_{\sigma h}^{(I,\infty)} Y^*)\simeq (X\otimes_{eh}^{(I,1)} Y)^*, \ \ (X^*\otimes_{\sigma h}^{(I,1)} Y^*)\simeq (X\otimes_{eh}^{(I,\infty)} Y)^*.$$
\end{lemma}
\begin{proof}The first operator space duality result is for instance \cite[Th 5.3]{EffrosRuanDual} and since, for instance $(X^*\otimes_{eh}^{(I,\infty)} Y^*)=CB_m^\sigma((X^*)^*\times (\ell^\infty(I))^*\times (Y^*)^*,\C)$ composing with completely bounded maps to standard operator space biduals $X\to (X^*)^*$, $Y\to (Y^*)^*, \ell^1(I)\to(\ell^\infty(I))^*$ one gets the result in the matricially normed space case using \cite{BlecherPaulsen}. The second duality result is mostly the definition \cite[p 148]{EffrosRuanDual} and  duality results of sequence spaces.\end{proof}

Note that one gets weak-* continuous completely bounded maps $$.\#S:M\otimes_{\sigma h}^{(I,1)} M\to W^*(M,S_i,i\in I)=N_\eta\subset L^2(N_\eta)_c\cap L^2(N_\eta)_r$$ this map extending $(m_i)\mapsto \sum_im_i\#S_i$ defined on Haagerup tensor product.

More precisely, since $S=(S_i)_{i\in I}\in \oplus_{i\in I}N_\eta=NCB((\ell^{\infty}(I))^*,N_\eta)$, one defines $x\#S=m[(1\otimes S\otimes 1)(x)]$ when $x\in M\otimes_{\sigma h}^{(I,1)} M=M\otimes_{\sigma h}(\ell^{\infty}(I))^*\otimes_{\sigma h}M$, $1\otimes S\otimes 1$ is defined  on this space by \cite{EffrosRuanDual} so that $(1\otimes S\otimes 1)(x)\in M\otimes_{\sigma h}N_\eta\otimes_{\sigma h}M$ is weak-* continuous in $x$ and $m:M\otimes_{\sigma h}N_\eta\otimes_{\sigma h}M\to N_\eta$ is the canonical weak-* continuous multiplication map (see beginning of subsection \ref{etavar} for more details).

Of course the image of this map is a set dense into an isomorphic copy of $(L^2(M,\eta\circ E_D)$ we will thus see as a closed subspace of $L^2(N_\eta).$

Let us examine the dual map in using our section \ref{columndual} and prove :

\begin{proposition}\label{InclusionEHCovariance}The dual map of $.\#S:M\otimes_{\sigma h}^{(I,1)} M\to (L^2(M,\eta\circ E_D)_c\cap (L^2(M,\eta\circ E_D)_r\subset L^2(N_\eta)_c\cap L^2(N_\eta)_r$  when restricted to $M\otimes_{ h}^{(I,1)} M$ gives a predual map $(.\#S)_*:(L^2(M,\eta\circ E_D)_c+ L^2(M,\eta\circ E_D)_r)\to L^1(M)\otimes_{eh}^{(I,\infty)}L^1(M)\subset ((M\otimes_{h} M)^I)^*$ for the original $.\#S.$
\end{proposition}
\begin{proof}
Since we checked the first map is weak-* continuous, it has a predual map, and, moreover, the former is from \cite[Prop 5.9]{EffrosRuanDual} the unique weak-* continuous extension of the restriction. The dual map of the restriction, by reflexivity of the target, has a target space the space of separately continuous maps, i.e. as written, the predual and thus necessarily agrees with the predual map as in the proof of their proposition 5.9. 
\end{proof}

We could have taken $D=M$ from the beginning in starting from a covariance map on $M$ instead of $\eta\circ E_D.$
\begin{definition}
We define the \textit{weak-* Haagerup tensor product relative to the covariance map $\eta$ } on $M$ by the operator space dual :$$M\otimes_{w^*h\eta}M= (\overline{(.\#S)_*[(L^2(M,\eta)_c+ L^2(M,\eta)_r)]}^{ L^1(M)\otimes_{eh}^{(I,\infty)}L^1(M)})^*$$
\end{definition}

By definition and the last proposition there is a  weak-* continuous completely quotient map (also called complete metric surjection, see e.g. \cite[Prop 2.4.2]{PisierBook}) from the $\sigma$ haagerup product and a weak-* continuous completely contractive injection to a Hilbert space (with intersection of row and column structure) which is even weak-* homeomorphic on bounded sets~: $$p_\eta:M\otimes_{\sigma h}^{(I,1)} M\to M\otimes_{w^*h\eta}M,\ \ i_\eta:M\otimes_{w^*h\eta}M\subset L^2(M,\eta)_c\cap (L^2(M,\eta)_r.$$

Note that $Ker(p_\eta)=Ker(.\#S).$

\begin{definition}
We define the \textit{extended Haagerup tensor product relative to the covariance map $\eta$} on $M$ by the image : $M\otimes_{eh\eta}M:=p_\eta(M\otimes_{eh}^{(I,1)} M)\subset M\otimes_{w^*h\eta}M.$
\end{definition}
The above definition uses the canonical injection $M\otimes_{eh}^{(I,1)} M\subset M\otimes_{\sigma h}^{(I,1)} M$ coming from \cite{EffrosRuandual}

\begin{exemple}\label{CondExp}
If $|I|=1$ and $\eta=id_D$, $(L^2(M,\eta\circ E_D))_r=(L^2(M)\otimes_DL^2(M))_r\simeq L^2(M)_r\otimes_{eh D^{op}}{}_M L^2(M)_{L^2(D)}^{op}=L^2(M)_r\otimes_{h D^{op}}{}_M L^2(M)_{L^2(D)}^{op}$  (by \cite{B97b} where the opposite identification is given $(L^2(M)\otimes_DL^2(M))_c\simeq {}_ML^2(M)_{L^2(D)}\otimes_{eh D}L^2(M)_c$ and using also e.g. \cite[Rmk 2.18]{M05} for the last equality)  is the usual module tensor product and $(.\#S)_*$ coincides with the restriction of the map explained at the beginning of this subsection ${}_{L^2(D)}(L^{1}(M))_{M}\otimes_{h D'}{}_{M}(L^{1}(M))_{L^2(D)}\subset (L^{1}(M))_{M}\otimes_{eh}[{}_{M}(L^{1}(M))],$ as we can see from it concrete inclusion in the dual of $M\otimes_hM$ in lemma \ref{IdentifED}. The closure considered in the definition of the weak-* haagerup tensor product is thus merely (using the density result in lemma \ref{ModuleLeftDual}) ${}_{L^2(D)}(L^{1}(M))_{M}\otimes_{h D'}{}_{M}(L^{1}(M))_{L^2(D)}.$ Its operator space dual is thus  what we defined as $M\otimes_{w^*h D}M$ in section \ref{HaagD} and it agrees with our new $M\otimes_{w^*h (id_D)}M.$ Since from \cite{M05}, one sees any element in $M\otimes_{eh D}M$ can be seen as an image of an element in $M\otimes_{eh} M$ and conversely, one also gets $M\otimes_{eh D}M\simeq M\otimes_{eh (id_D)}M.$
\end{exemple}

\begin{exemple}
If $|I|=1$ and $\eta=E_B$, the trace preserving conditional expectation to $B\subset D$ it is easy to see by construction (using $E_B$ semicircular systems agree with $id_B$ ones over $B$) that  $M\otimes_{w^*h (E_B)}M=M\otimes_{w^*h (id_B)}M\simeq M\otimes_{w^*h B}M,M\otimes_{eh (E_B)}M=M\otimes_{eh (id_B)}M\simeq M\otimes_{eh B}M.$ (The last isomorphisms from our previous example).
\end{exemple}

Another interesting examples comes from the relation of free probability to cross-products  explained in \cite[section 7.3]{S99}.
\begin{exemple}
If $|I|=2$ and $\alpha:D\to D$ is a trace preserving $*$-automorphism, we define $\eta_{11}(a)=\alpha(a)+\alpha^{-1}(a)=\eta_{22}(a), \eta_{12}(a)=-i(\alpha(a)-\alpha^{-1}(a))=-\eta_{21}(a)$, so that $\eta:A\to M_2(A)$ is a covariance map. It is shown in \cite[section 7.3]{S99} that for $(S_1,S_2)$ $\eta$-semicircular system, $C=S_1+iS_2$ then for all $d\in D,$ $dC=C\alpha(d)$.  This suggests a description of our previous tensor products we now explain.  
\end{exemple}

We briefly generalize our tensor products to $n$-ary tensor products. 

Thus for $j=1,...,n$, let $\eta_j:M\to B(\ell^2(I_j))\otimes M$ be  normal $\tau$-symmetric completely positive map with $\sup_{i\in I_j}||\eta_i(1)||\leq C_j$. We let $(S_i^{(j)})_{i\in I_j}$ a semicircular system of covariance $\eta_j$ over $M$. We sometimes call $\eta=(\eta_1,...,\eta_n)$ for short.
We can build them all in $N_{(\eta)}=N_{\eta_1}*_M...*_MN_{\eta_n}$ with the previous notation for $N_{\eta_j}=W^*(M,S_i^{(j)}, i\in I_j).$

One gets an obvious  definition of $M\otimes_{\sigma h}^{(I_1,1)}... \otimes_{\sigma h}^{(I_n,1)}M=(...(M\otimes_{\sigma h}^{(I_1,1)}M)... \otimes_{\sigma h}^{(I_n,1)}M$ by associativity and similarly for other tensor products  and of : $$.\#(S^{(1)},...,S^{(n)}):M\otimes_{\sigma h}^{(I_1,1)}... \otimes_{\sigma h}^{(I_n,1)}M\to N_{(\eta)}\subset L^2(N_{(\eta)})_c\cap L^2(N_{(\eta)})_r.$$

\begin{proposition}\label{InclusionEHCovariancenAry}The dual map of $.\#(S^{(1)},...,S^{(n)}):M\otimes_{\sigma h}^{(I_1,1)}... \otimes_{\sigma h}^{(I_n,1)}M\to N_{(\eta)}\subset L^2(N_{(\eta)})_c\cap L^2(N_{(\eta)})_r$  when restricted to $M\otimes_{h}^{(I_1,1)}... \otimes_{ h}^{(I_n,1)}M\to N_{(\eta)}$ gives a predual map $(.\#(S^{(1)},...,S^{(n)}))_*:L^2(N_{(\eta)})_c+ L^2(N_{(\eta)})_r)\to L^1(M)\otimes_{eh}^{(I_1,\infty)}... \otimes_{eh}^{(I_n,\infty)}L^1(M)\subset (M\otimes_{h}^{(I_1,1)}... \otimes_{ h}^{(I_n,1)}M)^*$ for the original map $.\#(S^{(1)},...,S^{(n)}).$
\end{proposition}

\begin{proof}This is again \cite[Prop 5.9]{EffrosRuanDual}.
\end{proof}

There are some subtleties in the multitensor case. We start by an intermediate definition.

\begin{definition}
We define the \textit{strict weak-* Haagerup tensor power relative to the list of  covariance maps $\eta=(\eta_1,...,\eta_n)$ } on $M$ by the operator space dual :$$M^{\otimes_{sw^*h}\eta}= (\overline{(.\#(S^{(1)},...,S^{(n)}))_*[(L^2(N_{(\eta)})_c+ L^2(N_{(\eta)})_r)]}^{ L^1(M)\otimes_{eh}^{(I_1,\infty)}... \otimes_{eh}^{(I_n,\infty)}L^1(M)})^*$$
with weak-* continuous completely quotient map and completely contractive injection (the injection is even weak-* homeomorphic on bounded sets) :$$p_{s\eta}:M\otimes_{\sigma h}^{(I_1,1)}... \otimes_{\sigma h}^{(I_n,1)}M\to M^{\otimes_{sw^*h}\eta},\ \ i_{s\eta}:M^{\otimes_{sw^*h}\eta}\subset L^2(N_{(\eta)})_c\cap (L^2(N_{(\eta)})_r.$$

\end{definition}

\begin{exemple}\label{CondExp2} As in example  \ref{CondExp}, using the density result in lemma \ref{ModuleLeftDual}, one sees $M^{\otimes_{sw^*h}(E_D,...,E_D)}=M^{\otimes_{w^*hD}n}$ completely isometrically and weak-* homeomorphically.
\end{exemple}

Let us now describe more explicitly the completion involved in the definition above as predual in operator algebraic terms. We start by recalling basic identities that will enable us to define a better notion of tensor product with associativity.

\begin{remark}
Using a (specific case $D=\C$ of a) description in proposition  \ref{HaagerupModule} $L^1(M)\simeq L^2(M)_r\otimes_{h M'}L^2(M)_c\simeq L^2(M)_r\otimes_{eh M'}L^2(M)_c$ (via \cite[Rmk 2.18]{M05}), one gets :
$$L^1(M)^{\otimes_{eh}(n+1)}
\simeq L^2(M)_r\otimes_{h M'}([B(L^2(M))]^{\otimes_{eh M'} n})\otimes_{h M'}L^2(M)_c$$

so that one recovers the identification for the duals (from \cite{EffrosKishimoto},\cite{EffrosExel} using \cite{M97} to identify definitions of tensor products) $$M^{\otimes_{\sigma h}(n+1)}=([B(L^2(M))]^{\otimes_{eh M'} n})^{\natural_{M',M'}}=CB_{M'}([B(L^2(M))]^{\otimes_{eh M'} n};B(L^2(M)))_{M'}.$$

For the Haagerup product, one replaces bounded operators by compact ones : $$L^1(M)^{\otimes_{h}(n+1)}
\simeq L^2(M)_r\otimes_{h M'}([\mathcal{K}(L^2(M))]^{\otimes_{h M'} n})\otimes_{h M'}L^2(M)_c,$$
$$M^{\otimes_{e h}(n+1)}=CB_{M'}([\mathcal{K}(L^2(M))]^{\otimes_{h M'} n};B(L^2(M)))_{M'}.$$

\end{remark}

\begin{definition}
We define the \textit{weak-* Haagerup tensor power relative to the list of  covariance maps $\eta=(\eta_1,...,\eta_n)$ } on $M$ by a $M'$-normal part (separately in all arguments) in the operator space dual :\begin{align*}M^{\otimes_{w^*h}\eta}= &[L^2(M)_r\otimes_{h M'}(\overline{(.\#S^{(1)})_*[L^2(M,\eta_1)_{L^2(M)}]}^{\ell^\infty(I_1)\overline{\otimes}B(L^2(M))}\cdots \otimes_{h M'}L^2(M)_c ]^{* M' normal}
 \\&\cap [L^2(M)_r\otimes_{h M'}(\overline{(.\#S^{(1)})_*[{}_{L^2(M)}L^2(M,\eta_1)]}^{\ell^\infty(I_1)\overline{\otimes}B(L^2(M))}\cdots \otimes_{h M'}L^2(M)_c  ]^{* M' normal}\end{align*}
with weak-* continuous completely quotient map and completely contractive injection (thanks to which one defines the intersection space structure above in $L^2(N_{(\eta)})_c+ (L^2(N_{(\eta)})_r$):$$p_{\eta}:M\otimes_{\sigma h}^{(I_1,1)}... \otimes_{\sigma h}^{(I_n,1)}M\to M^{\otimes_{w^*h}\eta},\ \ i_{\eta}:M^{\otimes_{w^*h}\eta}\subset M^{\otimes_{sw^*h}\eta}\subset L^2(N_{(\eta)})_c\cap (L^2(N_{(\eta)})_r.$$

We define the \textit{extended Haagerup tensor power relative to $\eta$} on $M$ by the image with induced norm from the inclusion : $M^{\otimes_{eh}\eta}:=p_\eta(M\otimes_{eh}^{(I_1,1)}... \otimes_{eh}^{(I_n,1)}M)\subset M^{\otimes_{w^*h}\eta}.$
\end{definition}
If $\eta$ is the empty list, both spaces are by convention $M$.
Note that the fact that $p_\eta$ is well defined is not obvious and will follow from the next automatic normality result.

The main technical result is the following
\begin{theorem}\label{SubmoduleEtaTHM}For $M$ a finite von Neumann algebra, $\eta=(\eta_1,...\eta_n)$ a list of covariance maps, we have a complete isometry :
$$(M\otimes_{w^*h\eta_1}M)\otimes_{e h M}\cdots\otimes_{e h M}(M\otimes_{w^*h\eta_n}M)\simeq M^{\otimes_{w^*h}\eta}= M^{\otimes_{sw^*h}\eta}.$$
Moreover, for $M\subset N$ finite von Neumann algebras, we have a canonical weak-* continuous inclusion  $M^{\otimes_{w^*h}\eta}\to N^{\otimes_{w^*h}\eta\circ E_M}.$
Finally, there is a canonical complete quotient map $$P_\eta:(M\otimes_{\sigma h}^{(I_1,1)}M)\otimes_{e h M}\cdots\otimes_{e h M}(M\otimes_{\sigma h}^{(I_n,1)}M)\to (M\otimes_{w^*h\eta_1}M)\otimes_{e h M}\cdots\otimes_{e h M}(M\otimes_{w^*h\eta_n}M)$$ such that for the canonical complete isometry  $$J_\eta:(M\otimes_{\sigma h}^{(I_1,1)}M)\otimes_{e h M}\cdots\otimes_{e h M}(M\otimes_{\sigma h}^{(I_n,1)}M)\to \underset{j=1...n}{\overset{\sigma h M}{\bigotimes}}(M\otimes_{\sigma h}^{(I_j,1)}M)\simeq M\otimes_{\sigma h}^{(I_1,1)}... \otimes_{\sigma h}^{(I_n,1)}M$$ we have the commutative diagram $p_{s\eta}\circ J_\eta=P_\eta.$
\end{theorem}

We will by now consider only the weak-* Haagerup product notation with the weak-* topology coming from the strict weak-* Haagerup definition.

This is an obvious consequence of the following lemma and theorem \ref{MainHaagerupModule} (since the intersection we describe in the bellow is then trivial since one of the spaces is included in the second giving the equality above).

\begin{lemma}\label{SubmoduleEta}The predual as a $M-M$ normal dual operator module of $M\otimes_{w^*h\eta_i}M$ is $$(M\otimes_{w^*h\eta_i}M)_\natural=\overline{(.\#S^{(i)})_*[L^2(M,\eta_i)_{L^2(M)}]}^{M',M'}=\overline{(.\#S^{(i)})_*[{}_{L^2(M)}L^2(M,\eta_i)]}^{M',M'}$$
the closure for the $M'-M'$ topology (cf \cite{MagajnaConvex}) of either right or left bounded vectors in $L^2(M,\eta_i)$ when sent via $(.\#S^{(i)})_* $ to $L^2(M)_c\otimes_{eh}^{(I_i,\infty)}L^2(M)_r \simeq \ell^{\infty}(I_i)\overline{\otimes} B(L^2(M)).$
Moreover
, we have a 
 complete isometry : $$(M\otimes_{w^*h\eta_1}M)\otimes_{e h M}\cdots\otimes_{e h M}(M\otimes_{w^*h\eta_n}M)\simeq M^{\otimes_{w^*h}\eta}.$$
 
 Finally, if we see canonically $$M^{\otimes_{w^*h}\eta}\subset M^{\otimes_{sw^*h}\eta}\subset N_{\eta_1}\otimes_{w^*h M}...\otimes_{w^*h M}N_{\eta_n},$$ $$M^{\otimes_{w^*h}\eta}\subset N_{\eta_1}\otimes_{eh M}...\otimes_{eh M}N_{\eta_n}\subset N_{\eta_1}\otimes_{w^*h M}...\otimes_{w^*h M}N_{\eta_n},$$ then 
 $$M^{\otimes_{w^*h}\eta}=N_{\eta_1}\otimes_{eh M}...\otimes_{eh M}N_{\eta_n}\cap M^{\otimes_{sw^*h}\eta},$$
 so that all inclusions above have weak-* closed image, all the spaces are normal dual operator M-M bimodules with all the inclusions weak-* continuous as well as $\iota_\eta$ which is even weak-* homeomorphic on bounded sets as well as the maps above which are isometric.
\end{lemma}
Note that this last statement gives a substitute of associativity, essential for later algebraic uses of this tensor product.
\begin{proof}

First, from the weak-* continuous complete quotient map, we have an inclusion  $(M\otimes_{w^*h\eta_i}M)_\natural\subset (M\otimes_{\sigma h}^{(I_i,1)}M)_\natural=L^2(M)_c\otimes_{eh}^{(I_i,\infty)}L^2(M)_r$ from \cite[Def 4.1, Th 4.2,Th 2.17]{M05} since $({}_MM_{\C})_{\natural}={}_{M'}[L^2(M)_c]_{\C}$ and 
$({}_{\C}M_M)_{\natural}={}_{\C}[L^2(M)_r]_{M'}$ (see e.g. \cite[proof of Corol 3.3]{M05}) and since an extended Haagerup product of strong modules is strong \cite[Prop 4.1]{M97}.

The identification $L^2(M)_c\otimes_{eh}\ell^\infty(I_i)\otimes_{eh}L^2(M)_r=L^2(M)_c\otimes_{w^*min}\ell^\infty(I_i)\otimes_{w^*min}L^2(M)_r
=\ell^{\infty}(I_i)\overline{\otimes} B(L^2(M))$ is well-known and it is clear from the definition that for $\xi\in L^2(M,\eta_i)$ right or left bounded, 
$(.\#S^{(i)})_*(\xi):(x_j)_{j\in I_i}\to (E_M[\xi x_jS^{(i)}_j])_{j\in I_i}$ is the element of $\ell^{\infty}(I_i)\overline{\otimes} B(L^2(M))$ in the identification above. Since the two left hand sides are closed in the $M'-M'$ topology, thus strong $M'-M'$ modules, it suffices, again from \cite[Th 2.17]{M05}, to identify their duals, e.g. for $Y=\overline{(.\#S^{(i)})_*[L^2(M,\eta_i)_{L^2(M)}]}^{M',M'}$: $$M\otimes_{w^*h\eta_i}M=(Y)^\natural =(L^2(M)_r\otimes_{hM'}Y\otimes_{hM'}L^2(M)_c)^*,$$

the last identity from \cite[Corol 3.4]{M05}. And for this by definition, it suffices to check \begin{align*}L^2(M)_r&\otimes_{hM'}Y\otimes_{hM'}L^2(M)_c\\&=\overline{(.\#(S^{(i)}))_*[(L^2(N_{\eta_i})_c+ L^2(N_{\eta_i})_r)]}^{L^2(M)_r\otimes_{h M'}(\ell^{\infty}(I_i)\overline{\otimes} B(L^2(M)))\otimes_{h M'}L^2(M)_c} \end{align*} with the identification of the last remark. Since the left hand side is closed this boils down to a density result, checkable on basic tensors (since we have Haagerup tensor product on both sides), but if one takes $x,y\in L^2(M), Y\ni \xi=\lim_{n}(.\#S^{(i)})_*(\xi_n)$, $\xi_n \in L^2(M,\eta_i)_{L^2(M)}$ (the convergence being in the $M-M$ topology and from \cite[Th 3.5]{MagajnaConvex} the net can be taken bounded in $(\ell^{\infty}(I_i)\overline{\otimes} B(L^2(M)))$) and $y_n\in M$ tending to $y$. Thus $(.\#S^{(i)})_*(y_n\xi_nx)=x\otimes (.\#S^{(i)})_*(\xi_n)\otimes y_n\to x\otimes (.\#S^{(i)})_*(\xi)\otimes y.$ since replacing $y_n$ by $y$ is easy by boundedness and then the normwise convergence $x\otimes (.\#S^{(i)})_*(\xi_n)\otimes y\to x\otimes (.\#S^{(i)})_*(\xi)\otimes y$ is easy from the definition of the $M'-M'$ topology (and basically its definition).

This concludes the proof of the first identity. For the second identity, one uses that for $\xi\in {}_{L^2(M)}L^2(M,\eta_i)] \xi\sqrt{\alpha}[\alpha+E_M(\xi^*\xi)]^{-1/2}\in L^2(M,\eta_i)_{L^2(M)}$ (actually it is both right and left bounded) and it is easy to see the  convergence of its image by $(.\#S^{(i)})_*$ to $(.\#S^{(i)})_*(\xi)$ when $\alpha\to \infty$ in the $M'-\C$ topology. With a symmetric argument, one thus checks the two inclusions concluding to the second equality :
$$(.\#S^{(i)})_*[{}_{L^2(M)}L^2(M,\eta_i)]\subset \overline{(.\#S^{(i)})_*[L^2(M,\eta_i)_{L^2(M)}]}^{M',\C},$$ $$ (.\#S^{(i)})_*[L^2(M,\eta_i)_{L^2(M)}]\subset \overline{(.\#S^{(i)})_*[{}_{L^2(M)}L^2(M,\eta_i)]}^{\C,M'}.$$

From \cite[Corol 3.5]{M05} (and \cite[Prop 3.13]{M05} for the automatic normality on the sides), if we call $$Y_i={}_{ M'}(\overline{(.\#S^{(i)})_*[L^2(M,\eta_i)_{L^2(M)}]}^{\ell^\infty(I_i)\overline{\otimes}B(L^2(M))}_{ M'},$$ $$Z_i={}_{ M'}(\overline{(.\#S^{(i)})_*[{}_{L^2(M)}L^2(M,\eta_i)]}^{\ell^\infty(I_i)\overline{\otimes}B(L^2(M))}_{ M'},$$ then  \begin{align*}M^{\otimes_{w^*h}\eta}&=[Y_1\otimes_{h M'}\cdots \otimes_{h M'} Y_n]^{\natural M' normal}\cap [Z_1\otimes_{h M'}\cdots \otimes_{h M'} Z_n]^{\natural M' normal}\\&\simeq Y_1^{\natural}\otimes_{eh M}\cdots \otimes_{eh M} Y_n^{\natural}\cap Z_1^{\natural}\otimes_{eh M}\cdots \otimes_{eh M} Z_n^{\natural}\end{align*}
with the last isomorphism by the multitensor variant of \cite[Theorem 3.2]{M05}.
But similarly  $Y_i^{\natural}=[L^2(M)_r\otimes_{h M'}(\overline{(.\#S^{(i)})_*[L^2(M,\eta_i)_{L^2(M)}]}^{\ell^\infty(I_i)\overline{\otimes}B(L^2(M))} \otimes_{h M'}L^2(M)_c]^{* }$ but $L^2(M)_r\otimes_{h M'}(\overline{(.\#S^{(i)})_*[L^2(M,\eta_i)_{L^2(M)}]}^{\ell^\infty(I_i)\overline{\otimes}B(L^2(M))} \otimes_{h M'}L^2(M)_c$ is a closed subspace of $L^1(M)\otimes_{eh}^{I_i,\infty}L^1(M)$ (by completeness and injectivity of module Haagerup tensor product) in which $L^2(M)_r\otimes_{h M'} L^2(M,\eta_i)_{L^2(M)}\otimes_{h M'}L^2(M)_c$
is dense, thus this is the closure of this space which is nothing but $(\overline{(.\#S)_*[(L^2(M,\eta_i)_c)]}$ as explained before. 

One thus obtains $(\overline{(.\#S)_*[(L^2(M,\eta_i)_c)]}=L^2(M)_r\otimes_{h M'}(\overline{(.\#S^{(i)})_*[L^2(M,\eta_i)_{L^2(M)}]}^{\ell^\infty(I_i)\overline{\otimes}B(L^2(M))} \otimes_{h M'}L^2(M)_c$. But similarly one sees this also equals $L^2(M)_r\otimes_{h M'}(\overline{(.\#S^{(i)})_*[L^2(M,\eta_i)_{L^2(M)}]}^{M'-M'} \otimes_{h M'}L^2(M)_c=(\overline{(.\#S)_*[(L^2(M,\eta_i)_c)]}.$ From the previous computations in the $M'-M'$ topology, one sees this is nothing but \begin{align*}L^2(M)_r&\otimes_{h M'}(\overline{(.\#S^{(i)})_*[L^2(M,\eta_i)_{L^2(M)}+{}_{L^2(M)}L^2(M,\eta_i)]}^{M'-M'} \otimes_{h M'}L^2(M)_c\\&=(\overline{(.\#S)_*[(L^2(M,\eta_i)_c)+(L^2(M,\eta_i)_r)]}.\end{align*}
 Thus by definition we have $M\otimes_{w^*h\eta_i}M=Y_i^{\natural}=Z_i^{\natural}$ (similarly) and one gets the stated isomorphism. (The reader should note that this change from norm to $M'-M'$ closure does not work for more than one tensor implying the more restrictive definition in this case).

Note finally that the inclusion \begin{align*}M^{\otimes_{w^*h}\eta}&\subset[Y_1\otimes_{h M'}\cdots \otimes_{h M'} Y_n]^{\natural }\cap [Z_1\otimes_{h M'}\cdots \otimes_{h M'} Z_n]^{\natural }\\&\simeq [Y_1\otimes_{h M'}\cdots \otimes_{h M'} Y_n+Z_1\otimes_{h M'}\cdots \otimes_{h M'} Z_n]^{\natural }=M^{\otimes_{sw^*h}\eta}\end{align*}
since the reasoning above again shows \begin{align*}Y_1\otimes_{h M'}&\cdots \otimes_{h M'} Y_n+Z_1\otimes_{h M'}\cdots \otimes_{h M'} Z_n\\&=(\overline{(.\#(S^{(1)},...,S^{(n)}))_*[(L^2(N_{(\eta)})_c+ L^2(N_{(\eta)})_r)]}^{ L^1(M)\otimes_{eh}^{(I_1,\infty)}... \otimes_{eh}^{(I_n,\infty)}L^1(M)})\end{align*} and the inclusion stated (even in the definition) is now clear.



From the definition, we have $\iota_{\eta_i}((M\otimes_{w^*h\eta_i}M))\subset N_{\eta_i}$ weak-* continuously so that we have a canonical map by functoriality $M^{\otimes_{w^*h}\eta}\to N_{\eta_1}\otimes_{eh M}...\otimes_{eh M}N_{\eta_n}.$
We will see later it is injective.
We have $N_{\eta_1}\otimes_{eh M}...\otimes_{eh M}N_{\eta_n} \subset N_{\eta_1}\otimes_{w^*h M}...\otimes_{w^*h M}N_{\eta_n}$ as a weak-* closed both dual operator spaces with the induced weak-* topology by proposition \ref{Submodules}.

In order to see the last weak-* continuous inclusion $M^{\otimes_{sw^*h}\eta}\subset N_{\eta_1}\otimes_{w^*h M}...\otimes_{w^*h M}N_{\eta_n}$ we won't build the explicitly, not so obvious, predual map but use a common inclusion in a third space. 
Using example \ref{CondExp2} and the projection $L^1(N_{(\eta)})\to L^1(N_{\eta_i}),$ one gets the weak-* continuous inclusions weak-* homeomorphic on bounded sets :$$I:N_{\eta_1}\otimes_{w^*h M}...\otimes_{w^*h M}N_{\eta_n}\subset N_{(\eta)}^{\otimes_{w^*h M}n}=N_{(\eta)}^{\otimes_{sw^*h (E_M,...,E_M)}}\subset N_{(\eta)}*_M(L(F_{n})\otimes M)$$ with this last embedding realized as $.\#(Y_1,...,Y_{n-1})$ the evaluation on $n-1$ free semicircular variables in $L(F_{n-1}).$
Now it is easy to see that $MS^{(1)}Y_1M...MS^{(n-1)}Y_{n-1}MS^{(n)}M$ generate an isomorphic $M-M$ Hilbert bimodule than $MS^{(1)}M...MS^{(n-1)}MS^{(n)}M$ in $L^2(N_{(\eta)}*_M(L(F_{n})\otimes M)),L^2(N_{(\eta)})$. Let $J$ the corresponding map with $J(S^{(i)})=S^{(i)}Y_i.$ Then $I^{-1}J\circ [ .\#(S^{(1)},...,S^{(n)})]=I^{-1}[ .\#(S^{(1)}Y_1,...,S^{(n-1)}Y_{n-1},S^{(n)})]$ is the weak-* continuous inclusion $M^{\otimes_{sw^*h}\eta}\subset N_{\eta_1}\otimes_{w^*h M}...\otimes_{w^*h M}N_{\eta_n}$ we want. The fact it has closed image is obvious by weak-* compactness.

From the previous inclusions, one thus deduces $M^{\otimes_{w^*h}\eta}\subset N_{\eta_1}\otimes_{eh M}...\otimes_{eh M}N_{\eta_n}\cap M^{\otimes_{sw^*h}\eta},$ let us prove the converse.
For, one takes $f\in M^{\otimes_{sw^*h}\eta}$ and one has to prove the normality condition on the predual in assuming also $f\in N_{\eta_1}\otimes_{eh M}...\otimes_{eh M}N_{\eta_n}$. Thus take by density $$z_1\in(.\#S^{(1)})_*[L^2(M,\eta_1)_{L^2(M)}]\otimes_{h M'}\cdots \otimes_{h M'} (.\#S^{(k)})_*[L^2(M,\eta_k)_{L^2(M)}],$$ $$z_2\in (.\#S^{(k+1)})_*[L^2(M,\eta_{n+1})_{L^2(M)}]\otimes_{h M'}\cdots \otimes_{h M'} (.\#S^{(n)})_*[L^2(M,\eta_n)_{L^2(M)}]$$
and one has to show $m\mapsto f(z_1\otimes mz_2)$ normal in $M'$ but following our isomorphisms, it is clear that $z_1,z_2$ comes from elements for instance in   ${}_{L^2(M)}L^1(N_{\eta_1})\otimes_{h M'}{}_{L^2(M)}L^1(N_{\eta_2})_{L^2(M)}\otimes_{h M'}...\otimes_{h M'}{}_{L^2(M)}L^1(N_{\eta_k})_{L^2(M)}$ giving the expected normality from the characterization of $N_{\eta_1}\otimes_{eh M}...\otimes_{eh M}N_{\eta_n}$ by a normality condition in \cite[Th 3.2]{M05}.

From the now proved equality with an intersection of two weak-* closed sets, $M^{\otimes_{w^*h}\eta}$ is thus a weak-* closed set of a dual operator space $M^{\otimes_{sw^*h}\eta}$, thus a dual operator space. The remaining weak-* continuities are then obvious.
\end{proof}

\begin{proof}[Proof of Thm \ref{SubmoduleEtaTHM}]

As explained, it only remains to build the inclusion. As explained, the equality implies that the normality condition in the definition of $M^{\otimes_{w^*h}\eta}$ can be removed showing that $p_\eta$ is well defined and weak-* continuous as stated, and equal to $p_{s\eta}$.

First, the canonical weak-* continuous map $i_{N,I}^\infty:M^{\otimes_{\sigma h} I}\to N^{\otimes_{\sigma h} I}$ is injective since the predual map is given by a tensor product to projection $E_M:L^1(N)\to L^1(M)$ which is surjective. From the inclusion $i_N^1:L^1(M)\to L^1(N)$ its tensor products and duals, we also see $M^{\otimes_{\sigma h} I}\to N^{\otimes_{\sigma h} I}$ is completely isometric since it is complemented by the map coming from tensor product of conditional expectations $E_M.$

Looking at the composition  $u=p_{\eta\circ E_M}\circ i_{N,I}^\infty: M^{\sigma h I}\to N^{\otimes_{w^* h}\eta\circ E_M}$
An element is in its kernel if and only if it is in the one of 
$\iota_{\eta\circ E_M}\circ u=(.\#(S_1,...,S_n)_{\eta\circ E_M})\circ i_{N,I}^\infty=\iota_{\eta}\circ p_{\eta}$ since a free product with amalgamation over $N_{(\eta\circ E_M)}=N*_MN_{(\eta)}$ gives the canonical space for $\eta\circ E_M$ variables. Thus $Ker u= Ker p_\eta$  
 and $u$ fatorizes to our expected injective map $u': M^{\otimes_{w^* h}\eta }\to N^{\otimes_{w^* h}\eta\circ E_M}$ unique with $u'\circ p_\eta=u.$
 
But note that looking at preduals in the strict Haagerup picture, on  $L^2(N_{(\eta\circ E_M)})$ we have the identity easily checked by duality and characterization of projections : $$(i_{N,I}^\infty)_*\circ(.\#(S_1,...,S_n)_{\eta\circ E_M})_*=(.\#(S_1,...,S_n)_{\eta})_*\circ E_{L^2(N_{(\eta\circ E_M)})\to L^2(N_{(\eta)})}$$
which enables to build an induced map $v: (N^{\otimes_{w^* h}\eta\circ E_M})_*\to (M^{\otimes_{w^* h}\eta })_*$ with $$v|_{(.\#(S_1,...,S_n)_{\eta\circ E_M})_*(L^2(N_{(\eta\circ E_M)})}=(.\#(S_1,...,S_n)_{\eta})_*\circ E_{L^2(N_{(\eta\circ E_M)})\to L^2(N_{(\eta)})}$$ so that  $$(i_{N,I}^\infty)_*\circ (p_\eta)_*=(p_{\eta\circ E_M})_*\circ v$$

and this induced map gives the expected predual map of the strict weak-* Haagerup product. Indeed, the dual relation gives $p_\eta\circ i_{N,I}^\infty=v^*\circ p_{\eta\circ E_M}$ that characterizes our $u'=v^*$. 

Let us turn to proving the commutative diagram. Note that to build $P_\eta,$ we use lemma \ref{IteratedModule} to get an extended Haagerup product of quotient maps $p_{\eta_i}$ using the module predual inclusions computed in lemma \ref{SubmoduleEta}. The construction of $J_\eta$ is obvious from associativity of normal Haagerup product and lemma \ref{IteratedModule}. Note that by this lemma if $$I_\eta:(M\otimes_{\sigma h}^{(I_1,1)}M)\otimes_{e h M}\cdots\otimes_{e h M}(M\otimes_{\sigma h}^{(I_n,1)}M)\subset (M\otimes_{\sigma h}^{(I_1,1)}M)\otimes_{w^* h M}\cdots\otimes_{w^* h M}(M\otimes_{\sigma h}^{(I_n,1)}M),$$
$$T_\eta:(M\otimes_{\sigma h}^{(I_1,1)}M)\otimes_{\sigma h M}\cdots\otimes_{\sigma h M}(M\otimes_{\sigma h}^{(I_n,1)}M)\to (M\otimes_{\sigma h}^{(I_1,1)}M)\otimes_{w^* h M}\cdots\otimes_{w^* h M}(M\otimes_{\sigma h}^{(I_n,1)}M),$$
 are the canonical inclusion and projection we have  $T_\eta\circ J_\eta=I_\eta.$ Moreover, by definition $P_\eta$ is obtained as a restrition of
$$Q_\eta:(M\otimes_{\sigma h}^{(I_1,1)}M)\otimes_{w^* h M}\cdots\otimes_{w^* h M}(M\otimes_{\sigma h}^{(I_n,1)}M)\to \underset{j=1...n}{\overset{w^* h M}{\bigotimes}}(M\otimes_{w^*h\eta_j}M)=\underset{j=1...n}{\overset{e h M}{\bigotimes}}(M\otimes_{w^*h\eta_j}M)$$ so that $Q_\eta\circ I_\eta=P_\eta.$ Thus it remains to check $p_{s\eta}=Q_\eta\circ T_\eta$. This is obvious since the predual maps are compositions of inclusions (using injectivity of Haagerup tensor product on the inclusion of the first part of lemma \ref{SubmoduleEta}) :$$\underset{j=1...n}{\overset{ h M'}{\bigotimes}}(M\otimes_{w^*h\eta_j}M)_\natural\subset \underset{j=1...n}{\overset{ h M'}{\bigotimes}}(M\otimes_{\sigma h}^{(I_j,1)}M)_\natural\subset \underset{j=1...n}{\overset{ eh M'}{\bigotimes}}(M\otimes_{\sigma h}^{(I_j,1)}M)_\natural.$$
 \end{proof}

In order to deal with the analogue of commutants needed in this context to produce the operadic structure needed to get composition of analytic functions, and not only in the weak-* Haagerup context but also in the extended Haagerup context we need a definition inspired from our result in the module case.

\begin{definition}
A list of covariance maps $\eta=(\eta_1,...,\eta_n)$ (resp. a set of covariance map $\eta=\{\eta_i, i\in I\}$) has (w*CIEHP) or has weak-* closed inclusion of extended Haagerup powers if the unit ball of  $M^{\otimes_{eh}\eta}$ is weak-* closed in $M^{\otimes_{w^*h}\eta}$ (resp for any $i_1,...,i_n\in I$, $(\eta_{i_1},...,\eta_{i_n})$ has (w*CIEHP))
\end{definition} 
We thus saw in proposition \ref{Submodules} that $\{E_B\}$ has (w*CIEHP) and this can be generalized to $\{E_B, B\subset M\}$. It is an interesting question to determine the set of covariance maps having this property but we won't try to solve this technical problem here.

\subsection{Spaces of variables relative to $\eta$ and multiplication on generalized Haagerup tensor products}\label{etavar}

The goal of this section is first to define spaces of variables relative to $\eta$, or what we could call $\eta$-commutants since they will coincide with commutants of $D$ in the case $\eta=E_D$.

Then we will produce projections on these spaces and we will provide multiplication maps in the spirit of Theorem \ref{HaagerupModule} (4).

Note that if $X$ is a normal dual operator D-bimodule (recall this is written $X\in {}_DNDOM_D$, the canonical map $D\otimes_hX\otimes_hD\to X,$ is normal when restricted to each argument $D,X,D$ (see e.g. \cite[Th 2.9, Tmk 2.10]{M05}, thus by \cite[Prop 5.9]{EffrosRuanDual} is extends uniquely to a weak-* continuous CB map  $m:D\otimes_{\sigma h}X\otimes_{\sigma h}D\to X.$ 

For such an $X$, there is a completely bounded separately weak-* continuous pairing $$.\#.: D\otimes_{\sigma h}^{(I,1)} D\times \ell^{\infty}(I,X)\to X.$$ Recall that $M_n(\ell^{\infty}(I,X))=\oplus_{i\in I}M_n(X)\simeq CB(\ell^1(I),M_n(X))\simeq  CB([M_n(X)]_*,\ell^\infty(I))\simeq NCB((\ell^\infty(I))^*,M_n(X))$ (where $\ell^1(I)$ is as usual given its maximal operator space structure). 

It is defined for $x\in M_n(\ell^{\infty}(I,X))$ by $d\#x=(id_{M_n}\otimes m)\mathcal{S}_\sigma[(1\otimes x\otimes 1)(d)]$ since $(1\otimes x\otimes 1)(d)\in D\otimes_{\sigma h}M_n(X)\otimes_{\sigma h}D$ is defined by functoriality (see \cite[p 149]{EffrosRuanDual}) of the normal Haagerup tensor product for normal maps as an element $x\in NCB((\ell^{\infty}(I))^*,M_n(X))$ and $S_\sigma: D\otimes_{\sigma h}(M_n\bar{\otimes}X)\otimes_{\sigma h}D\to M_n\bar{\otimes}[D\otimes_{\sigma h}X\otimes_{\sigma h}D]$ is the shuffle map of \cite[Th 6.1]{EffrosRuanDual}.

\begin{definition}
Let $X\in  {}_DNDOM_D$ and $\eta:D\to B(L^2(I))\otimes D$ a symmetric covariance map. The set of $\eta$ variables (or $\eta$-commutant), written $\eta'\cap \ell^{\infty}(I,X)$ is the subset of $\ell^{\infty}(I,X)$ of elements $x=(x_i)_{i\in I}$ such that $d\#x=0$ for all $d\in Ker (p_\eta)\subset D\otimes_{\sigma h}^{(I,1)} D.$
\end{definition}

We first note this is the convenient condition to define multiplication maps from our tensor product relative to $\eta$. 

\begin{proposition}\label{MultiplicationEtaVar}
Let $X\in  {}_DNDOM_D$ and $\eta:D\to B(L^2(I))\otimes D$ a symmetric covariance map. The paring $.\#.$ induces (after switch of the arguments) 
 complete contractions in:
$$M_{\#\sigma h}^{D,I,X}\in CB_D((\ell^{\infty}(I,X))\widehat{\otimes}^{D-D}(D\otimes_{\sigma h}^{(I,1)}D) , X)_D,$$
$$M_{\#w^*h}^{D,\eta,X}\in CB_D([\eta'\cap \ell^{\infty}(I,X)]\widehat{\otimes}^{D-D}(D\otimes_{w^*h\eta}D) , X)_D.$$

They are weak-* continuous in the second argument (i.e. $d\mapsto d\#x$ is weak-* continuous) and the first one also extends uniquely to a complete contraction
$$M_{\#\sigma h*}^{D,I,X}\in NCB_D([CB(\ell^{\infty}(I,X),(D\otimes_{\sigma h}^{(I,1)}D)_\natural)]^\natural , X)_D.$$

Moreover, the pairing $.\#.$  extends to $\ell^\infty(I,D\otimes_{\sigma h}^{(I,1)} D)\times \ell^{\infty}(I,X)\to \ell^{\infty}(I,X)$ and for any $x\in \eta'\cap \ell^{\infty}(I,X)$, any $d\in \ell^\infty(I,Ker (p_\eta)), d\#x=0.$  Finally, $\ell^\infty(I,D\otimes_{w^*h\eta}D)\simeq \ell^\infty(I,D\otimes_{\sigma h}^{(I,1)} D)/\ell^\infty(I,Ker (p_\eta)).$ As a consequence, the extended pairing induces one $CB([\eta'\cap \ell^{\infty}(I,X)]\widehat{\otimes}^{D-D}\ell^\infty(I,D\otimes_{w^*h\eta}D),  \ell^{\infty}(I,X)).$
\end{proposition}

\begin{proof}

 One considers the case $A=\ell^1(I),A^{**}=(\ell^\infty(I))^*, B=(\ell^{\infty}(I,X)), C_1=C_2=D$ of the map built in lemma \ref{Shuffle} and one gets : 
 $$S:
 \ell^{\infty}(I,X)\widehat{\otimes}^{D-D}[D\otimes_{\sigma h}(\ell^\infty(I))^*\otimes_{\sigma h}D]\to D\otimes_{\sigma h}(\ell^{\infty}(I,X)\widehat{\otimes} \ell^1(I))^{**}\otimes_{\sigma h}D.$$
 
But  we saw $(\ell^{\infty}(I,X)\widehat{\otimes} \ell^1(I))^{*}=CB(\ell^{\infty}(I,X),\ell^\infty(I))=CB(CB(X_*,\ell^{\infty}(I)),\ell^\infty(I))$
 and the evaluation map thus gives a completely bounded map $$E:X_*\to CB(CB(X_*,\ell^{\infty}(I)),\ell^\infty(I))=(\ell^{\infty}(I,X)\widehat{\otimes} \ell^1(I))^{*}$$
and by duality a normal CB map $$E^*:(\ell^{\infty}(I,X)\widehat{\otimes} \ell^1(I))^{**}\to X.$$
 Note that for any $x\in \ell^{\infty}(I,X), f\in (\ell^1(I))^{**}$, $E^*(x\otimes f)=x(f)$ in the identification $\ell^{\infty}(I,X)\simeq_\iota CB(X_*,\ell^{\infty}(I) \simeq_\kappa NCB((\ell^1(I))^{**},X)$ for which $\kappa(x)=\iota(x)^*$ is the adjoint map
 Indeed for $y\in X_*$ we have
 $$\langle E(y),(x\otimes f)\rangle=\langle \iota(x)[y], f\rangle\langle y, \kappa(x)[f]\rangle=:\langle y, x(f)\rangle.$$
 
Thus using \cite{EffrosRuanDual} again to build $Id\otimes E^* \otimes Id$  and the extended multiplication map for normal dual modules $m:D\otimes_{\sigma h}X\otimes_{\sigma h}D\to X,$ it remains to show that $m\circ (Id\otimes E^* \otimes Id)\circ S$ is the expected map induced by $.\#.$ and we only have to check on elementary tensor products. 
Since the definition of both maps end with $m$, it suffices to show that  $(Id\otimes E^* \otimes Id)\circ S(x\otimes d)=(1\otimes x\otimes 1)(d)$ for $x\in \ell^{\infty}(I,X)=NCB((\ell^\infty(I))^*,X), d\in [D\otimes_{\sigma h}(\ell^\infty(I))^*\otimes_{\sigma h}D]$ and from the normality we checked in the second argument (for $S$, the folowing compositions being obvious), it suffices to check equality on $d=c_1\otimes f\otimes c_2, c_1, c_2\in D, f\in  \ell^1(I)^*$, Then $$(Id\otimes E^* \otimes Id)\circ S(x\otimes d)=c_1\otimes E^*(x\otimes f)\otimes c_2=c_1\otimes x( f)\otimes c_2$$
 as expected.
 
For the second map induced by quotient, we use again the universal property proved in corollary \ref{ShuffleStrongModule} first to note that 
 $ .\#.\in CB([\eta'\cap \ell^{\infty}(I,X)],CB_D( D\otimes_{\sigma h}^{(I,1)} D, X)_D)$ by restriction. In order to get a map in $ CB_D([\eta'\cap \ell^{\infty}(I,X)]\widehat{\otimes}^{D-D}(D\otimes_{w^*h\eta}D) , X)_D,$ it suffices to note that by definition any element in $d\in Ker (p_\eta)$ induces the 0 map once restricted on the commutant, so that our pairing induces by quotient a map $ .\#.\in CB([\eta'\cap \ell^{\infty}(I,X)],CB_D( D\otimes_{w^*h\eta}D, X)_D).$ This concludes. To see normality in the second argument, since the map before quotient was normal, one can consider pointwise its predual map $(.\#x)_*:X_*\to [D\otimes_{\sigma h}^{(I,1)} D]_*$ but $Ker p_\eta=(D\otimes_{w^*h\eta}D)_*]^\perp$ and from the bipolar theorem $(D\otimes_{w^*h\eta}D)_*=(Ker p_\eta)_\perp$ so that the vanishing on $Ker p_\eta$ of $(.\#x)_*(y)$ for $y\in X, x\in  \eta'\cap \ell^{\infty}(I,X)$ guaranties $(.\#x)_*$ is then valued in $(D\otimes_{w^*h\eta}D)_*$ concluding to the weak-* continuity.
 
 In order to build the extension of the first map, one build the module predual map, one build the module predual map in $CB_{D'}(X_\natural, CB(\ell^{\infty}(I,X),(D\otimes_{\sigma h}^{(I,1)}D)_\natural))_{D'}$ and uses that the source and target space are strong module (see lemma \ref{Shuffle} for the target) so that the dual map will be also contractive by \cite[Prop 3.12]{M05} with same norm. Of course one uses again $\ell^{\infty}(I,X)\simeq CB(X_*,\ell^\infty(I)).$ Note that $(D\otimes_{\sigma h}^{(I,1)}D)_\natural\simeq D_{\natural_{D,\C}}\otimes_{eh}\ell^\infty(I)\otimes_{eh}D_{\natural_{\C,D}}$ and since $X$ is a normal dual operator module we have a normal contraction $D\otimes_{\sigma h} X\otimes_{\sigma h} D\to X$ thus from \cite[Prop 3.11]{M05} there is a unique module predual map $m_\natural X_\natural\to D_{\natural_{D,\C}}\otimes_{eh}X_*\otimes_{eh}D_{\natural_{\C,D}}$ also contractive since both module are strong so that one can use \cite[Prop 3.12]{M05}. One thus gets our map by composing this $m_\natural$ to a map in 

 \begin{align*}CB_{D'}&(D_{\natural_{D,\C}}\otimes_{eh}X_*\otimes_{eh}D_{\natural_{\C,D}}, CB(\ell^{\infty}(I,X),(D\otimes_{\sigma h}^{(I,1)}D)_\natural))_{D'}\\&=CB(CB(X_*,\ell^\infty(I)), CB_{D'}(D_{\natural_{D,\C}}\otimes_{eh}X_*\otimes_{eh}D_{\natural_{\C,D}},D_{\natural_{D,\C}}\otimes_{eh}\ell^\infty(I)\otimes_{eh}D_{\natural_{\C,D}})_{D'})\end{align*}
of course the canonical map corresponding to $T\mapsto (1\otimes T\otimes 1).$ It is apparent that for $x\in \ell^{\infty}(I,X)$ fixed, the dual map is indeed the same as $.\#x$ so that one obtains in the way an extension of $.\#.$ as expected.

For the uniqueness of this extension, it suffices to check the canonical map $B\widehat{\otimes}^{D-D}(D\otimes_{\sigma h}^{(I,1)}D)\to [CB(B,(D\otimes_{\sigma h}^{(I,1)}D)_\natural)]^\natural$ has weak-* dense image.

But note that the map decomposes as  \begin{align*}B\widehat{\otimes}^{D-D}(D\otimes_{\sigma h}^{(I,1)}D)&=[CB(B,(D\otimes_{\sigma h}^{(I,1)}D)^\natural)]_\natural\\&\to [CB(B,(D\otimes_{\sigma h}^{(I,1)}D)^\natural)]^\natural\\&\to [CB(B,(D\otimes_{\sigma h}^{(I,1)}D)_\natural)]^\natural
\end{align*}
where the first map is weak-* dense by Goldstine lemma and the interpretation of the module bidual as a complete quotient of the ordinary bidual in \cite[Corol 3.5 (iii)]{M05}, and the second map is also a weak-* continuous complete quotient map by \cite[Prop 3.12 (ii)]{M05} as the module dual of the complete isometry $CB(B,(D\otimes_{\sigma h}^{(I,1)}D)_\natural)\subset CB(B,(D\otimes_{\sigma h}^{(I,1)}D)^\natural)$, in a case where source and target spaces are strong modules.



For the last extension one starts by seeing $ .\#.\in CB(\ell^{\infty}(I,X),NCB_D(D\otimes_{\sigma h}^{(I,1)} D, X)_D)$  thus giving using normality and the predual map $ .\#.\in CB(\ell^{\infty}(I,X),CB(X_*,[D\otimes_{\sigma h}^{(I,1)} D]_*))$ and by universal property  of $\ell^1$ direct sums $ .\#.\in CB(\ell^{\infty}(I,X),CB(\ell^1(I,X_*),\ell^1(I,[D\otimes_{\sigma h}^{(I,1)} D]_*)))$ this thus gives again by dual map (and checking by hand the modularity by density of finite sums) the expected $$ .\#.\in CB(\ell^{\infty}(I,X),NCB_D(\ell^{\infty}(I,D\otimes_{\sigma h}^{(I,1)} D), \ell^{\infty}(I,X))_D).$$ The extended relation $d\#x=0$ is obvious by normality and the first argument and using weak-* density of finitely supported sequences in $\ell^\infty(I,Ker (p_\eta))$ and the induced map will be obtained as before once proved the quotient relation. 
But the canonical map from the universal property $\ell^\infty(I,D\otimes_{\sigma h}^{(I,1)} D)\to\ell^\infty(I,D\otimes_{w^*h\eta}D) $ is onto and induces obviously a quotient map  $\ell^\infty(I,D\otimes_{\sigma h}^{(I,1)} D)/\ell^\infty(I,Ker (p_\eta))\to\ell^\infty(I,D\otimes_{w^*h\eta}D).$
Moreover, the universal map corresponds to the restriction map $$\ell^\infty(I,D\otimes_{\sigma h}^{(I,1)} D)\simeq CB((D\otimes_{\sigma h}^{(I,1)} D)_*,\ell^\infty(I))\to CB((D\otimes_{w^*h\eta}D)_*,\ell^\infty(I))\simeq \ell^\infty(I,D\otimes_{w^*h\eta}D) $$ and thus since $\ell^\infty(I)$ is injective, by \cite[lemma 4.1.7]{EffrosRuan}, this map is a complete quotient map, and it suffices to identify its kernel to be $\ell^\infty(I,Ker (p_\eta))$ to conclude. We already said the kernel contain this space, but if $x$ in the kernel, it is obvious that the coordinate map $x(i)\in D\otimes_{\sigma h}^{(I,1)} D)$ is sent to the coordinate for $0 \in \ell^\infty(I,D\otimes_{w^*h\eta}D)$, i.e. $x(i)\in Ker(p_\eta)$ and this concludes.

\end{proof}

\section{Generalized Analytic non-commutative functions relative to covariance maps}
\subsection{Universal Analytic functions}

We are now ready to introduce our generalized analytic functions. We consider $B$ a finite von Neumann algebra, and $B_R$ is then the unit ball of radius $R$.
Let also $I=(I_1,...,I_n)$ a family of sets (made to evaluate $X_i\in \ell^\infty(I_i,N)$ some finite $N\supset B$)
 The set of monomials $m$  in $X_1,...,X_n,$ is written $M(X_1,...,X_n)$, $|m|$ its length and $I(m)=(I_{i_1},...,I_{i_{|m|}})$  if $m=X_{i_1}...X_{i_{|m|}}$.
Recall that for $J= (J_1,...,J_n)$, then $B^{\otimes_{\sigma h} [J,1]}:=(...(B\otimes_{\sigma h}^{(J_1,1)}B)...\otimes_{\sigma h}^{(J_n,1)}B).$ We also set : $$B^{\otimes_{\sigma e h} [J,1]}:=\underset{i=1...n}{\overset{e h B}{\bigotimes}}(B\otimes_{\sigma h}^{(J_i,1)}B)\subset \underset{i=1...n}{\overset{\sigma h B}{\bigotimes}}(B\otimes_{\sigma h}^{(J_i,1)}B)=B^{\otimes_{\sigma h} [J,1]}.$$

Thus we can define  analytic functions with $\ell^1$ direct sums in the above sense for $C\subset B$ a von Neumann subalgebra :
$$B_{\sigma h}\langle X_1,...,X_n:I,R,C\rangle:= B\oplus^{1}_C\ell^1_C\left(R^{|m|}B^{\otimes_{\sigma h}[I(m),1]};m\in M(X_1,...,X_n), |m|\geq 1\right),$$
$$B_{\sigma e h}\langle X_1,...,X_n:I,R,C\rangle:= B\oplus^{1}_C\ell^1_C\left(R^{|m|}B^{\otimes_{\sigma e h}[I(m),1]};m\in M(X_1,...,X_n), |m|\geq 1\right),$$
When $D=C$ we don't write the extra index $C$ and write $C_1-C_2$ if we consider different modules on each side.
$R^{|m|}E$ means the space $E$ with standard  norm multiplied by $R^{|m|}$. We may write $B_{\sigma h}\langle X:I,R,C\rangle$ for short. 
There is also an obvious variant with several radius of convergence $R,S$, $B_{\sigma h}\langle X_1,...,X_n;I,R;Y_1,...,Y_m:J, S,C\rangle,$ for a second list of sets $J=(J_1,...,J_m).$ 
We will use it freely later.



Then, we have several basic results, the first one considers evaluations :



\begin{theorem}\label{analyticEvaluationsigma}
$B_{\sigma e h}\langle X_1,...,X_n:I,R,C\rangle \subset B_{\sigma h}\langle X_1,...,X_n:I,R,C\rangle$ are  Banach algebras (even matrix normed) and respectively strong and normal dual operator $C$ bimodules, with separately weak-* continuous product in the second case. The algebra generated by $B,X_1,...,X_n$  is weak-* dense in $B_{\sigma h}\langle X:I,R,C\rangle.$ When $C=\C$ they are even $*$-algebras.

For $N\supset B$ a finite von Neumann algebra, $P\in B_{\sigma h}\langle  X_1,...,X_n:I,R,C\rangle$ defines a map (still written) $$P:\prod_{i=1}^n[ \ell^\infty(I_i,N)]_R\to  N,$$  by evaluation (in the algebraic case and then uniquely extended) and $ev_{(X_1,...,X_n)}^{\sigma h}:P\mapsto P(X_1,...,X_n)\in W^*(B,X_1,...,X_n)$ for $X_i\in [\ell^\infty(I_i,N)]_R$ is a weak-* continuous $C$-bimodular algebra morphism.  This is also a $*$ morphism when $C=\C$ and $X_i=X_i^*.$
\end{theorem}
\begin{proof}From lemma \ref{Ell1Sum}, it suffices to note that each term of the direct sum is a strong (resp. normal dual) operator $B$ (thus $C$) bimodule,  to check that the direct sum is a strong (resp. normal dual) operator $C$ bimodule. Both cases are easy by associativity of $eh$ and $\sigma h$ products. 

For $m_1,m_2\in M(X_1,...,X_n)$
from associativity of module $\sigma h$ product  and its definition as quotient, one deduces a $B$ bimodular completely contractive map :
$$R^{|m_1|}B^{\otimes_{\sigma h}[I(m_1),1]}\widehat{\otimes}_BR^{|m_2|}B^{\otimes_{\sigma h}[I(m_2),1]}\to R^{|m_1|}B^{\otimes_{\sigma h}[I(m_1),1]}{\otimes}_{\sigma h B}R^{|m_2|}B^{\otimes_{\sigma h}[I(m_2),1]}\to R^{|m_1|+|m_2|}B^{\otimes_{\sigma h}[I(m_1m_2),1]}.$$
Note that the separate weak-*continuity of this map is easy from the charaterization of \cite[(5.22)]{EffrosRuanDual}.

From the commutation of $\widehat{\otimes}_C$ and $\ell^1_C$ and there universal properties, one deduces the matrix normed algebra structure :$$B_{\sigma h}\langle X_1,...,X_n:I,R,C\rangle\widehat{\otimes}_CB_{\sigma h}\langle X_1,...,X_n:I,R,C\rangle\to B_{\sigma h}\langle X_1,...,X_n:I,R,C\rangle.$$
Separate weak-* continuity is easily extended to this setting by norm density of finite sums in the $\ell^1$ direct sums, the commutations above and weak-* continuity of the canonical inclusion. Note also that once obtained weak-* continuity in the second argument for $B\langle X_1,...,X_n:I,R,C\rangle\widehat{\otimes}_CR^{|m_2|}B^{\otimes_{\sigma h}[I(m_2),1]}\to B\langle X_1,...,X_n:I,R,C\rangle$, one deduces the full case by considering for fixed first argument the predual maps and gathering them by universal property of $c_0$ direct sums. The case of $B_{\sigma e h}$ is similar.

For the weak-* density of polynomials, it suffices to check it in each space of multilinear variables since finite sums of variables are normwise dense in the full infinite sum. But this is easy since from \cite[lemma 5.8]{EffrosRuanDual}, the algebraic tensor products is weak-* dense (using also $\ell^1(I_i)$ is weak-* dense in its bidual by goldstine lemma and thus so are finite sums in it) in the $\sigma$ Haagerup tensor product.

The antilinear map $a_1\otimes e_i ...\otimes e_k\otimes  a_{|m|}\mapsto a_{|m|}^*\otimes e_k\otimes....\otimes e_i\otimes a_1^*$ extends weak-* continuously to $B^{\otimes_{\sigma h} [I,1]} $ giving a completely contractive map we call ${.}^*:B^{\otimes_{\sigma h} [I,1]}\to B^{\otimes_{\sigma h} [I,1]}$   The corresponding map   on $\ell^1$ direct sums gives our $*$ algebra structure.

For $m=X_{i_1}...X_{i_{|m|}}$, we got from associativity of corollary \ref{ShuffleStrongModule}, $B^{\otimes_{\sigma h} [I,1]}\simeq \underset{j=1...|m|}{\overset{\sigma h B}{\bigotimes}}(B\otimes_{\sigma h}^{ (I_j,1)}B)$ weak-* homeomorphically and the last shuffle map of lemma \ref{Shuffle}, a complete contraction weak-* continuous in the second argument: $$S_m^{\sigma h}\left(\underset{j=1...|m|}{\widehat{\bigotimes}}[ \ell^\infty(I_{i_j},N)]\right)\widehat{\otimes}^{B-B} B^{\otimes_{\sigma h} [I,1]}\to \underset{j=1...|m|}{\overset{\sigma h B}{\bigotimes}}\left( [CB(\ell^\infty(I_{i_j},N),(B\otimes_{\sigma h}^{ (I_j,1)}B)_\natural)]^\natural\right)$$
and then compose with the tensor products of the normal extension obtained in proposition \ref{MultiplicationEtaVar} to get a map valued in $N^{\otimes_{\sigma h B}(|m|+1)}$ :

$$M^{\sigma h(I)}:=\underset{j=1...|m|}{\overset{\sigma h B}{\bigotimes}}M_{\#\sigma h*}^{B,I_{i_j},N}: \underset{j=1...|m|}{\overset{\sigma h B}{\bigotimes}}\left( [CB(\ell^\infty(I_{i_j},N),(B\otimes_{\sigma h}^{ (I_j,1)}B)_\natural)]^\natural\right)\to N^{\otimes_{\sigma h B}(|m|+1)},$$

 which we finally compose to the canonical weak-* continuous multiplication map to $m_\sigma: N^{\otimes_{\sigma h B}(|m|+1)}\to N$. Since the map $m_\sigma \circ M^{\sigma h(I)}\circ S_m^{\sigma h}$ is $C$ bimodular, it extends to direct sum, after using also diagonal maps to projective tensor product, and restriction to unit balls to obtain contractions after multiplication of the norm by $R^{|m|}.$ Normality then enables to build module predual maps with the same cb norm, that can be gathered via the universal property of $c_0$ sums and gives the expected weak-* continuity. This concludes to the definition of $ev.$ The $*$ algebra morphism property is easy on the polynomial algebra and extended by separate weak-* continuity and density.
\end{proof}

For a second list of sets $I'=(I'_1,...,I'_m),$ we will write $B_{\sigma (e)h\otimes I'}\langle X_1,...,X_n:I,R,C\rangle$ or for short, $B_{\sigma (e)h\otimes I'}\langle X:I,R,C\rangle$ the subspace of $B_{\sigma (e)h}\langle X_1,...,X_n,Y_1,...,Y_m:(I,I'),R,C\rangle$ linear in each $Y_k$ with $Y_k$'s in increasing order of $k$ in each monomial.

We will use this space to deal with free difference quotients on this space of analytic functions.

\begin{proposition}\label{analyticDerivationsigma}
%

$B_{\sigma (e) h}\langle X_1,...,X_n:I,R,C\rangle$ is mapped completely boundedly  weak-* continuously by the iterated free difference quotients  $\partial_{(i_1,...,i_k)}^{k}=(\partial_{X_{i_1}}\otimes 1^{\otimes k-1})\ldots \partial_{X_{i_k}}$ to 
\begin{align*}&B_{\sigma (e)h\otimes I_{(i_1,...,i_k)}}\langle X:I,S,C\rangle\\&\subset B_{\sigma  h}\langle X:D,S,C-\C\rangle\otimes_{\sigma h}^{(I_{i_1},1)}B_{\sigma  h}\langle X:D,S,\C\rangle^{\otimes_{\sigma  h} (k-1)}\otimes_{\sigma h}^{(I_{i_k},1)}B_{\sigma h}\langle X:D,S,\C-C\rangle\end{align*}
 (these last inclusions being completely contractive but not isometric) with $I_{(i_1,...,i_k)}=(I_{i_1},...,I_{i_k})$ for $S<R.$ 

Moreover $B_{\sigma (e) h \otimes I'}\langle  X_1,...,X_n:I,R,C\rangle$ is a (matrix normed) bimodule over $B_{\sigma (e) h}\langle X_1,...,X_n:I,R,C\rangle$ with separately weak-* continuous product so that $\partial_{X_{i}}$ becomes a derivation and, for $N\supset B$ a finite von Neumann algebra, $P\in B_{\sigma (e) h \otimes I'}\langle  X_1,...,X_n:I,R,C\rangle$ defines a map (still written) $$P:\prod_{i=1}^n[ \ell^\infty(I_i,N)]_R\to  N^{\otimes_{\sigma (e) h}[I',1]},$$  by evaluation (in the algebraic case and then uniquely extended) and $ev_{(X_1,...,X_n)}^{\sigma (e) h \otimes I'}:P\mapsto P(X_1,...,X_n)\in W^*(B,X_1,...,X_n)^{\otimes_{\sigma (e) h}[I',1]},$ for $X_i\in [ \ell^\infty(I_i,N)]_R$, is a (weak-* continuous in the case without index e) $C$-bimodular contraction and compatible with the $B_{\sigma (e) h}\langle X_1,...,X_n:I,R,C\rangle$ module structure with the relation $$ ev_{(X_1,...,X_n)}^{\sigma  h \otimes I'}\circ\iota_{\otimes I'}=\iota_{N,I'}\circ ev_{(X_1,...,X_n)}^{\sigma e h \otimes I'}$$ for the canonical inclusions $\iota_{N,I'}:N^{\otimes_{\sigma e h}[I',1]}\to N^{\otimes_{\sigma  h}[I',1]}, \iota_{\otimes I'}: B_{\sigma e h \otimes I'}\langle  X_1,...,X_n:I,R,C\rangle\to B_{\sigma  h \otimes I'}\langle  X_1,...,X_n:I,R,C\rangle$

\end{proposition}

\begin{proof}
Let us write $n_{X_i}(m)$ the number of variable $X_i$ in a monomial.
For the free difference quotient, we start from the formal differentiation :
$$\partial_{X_i} :B_{\sigma (e) h}\langle X_1,...,X_n:\eta,R,C\rangle\to\ell^1_C\left(S^{|m|}(B^{\otimes_{\sigma (e) h}[I(m),1]})^{\oplus^1_\C n_{X_i}(m)};m\in M(X_1,...,X_n), |m|\geq 1\right).$$
 These are completely bounded maps since $|m|(S/R)^{|m|}$ are bounded.
For the free difference quotient, to see there is a canonical map to the expected $B\langle X_1,...,X_n:D,R\rangle\otimes_{eh D}B\langle X_1,...,X_n:D,R\rangle,$ one uses the following lemma to each term of the direct sum inductively, and then the universal property of $\ell^1$ direct sums to combine them. 
 The iterated case is similar and building the evaluation map similar to the one of the previous theorem. The matrix normed module map is induced by the multiplication on $B_{\sigma (e)h}\langle X_1,...,X_n,Y_1,...,Y_m:(I,I'),R,C\rangle.$
The relation between evaltuations is similar to a relation that will be explained in detail in the proof of Theorem \ref{analyticEvaluation} bellow and also left to the reader.
The derivation property for $\partial_{X_{i}}$ is obvious on polynomials and then extended by weak-* density using various separate weak-* continuities.
\end{proof}

The following result is a normal Haagerup variant of \cite[lemma 7]{OP97}, the proof is the same using universal property of $\ell^1-c_0$ direct sums and standard duality tricks.
\begin{lemma}
Let $E_1,E_2,F_1,F_2$ dual operator spaces, let $X=(E_1\oplus^1 E_2)\otimes_{\sigma h}(F_1\oplus^1 F_2)$. Let $S$ be the closure of the subspace obtained by functoriality of Haagerup tensor product  $(E_1\otimes_{\sigma h } F_1)+(E_2\otimes_{\sigma h  } F_2)\subset X.$ Then we have :
$$S\simeq (E_1\otimes_{\sigma h } F_1)\oplus^1(E_2\otimes_{\sigma h } F_2),$$
completely isometrically.
\end{lemma}

\begin{proof}From the complete isometric inclusion 
$$(E_1\otimes_{h } F_1)\oplus^1(E_2\otimes_{ h } F_2)\to (E_1\oplus^1 E_2)\otimes_{ h}(F_1\oplus^1 F_2)$$ one gets by duality the composition of canonical inclusion and  complete quotient map $$((E_1)_*\oplus^\infty (E_2)_*)\otimes_{e h}((F_1)_*\oplus^\infty (F_2)_*)\subset (E_1^*\oplus^\infty E_2^*)\otimes_{e h}(F_1^*\oplus^\infty F_2^*)\to(E_1^*\otimes_{eh } F_1^*)\oplus^\infty(E_2^*\otimes_{e h } F_2^*) $$
but by the universal property one have a map $$((E_1)_*\oplus^\infty (E_2)_*)\otimes_{e h}((F_1)_*\oplus^\infty (F_2)_*)\to ((E_1)_*\otimes_{eh } (F_1)_*)\oplus^\infty((E_2)_*\otimes_{e h } (F_2)_*)$$ 
that coincides with the previous map when composed with the canonical injection.

To see this map is a complete quotient map, take contractions $u\in ((E_1)_*\otimes_{eh } (F_1)_*)\subset CB(E_1\otimes_{h } F_1,\C), v\in ((E_2)_*\otimes_{e h } (F_2)_*)\subset CB(E_2\otimes_{h } F_2,\C)$ and now use Christensen-Sinclair theorem as in \cite[lemma 7]{OP97} but in the version of \cite[Thm 5.1]{EffrosRuanDual} so that $u(x_1\otimes x_2)=u_1(x_1)u_2(x_2), v(y_1\otimes y_2)=v_1(y_1)v_2(y_2)$ with $v_1\in CB(E_2,B(\C,H)),v_2\in CB(F_2,B(H,\C))$, $u_1\in CB(E_1,B(\C,H)),u_2\in CB(F_1,B(H,\C))$ contractions with all the maps weak-* continuous. Then the map $\alpha=u_1\oplus v_1:(E_1\oplus^1 E_2)\to B(\C,H), \beta =u_2\oplus v_2:(F_1\oplus^1 F_2)\to B(H,\C)$ are complete contractions and weak-* continuous and $\alpha.\beta$ is the extension of $u+v$ to $(E_1\oplus^1 E_2)\otimes_{ h}(F_1\oplus^1 F_2)$. By separate weak-* continuity this extension is in $((E_1)_*\oplus^\infty (E_2)_*)\otimes_{e h}((F_1)_*\oplus^\infty (F_2)_*)$ so that one gets the stated complete quotient map that gives the desired complete isometry by duality.

%
\end{proof}
It will be convenient for us to gather evaluations of free difference quotient in a map with values in series to obtain a Taylor expansion. We need extra notation. For convenience, we write for a monomial $m=X_{i_1}...X_{i_{|m|}}$, $\partial^{|m|}_m=\partial^{|m|}_{(i_1,...,i_{|m|})}.$ For $m\in M(Y_1,...,Y_n)$ with $m=Y_{i_1}...Y_{i_{|m|}}$ we write $m(X)=X_{i_1}...X_{i_{|m|}}\in M(Y_1,...,Y_n)$. A submonomial $l\subset m$ is merely $l=Y_{j_1}...Y_{j_{|l|}}$ for which there is an increasing sequence $k_1<...<k_{|l|}$ with $j_o=i_{k_o}$ the sequence $k$ being fixed in the submonomial (as a range of a map is fixed in a map).  $(l\subset m)(T,Z)$ is then the monomial in $T_1,...,T_{|l|}, Z_1,...,Z_n$ defined by $(l\subset m)(T,Z)=Z_{i_1}...Z_{i_{k_1}-1}T_1Z_{i_{k_1}+1}...Z_{i_{k_2}-1}T_2...$ i.e. we replace $Y_i$ by $Z_i$ except at the positions of $k$ where we substitute $T_i$ in increasing order.

Finally recall that, with $\ell^1$ direct sums, comes canonical injections so that for a word in $m\in M(Z_1,...,Z_n,Y_1,...,Y_m)$
having only one of each $Y_i$ in increasing order, there is a completely contractive weak-* continuous maps $$\iota_m: R^{|m|}B^{\otimes_{\sigma (e)h}[(I,I')(m),1]}\to B_{\sigma (e) h}\langle X:I,R,C\rangle.$$
\begin{lemma}\label{TaylorAnalytic}Let $N\supset B$ a finite von Neumann algebra  $X_i\in [ \ell^\infty(I_i,N)]$ and $n\in M(Y_1,...,Y_n)$, 
$A=W^*(B,X_1,...,X_n)$. Assume $\min_i[R-||X_i||]=S>0,$ then for any $P\in B_{\sigma (e) h \otimes I'}\langle  X_1,...,X_n:I,R,C\rangle$ we have 
\begin{align*}ev_{(X_1,...,X_n)}^{\sigma (e) h-an}(\partial^{|n|}_{n(X)}P):=&\left([\partial^{|n|}_{n(X)}P](X)\oplus \bigoplus_{n\subset m\in M(Y_1,...,Y_n)}\iota_{(n\subset m)(Y,Z)}(ev_{(X_1,...,X_n)}^{\sigma (e) h\otimes I(m)}[\partial^{|m|}_{m(X)}(P)])\right)\\&\in A_{\sigma (e) h \otimes I(n)}\langle  Z_1,...,Z_n:I,S,C\rangle\end{align*}
 and $ev_{(X_1,...,X_n)}^{\sigma eh -an}$ is a (weak-* continuous in the case without index e) $C$-bimodular contraction and algebra homomorphism such that for any $Z_i\in [ \ell^\infty(I_i,N)]_S$
we have $$ev_{(Z_1,...,Z_n)}^{\sigma h}(ev_{(X_1,...,X_n)}^{\sigma h-an}(\partial^{|n|}_{n(X)}P))=ev_{(X_1+Z_1,...,X_n+Z_n)}^{\sigma h}(\partial^{|n|}_{n(X)}P).$$
 \end{lemma}

\begin{proof}
The last equation is obvious on polynomials and is nothing but a rewriting of \cite[(17)]{dabrowski:SPDE} in a more abstract language, or said otherwise this is a non-commutative Taylor formula for polynomials. Once $ev_{(X_1,...,X_n)}^{\sigma (e) h-an}$ built, this will extend by weak-* density and continuity.
Let $T=\max(||X_i||)$ so that $S+T=R$.

To build $ev_{(X_1,...,X_n)}^{\sigma (e) h-an}$ we decompose in the composition of two maps, one associated with a formal differentiation \begin{align*}ev^{\sigma (e) h-an}(\partial^{|n|}_{n(X)}P):=&\left([\partial^{|n|}_{n(X)}P]\oplus \bigoplus_{n\subset m\in M(Y_1,...,Y_n)}\iota_{(n\subset m;P)(Y,Z;X)}([\partial^{|m|}_{m(X)}(P)])\right)\\&\in B_{\sigma (e) h \otimes I(n)}\langle  Z_1,...,Z_n: I,S;X_1,...,X_n:I,T,C\rangle\end{align*}
with $\iota_{(n\subset m;P)(Y,Z;X)}$ as before but with all the other variables of $P$ formally evaluated at $X$. Seeing this map is completely contractive is only applying binomial formula on $(S+T)^n$ for each monomial block $P$ in which case the map reduces to a map of the form $P\mapsto(P....,P)$ of copied monomial index with different weight in the monomial expansion. It then suffices to apply a universal property of $\ell^1$ sums to go beyond the monomial case.
 This map is then composed with evaluation in $X_i$ variables which is treated as before and as much weak-* continuous as before.
\end{proof}

We can also use evaluation in a flexible way to produce operations on non-commutative analytic functions we will need later :

\begin{lemma}
\label{PartialEvaluation}Let $N\supset B$ a finite von Neumann algebra  $X_i\in [ \ell^\infty(I_i,N)]_R$ and $n=X_{i_1}...,X_{i_m}\in M(X_1,...,X_n)$. 
For any $P\in B_{\sigma (e) h }\langle  X_1,...,X_n:I,R,C\rangle$ 
$a_1,...,a_{m-1}\in A=W^*(B,X_1,...,X_n),j_k\in I_{i_k},$ there is $Q\in  CB(A,B_{\sigma (e) h }\langle  X_1,...,X_n:I,R,C\rangle)$ such that for any $a,b,c\in A$, \begin{align*}\tau(ev_{(X_1,...,X_n)}^{\sigma h}(Q(a))b)=\tau(bS^{(m,i_m)}_{j_m}a_{m-1}...a_1S^{(1,i_1)}_{j_1}a[ev_{(X_1,...,X_n)}^{\sigma h}(\partial_n^{|n|}P))]\#(S^{(1,i_1)},...,S^{(m,i_m)}),\end{align*}

\begin{align*}\tau&([ev_{(X_1,...,X_n)}^{\sigma h\otimes I_k}(\partial_{X_{i_{m+1}}}Q(a))\#S^{(m+1,i_{m+1})}]cS^{(m+1,i_{m+1})}_{j_{m+1}}b)\\&=\tau(cS^{(m+1,i_{m+1})}_{j_{m+1}}bS^{(m,i_m)}_{j_m}a_{m-1}...a_1S^{(1,i_1)}_{j_1}a[ev_{(X_1,...,X_n)}^{\sigma h}(\partial_{nX_{i_{m+1}}}^{|n|+1}P))]\#(S^{(1,i_1)},...,S^{(m+1,i_{m+1})})),\end{align*}
where $(S^{(1,i_1)},...,S^{(m+1,i_{m+1})}))$ are a semicircular system free with amalgamation of one another and from $A$ and $S^{(j,i_j)}$ being an $\eta_{i_j}$ semicircular over $B$.
\end{lemma} 
Once $Q(a)$ is understood as a polynomial version of $$E_A[S^{(m,i_m)}_{j_m}a_{m-1}...a_1S^{(1,i_1)}_{j_1}a[ev_{(X_1,...,X_n)}^{\sigma h}(\partial_n^{|n|}P))]\#(S^{(1,i_1)},...,S^{(m,i_m)})],$$
 the result is obvious, we make $C$ act on the space for $Q$ by $(c_1Qc_2)(a)=Q(ac_1)c_2$ so that the map $P\mapsto Q$ will be $C$ bimodular and will pass to modular $\ell^1$ direct sums. The existence of this map by separate evaluation is easy on each monomial space and left to the reader.

\subsection{Definition and evaluations}

We are now ready to introduce our generalized analytic functions relative to covariance maps. We consider $B$ a finite von Neumann algebra. 
Let also $\eta=(\eta_1,...,\eta_n)$ a family of covariance maps $\eta_i:B\to M_{I_i}(B).$

We are going to define $B\langle X_1,...,X_n:\eta,R\rangle$ 
This will be a set of non-commutative analytic functions 
that will be evaluated at $X_i$ in a set of $\eta_i$ variables. The radius of convergence in $X$ variables will be $R$. For a vector space $X$, $X_R$ will be the ball of radius $R$. 
 The set of monomials $m$  in $X_1,...,X_n,$ is written $M(X_1,...,X_n)$, $|m|$ its length and $\eta(m)=(\eta_{i_1},...,\eta_{i_{|m|}})$  if $m=X_{i_1}...X_{i_{|m|}}$.

Thus we can define  analytic functions with $\ell^1$ direct sums in the above sense for $C\subset B$ a von Neumann subalgebra :
$$B\langle X_1,...,X_n:\eta,R,C\rangle:= B\oplus^{1}_C\ell^1_C\left(R^{|m|}B^{\otimes_{w^*h}\eta(m)};m\in M(X_1,...,X_n), |m|\geq 1\right).$$
When $D=C$ we don't write the extra index $C$.
$R^{|m|}E$ means the space $E$ with standard  norm multiplied by $R^{|m|}$.
There is also an obvious variant with several radius of convergence $R,S$, $B\langle X_1,...,X_n:\eta,R;Y_1,...,Y_m: \eta',S\rangle$ for a second list of covariance maps $\eta'=(\eta'_1,...,\eta'_m).$ We will use it freely later.


Note we will still call weak-* topology the topology on $B_{\sigma e h}\langle X_1,...,X_n:I,R,C\rangle \subset B_{\sigma h}\langle X_1,...,X_n:I,R,C\rangle$ induced by the weak-* topology (even if this is a dense subspace).

Then, we have several basic results, the first one considers evaluation and composition :



\begin{theorem}\label{analyticEvaluation}
$B\langle X_1,...,X_n:\eta,R,C\rangle$ is a Banach algebra (even matrix normed) and a normal dual operator $C$ bimodule. $$P^{\sigma e h}:B_{\sigma e h}\langle X_1,...,X_n:I,R,C\rangle\to B\langle X_1,...,X_n:\eta,R,C\rangle$$ is a weak-* continuous complete quotient map inducing a final topology of the weak-* topology (induced by $B_{\sigma h}$) which is separated and we will call $\sigma e h$-weak * topology. The algebra generated by $B,X_1,...,X_n$ in it is $\sigma e h$-weak-* dense  thus weak-* dense . When $C=\C$ it is even a $*$-algebra.

For $N\supset B$ a finite von Neumann algebra, $P\in B\langle  X_1,...,X_n:\eta,R,C\rangle$ defines a map (still written) $$P:\prod_{i=1}^n[\eta_i'\cap \ell^\infty(I_i,N)]_R\to  N,$$  by evaluation (in the algebraic case and then uniquely extended) and $ev_{(X_1,...,X_n)}:P\mapsto P(X_1,...,X_n)\in W^*(B,X_1,...,X_n)$ for $X_i\in [\eta_i'\cap \ell^\infty(I_i,N)]_R$ is a $C$-bimodular algebra morphism. Moreover we have $ev_{(X_1,...,X_n)}\circ P^{\sigma e h}=ev_{(X_1,...,X_n)}^{\sigma h}$ so that $ev_{(X_1,...,X_n)}$ is $\sigma eh$-weak-* continuous. This is also a $*$ morphism when $C=\C$ and $X_i=X_i^*.$

\end{theorem}
The reader should note that our proof of the relation with previous evaluations uses strongly one considers $P^{\sigma e h}$ and not a variant defined on $B_{\sigma h}\langle X:I,R,C\rangle.$

\begin{proof}
For $m_1,m_2\in M(X_1,...,X_n)$, from the canonical isomorphism of $B^{\otimes_{w^*h}\eta(m_i)}$ as an extended Haagerup product over $B$ proved in lemma  \ref{SubmoduleEta}, one deduces a $B$ bimodular completely contractive map :
$$R^{|m_1|}B^{\otimes_{w^*h}\eta(m_1)}\widehat{\otimes}_BR^{|m_2|}B^{\otimes_{w^*h}\eta(m_2)}\to R^{|m_1|+|m_2|}B^{\otimes_{w^*h}\eta(m_1m_2)}.$$
From the commutation of $\widehat{\otimes}_C$ and $\ell^1_C$ and there universal properties, one deduces the matrix normed algebra structure :$$B\langle X_1,...,X_n:\eta,R,C\rangle\widehat{\otimes}_CB\langle X_1,...,X_n:\eta,R,C\rangle\to B\langle X_1,...,X_n:\eta,R,C\rangle.$$
From lemma \ref{Ell1Sum}, it suffices to note that each term of the direct sum is a normal dual operator $B$ (thus $C$) bimodule, what we did in lemma \ref{SubmoduleEta}, to check that the direct sum is a normal dual operator $C$ bimodule. 

$P^{\sigma e h}$ is easily combined from the maps $P_\eta$ built in theorem \ref{SubmoduleEtaTHM}. From the commutative diagram there, we also have a similar complete quotient map $P^{\sigma  h}:B_{\sigma h}\langle X_1,...,X_n:I,R,C\rangle\to B\langle X_1,...,X_n:\eta,R,C\rangle$ which is obviously weak-* continuous from this property for $p_{s\eta}$ and such that if $\iota:B_{\sigma e h}\langle X_1,...,X_n:I,R,C\rangle\subset B_{\sigma h}\langle X_1,...,X_n:I,R,C\rangle$ is the inclusion we have $P^{\sigma  h}\circ\iota =P^{\sigma e h}$. One deduces the stated weak-* continuity for $P^{\sigma e h}.$ The separatedness of the final topology then follows from the separatedness of the weak-* topology at the target.  The $\sigma e h$-weak-* density then follows readily from theorem \ref{analyticEvaluationsigma}.


The antilinear map $a_1\otimes a_2\mapsto a_2^*\otimes a_1^*$ extends weak-* continuously to $B\otimes_{\sigma h} B $ giving a completely contractive map we call ${.}^*:B\otimes_{\sigma h} B\to B\otimes_{\sigma h} B$ and $(X\#S^{(1)})^*=X^*\#S^{(1)}$. Thus it induces a map $B\otimes_{w^*h \eta_i}B\to B\otimes_{w^*h \eta_i}B.$ Of course we extend it to $B^{\otimes_{w^*h} \eta(m)}\to B^{\otimes_{w^*h} \eta(m)}$ by $a_1\otimes_B ...\otimes_B a_{|m|}\mapsto a_{|m|}^*\otimes_B....\otimes_B a_1^*, a_i\in B\otimes_{w^*h \eta(m)_i}B.$ The corresponding map   on $\ell^1$ direct sums gives our $*$ algebra structure.

We now build evaluation maps $ev_{(X_1,...,X_n)},ev_{(X_1,...,X_n)}^{\sigma e h}$ the second one on $B_{\sigma e h}\langle X_1,...,X_n:I,R,C\rangle$. We show the expected relation $ev_{(X_1,...,X_n)}P^{\sigma e h}=ev_{(X_1,...,X_n)}^{\sigma e h}$ and we will then show $ev_{(X_1,...,X_n)}^{\sigma e h}$ coincides with the restriction of the map built in theorem \ref{analyticEvaluationsigma}. Separating the reasoning in two steps will enable to divide questions around quotients from those around inclusion of extended to normal Haagerup products that have to be treated with different preliminary arguments we established earlier.

For $m=X_{i_1}...X_{i_{|m|}}$, we got from associativity and shuffle maps of corollary \ref{ShuffleStrongModule} a complete contraction $$S_m:\left(\underset{j=1...|m|}{\widehat{\bigotimes}}[\eta_{i_j}'\cap \ell^\infty(I_{i_j},N)]\right)\widehat{\otimes}^{B-B} B^{\otimes_{w^*h} \eta(m)}\to \underset{j=1...|m|}{\overset{eh B}{\bigotimes}}\left([\eta_{i_j}'\cap \ell^\infty(I_{i_j},N)]\widehat{\otimes}^{B-B}(B\otimes_{w^*h \eta_{i_j}}B)\right)$$
and similarly
$$S_m^{\sigma e h}:\left(\underset{j=1...|m|}{\widehat{\bigotimes}}[ \ell^\infty(I_{i_j},N)]\right)\widehat{\otimes}^{B-B} B^{\otimes_{\sigma e h} [I(m),1]}\to \underset{j=1...|m|}{\overset{eh B}{\bigotimes}}\left([\ell^\infty(I_{i_j},N)]\widehat{\otimes}^{B-B}(B\otimes_{\sigma h}^{( I_{i_j},1)}B)\right).$$

From the commutative diagram in the stated corollary, one deduces the relation between them if we recall $P_{\eta(m)}=p_{\eta_{i_1}}\otimes_{eh B}...\otimes_{eh B}p_{\eta_{i_{|m|}}}$ from theorem \ref{SubmoduleEtaTHM}, we have when restricted to common domains $$S_m\circ (1\otimes P_{\eta(m)})=[\underset{j=1...|m|}{\overset{eh B}{\bigotimes}}(1\otimes  p_{\eta_{i_j}})]\circ S_m^{\sigma e h}.$$

Then we can compose with the tensor products of maps obtained in proposition \ref{MultiplicationEtaVar} to get a map valued in $N^{\otimes_{eh B}|m|}$ :
$$M^{\sigma e h(I)}:=\underset{j=1...|m|}{\overset{eh B}{\bigotimes}}M_{\#\sigma h}^{B,I_{i_j},N}: \underset{j=1...|m|}{\overset{eh B}{\bigotimes}}\left([\ell^\infty(I_{i_j},N)]\widehat{\otimes}^{B-B}(B\otimes_{\sigma h}^{( I_{i_j},1)}B)\right)\to N^{\otimes_{eh B}(|m|+1)},$$
$$M^{w^* h(\eta)}:=\underset{j=1...|m|}{\overset{eh B}{\bigotimes}}M_{\#w^* h}^{B,\eta_{i_j},N}: \underset{j=1...|m|}{\overset{eh B}{\bigotimes}}\left([\eta_{i_j}'\cap \ell^\infty(I_{i_j},N)]\widehat{\otimes}^{B-B}(B\otimes_{w^*h \eta_{i_j}}B)\right)\to N^{\otimes_{eh B}(|m|+1)} .$$
By definition from restriction and quotient we have $M^{w^* h(\eta)}\circ [\underset{j=1...|m|}{\overset{eh B}{\bigotimes}}(1\otimes  p_{\eta_{i_j}})]=M^{\sigma e h(I)}.$
We of course finally compose to a multiplication map to $m:N^{\otimes_{eh B}(|m|+1)}\to N$. Since the maps $m\circ M^{w^* h(\eta)}\circ S_m$ and $m\circ M^{\sigma e h(I)}\circ S_m^{\sigma e h}$ are $C$ bimodular completely contractive , they extend to direct sum,  using also diagonal maps to projective tensor product, and restriction to unit balls to obtain contractions after multiplication of the norm by $R^{|m|}.$ This concludes to the definition of  $ev_{(X_1,...,X_n)},ev_{(X_1,...,X_n)}^{\sigma e h}$ respectively. From $m\circ M^{w^* h(\eta)}\circ S_m\circ (1\otimes P_{\eta(m)})=m\circ M^{\sigma e h(I)}\circ S_m^{\sigma e h}$, one deduces the expected relation $ev_{(X_1,...,X_n)}P^{\sigma e h}=ev_{(X_1,...,X_n)}^{\sigma e h}.$

It remains to check that $P^{\sigma e h}$ is the restriction of $P^{\sigma h}$. First note that using the argument for the  commutative diagram in lemma \ref{Shuffle} and with some notation from there and with $I:\underset{j=1...|m|}{\overset{e h B}{\bigotimes}}A_j\to\underset{j=1...|m|}{\overset{\sigma h B}{\bigotimes}}A_j $, for $A_j=\left( [CB(\ell^\infty(I_{i_j},N),(B\otimes_{\sigma h}^{ (I_j,1)}B)_\natural)]^\natural\right),$ and $J_{\eta(m)}$ from theorem \ref{SubmoduleEtaTHM}, one gets: $$I\circ(\underset{j=1...|m|}{\overset{eh B}{\bigotimes}}k_{\ell^\infty(I_{i_j},N),(B\otimes_{\sigma h}^{( I_{i_j},1)}B)}^B)\circ S_m^{\sigma e h}=S_m^{\sigma h}\circ (1\otimes J_{\eta(m)})$$

If similarly, $I_N: N^{\otimes_{eh B}(|m|+1)}\to  N^{\otimes_{\sigma h B}(|m|+1)}$ then it is easy to see on canonical representation that 
$M^{\sigma h(I)}\circ I=I_N\circ M^{\sigma e h(I)*}$  with 
$$M^{\sigma e h(I)*}:=\underset{j=1...|m|}{\overset{e h B}{\bigotimes}}M_{\#\sigma h*}^{B,I_{i_j},N}: \underset{j=1...|m|}{\overset{e h B}{\bigotimes}}\left( [CB(\ell^\infty(I_{i_j},N),(B\otimes_{\sigma h}^{ (I_j,1)}B)_\natural)]^\natural\right)\to N^{\otimes_{e h B}(|m|+1)},$$

so that one can use the extension relation obtained in proposition \ref{MultiplicationEtaVar}  $M_{\#\sigma h*}^{B,I_{i_j},N}\circ k_{\ell^\infty(I_{i_j},N),(B\otimes_{\sigma h}^{( I_{i_j},1)}B)}^B=M_{\#\sigma h}^{B,I_{i_j},N}$ to get  using also $m_\sigma \circ I_N=m$ (note this is the relation that has no analogue for the projection between normal and extended Haagerup products) $$m\circ M^{\sigma e h(I)}\circ S_m^{\sigma e h}=m_\sigma\circ I_N\circ M^{\sigma e h(I)}\circ S_m^{\sigma e h}=m_\sigma\circ M^{\sigma h(I)}\circ S_m^{\sigma h}\circ (1\otimes J_{\eta(m)}).$$
Since $\iota :B_{\sigma e h}\langle X_1,...,X_n:I,R,C\rangle \subset B_{\sigma h}\langle X_1,...,X_n:I,R,C\rangle$ is gathered by direct sum from $J_{\eta(m)}$, we have thus obtained our extension relation $ev_{(X_1,...,X_n)}^{\sigma  h}\circ \iota= ev_{(X_1,...,X_n)}^{\sigma  e h}.$

The $*$ algebra morphism property  for $ev_{(X_1,...,X_n)}$ is easy since so are $P^{\sigma e h}$ and $ev_{(X_1,...,X_n)}^{\sigma  h}$.
\end{proof}

For a second list of covariance maps $\eta'=(\eta'_1,...,\eta'_m).$ we will write $B_{\otimes \eta'}\langle X_1,...,X_n:\eta,R,C\rangle$ or for short, $B_{\otimes \eta'}\langle X:\eta,R,C\rangle$ the subspace of $B\langle X_1,...,X_n,Y_1,...,Y_m:(\eta,\eta'),R,C\rangle$ linear in each $Y_k$ with $Y_k$'s in increasing order of $k$ in each monomial.
One also deduces an induced complete quotient map and the corresponding $\sigma eh$-weak-* topology : $$P^{\sigma eh}_{\otimes \eta'}:B_{\sigma e h\otimes I'}\langle X_1,...,X_n:I,R,C\rangle\to B_{\otimes \eta'}\langle X_1,...,X_n:\eta,R,C\rangle.$$

We will use this space to deal with free difference quotients on this space of analytic functions. 

\begin{proposition}\label{analyticDerivation}
%

$B\langle X_1,...,X_n:\eta,R,C\rangle$ is mapped weak-* continuously completely boundedly  by the iterated free difference quotients  $\partial_{(i_1,...,i_k)}^{k}=(\partial_{X_{i_1}}\otimes 1^{\otimes k-1})\ldots \partial_{X_{i_k}}$ to 
$B_{\otimes \eta'}\langle X:\eta,S,C\rangle$  with $\eta'=(\eta_{i_1},...,\eta_{i_k}).$
Moreover, $P^{\sigma eh}_{\otimes \eta'}\partial_{(i_1,...,i_k)}^{k}=\partial_{(i_1,...,i_k)}^{k}P^{\sigma eh}$ so that $\partial_{(i_1,...,i_k)}^{k}$ is also $\sigma eh$-weak-* continuous.

Finally $B_{\otimes \eta'}\langle  X_1,...,X_n:\eta,R,C\rangle$ is a (matrix normed) bimodule over $B\langle X_1,...,X_n:\eta,R,C\rangle$ and, for $N\supset B$ a finite von Neumann algebra, $P\in B_{\otimes \eta'}\langle  X_1,...,X_n:I,R,C\rangle$ defines a map (still written) $$P:\prod_{i=1}^n[\eta_i'\cap \ell^\infty(I_i,N)]_R\to  N^{\otimes_{w^* h}\eta'\circ E_B},$$  by evaluation (in the algebraic case and then uniquely extended) and $ev_{(X_1,...,X_n)}^{\otimes \eta'}:P\mapsto P(X_1,...,X_n)\in W^*(B,X_1,...,X_n)^{\otimes_{w^* h}\eta'\circ E_B},$ for $X_i\in [\eta_i'\cap \ell^\infty(I_i,N)]_R$, with $$ev_{(X_1,...,X_n)}^{\otimes \eta'}\circ P^{\sigma eh}_{\otimes \eta'}=P_{N,\eta'}\circ ev_{(X_1,...,X_n)}^{\sigma e h\otimes I'}$$ with the canonical map $P_{N,\eta'}:N^{\otimes_{\sigma e  h}[I',1]}\to N^{\otimes_{w^* h}\eta'\circ E_B}$  so that $ev_{(X_1,...,X_n)}^{\otimes \eta'}$ is a $C$-bimodular contraction which is continuous from $\sigma eh$-weak-* to weak-* topology, and compatible with the $B\langle X_1,...,X_n:\eta,R,C\rangle$ module structure.

\end{proposition}
\begin{proof}[Sketch of Proof]
The proof being similar, we mostly leave it to the reader. We only emphasize some key tools. We use the following commutation relations for $B\subset N$, $\iota_{w^* h \eta_i}:B\otimes_{w^*h \eta_i}B\to N\otimes_{w^*h \eta_i}N$ the inclusion from theorem \ref{SubmoduleEtaTHM} is defined in the proof there to satisfy with $\iota_{\sigma h }:B\otimes_{\sigma h}^{(I_i,1)}B\to N\otimes_{\sigma h}^{(I_i,1)}N$ the relation $p_{\eta_i\circ E_B}\circ\iota_{\sigma h }=\iota_{w^* h \eta_i}\circ p_{\eta}.$ Moreover, one also uses $p_{\eta_i\circ E_B}$ is $N$ bimodular between normal bimodule so that bimodularity extends to normal or extended haagerup products. Finally we use the relation $ ev_{(X_1,...,X_n)}^{\sigma  h \otimes I'}\circ\iota_{\otimes I'}=\iota_{N,I'}\circ ev_{(X_1,...,X_n)}^{\sigma e h \otimes I'}$ stated in proposition \ref{analyticDerivationsigma} to check $ev_{(X_1,...,X_n)}^{\sigma e h \otimes I'}$ is continuous from the topology induced from the weak-* topology by $\iota_{\otimes I'}$ to the topology induced on $N^{\otimes_{\sigma e  h}[I',1]}$ for the topology induced by the weak-* topology of $N^{\otimes_{\sigma   h}[I',1]}$, i.e. the initial topology for $\iota_{N,I'}.$ We then use $P_{N,\eta'}$ is continuous from this topology to the weak-* topology. Since the $\sigma eh$-weak-* is merely the final topology via  $P^{\sigma eh}_{\otimes \eta'}$ from the one induced from the weak-* topology by $\iota_{\otimes I'}$, the relation we stated $ev_{(X_1,...,X_n)}^{\otimes \eta'}\circ P^{\sigma eh}_{\otimes \eta'}=P_{N,\eta'}\circ ev_{(X_1,...,X_n)}^{\sigma e h\otimes I'}$ implies exactly the stated continuity.
\end{proof}

\subsection{Conjugate variables and analytic relations}

We recall the definition from \cite{Shl00} of $\eta$-conjugate variables slightly extended it to a family of ($\tau$-symmetric) covariance maps $\eta=(\eta_1,...,\eta_n)$ with $\eta_i:B\to B\otimes B(L^2(I_i))$. The case stated there is the case with $|I_i|=1$.

Note that with the remark before lemma 3.2, the original definition uses a change concrete realization of te variable with same law to look at algebraically free variables in a way perfectly similar to the more recent definition in \cite{MS14} (in the case $\eta=\tau$). Our phrasing is more similar to this version. 

\begin{definition}
Let $(N,\tau)\supset B$ a $W^*$ probability space and $\eta$ as above. Let $X=(X_1,...,X_n)$, $X_i\in \ell^\infty(I_i,N)$. We say that $(\xi^{(1)},...,\xi^{(n)})$ with $\xi_i\in \ell^\infty(I_i,L^1(W^*(B,X)))$ are (first-order) \emph{conjugate variables for $X$ relative to $\eta$} if for any $P\in B\langle X_1,...,X_n\rangle\subset B_{\sigma e h}\langle X_1,...,X_n:I,R,C\rangle$, the algebra generated by $B,X_1,..,X_n$ in it we have :
$$\forall i, \forall j\in I_i \ \ \ \ \tau(P(X) \xi^{(i)}_j)=\tau(S^{(i)}_j[(\partial_{X_i}P)(X)]\#S^{(i)}), $$

with $(\partial_{X_i}P)(X)\in N\otimes_{\sigma h}^{(I_i,1)}N$ defined in proposition \ref{analyticDerivationsigma}, $(S^{(i)}_j)_{j\in I_i}$ a family of $\eta_i$ semicircular variables relative to $N$ and $.\#S^{(i)}$ is defined before proposition \ref{InclusionEHCovariance}.
\end{definition}

\begin{lemma}
Let $(\xi^{(1)},...,\xi^{(n)})$ be conjugate variables for $X$ relative to $\eta$, $X_i\in \ell^\infty(I_i,N)$. 
\begin{enumerate}
\item If $||X_i||< R$ for all $i$ then, for any $P\in B_{\sigma h}\langle X_1,...,X_n:I,R,C\rangle$ we have $$\forall i, \forall j\in I_i \ \ \ \ \tau(P(X) \xi^{(i)}_j)=\tau(S^{(i)}_j[(\partial_{X_i}P)(X)]\#S^{(i)}). $$
\item If $\xi^{(i)}\in \ell^\infty(I_i,W^*(B,X))$ (we say bounded conjugate variables) then $$\xi^{(i)}\in \eta_i'\cap  \ell^\infty(I_i,W^*(B,X)).$$
\end{enumerate}
\end{lemma}
\begin{proof}
(1) is obvious from the weak-* continuity and density in theorem \ref{analyticEvaluationsigma}, proposition \ref{analyticDerivationsigma} for evaluations and free difference quotients and the one of $d\mapsto d\#S^{(i)}$ from proposition \ref{MultiplicationEtaVar}

(2) For $d_1,d_2\in D $ since $D$ is in the kernel of $\partial_{X_i}, $ then $$\tau(P(X) d_1\xi^{(i)}_jd_2)=\tau(S^{(i)}_j[(\partial_{X_i}(d_2Pd_1)(X)]\#S^{(i)})=\tau(d_1S^{(i)}_jd_2[(\partial_{X_i}(P)(X)]\#S^{(i)}). $$ 
Now, since $d\mapsto d\#\xi^{(i)}$, and $x\mapsto x\#S^{(i)}$ are weak-* continuous from proposition \ref{MultiplicationEtaVar} and $D\otimes_{alg}\ell^1(I_i)\otimes_{alg} D \subset D\otimes_{\sigma h}^{(I_i,1)} D$ is weak-* dense, one deduces, for all $d\in D\otimes_{\sigma h}^{(I_i,1)} D$ :
$$\tau(P(X) d\#\xi^{(i)})=\tau(d\#S^{(i)}[(\partial_{X_i}(P)(X)]\#S^{(i)}). $$ 
Thus since if $d\in Ker(p_{\eta_i})$, we have  $d\#S^{(i)}=0$ by definition and since $P(X)$ are weak-* dense in $W^*(B,X)$ on deduces $d\#\xi^{(i)}=0$ as expected.
\end{proof}
\begin{exemple}We have the typical example of \cite[Prop 3.12]{Shl00}. Let $X_1,...,X_n\in (N,\tau)$  free with amalgamation over $B$ with $S_1,...,S_n$ free with amalgamation with respect to one another and $S_i$ an $\eta_i$ semicircular variable over $B$, then for $\epsilon>0,$ $Y=(X_1+\sqrt{\epsilon} S_1,...,X_n+\sqrt{\epsilon} S_n)$ has conjugate variable $(\frac{1}{\sqrt{\epsilon}}E_{W^*(Y,B)}(S_1),...,\frac{1}{\sqrt{\epsilon}}E_{W^*(Y,B)}(S_n)).$
\end{exemple}

We are now ready to get an absence of analytic relations in the spirit of \cite[lemma 37]{dabrowski:SPDE}. Our preparatory work reduced the proof to be formally  the same.
\begin{theorem}\label{MainRelation}
Let $(\xi^{(1)},...,\xi^{(n)})$ be conjugate variables for $X$ relative to $\eta$ with $||X_i||<R$. Then the evaluation map $ev_{(X_1,...,X_n)}^{\sigma h}$ on $B_{\sigma h}\langle X_1,...,X_n:I,R,B\rangle$ has its kernel included in the one of the complete quotient map $P^{\sigma h}:B_{\sigma h}\langle X_1,...,X_n:I,R,B\rangle\to B\langle X_1,...,X_n:\eta,R,B\rangle.$

Especially, if we assume $X_i\in \eta_i'\cap\ell^\infty(I_i,N)$  Then the evaluation map $ev_{(X_1,...,X_n)}$ on $B\langle X_1,...,X_n:\eta,R,B\rangle$ from theorem \ref{analyticEvaluation} is one-to-one.
\end{theorem}

\begin{proof}
Consider $P$ in the kernel of $ev_{(X_1,...,X_n)}^{\sigma h}$.
From the derivation property, of $\partial_{X_i}$ on $B_{\sigma h}\langle X_1,...,X_n:I,R,C\rangle$ we have for  
$P,Q,R\in B_{\sigma h}\langle X_1,...,X_n:I,R,C\rangle$ we have \begin{align*} \tau&(Q(X)S^{(i)}_jR(X)[(\partial_{X_i}P)(X)]\#S^{(i)})\\&=\tau(R(X)P(X)Q(X) \xi^{(i)}_j)-\tau(S^{(i)}_j[(\partial_{X_i}R)PQ+RP(\partial_{X_i}Q)](X)\#S^{(i)})=0. \end{align*}
By weak-* density, one deduces that one can replace $Q(X),R(X)$ by $a, b\in W^*(B,X)$ and then again by density one deduces $ev_{(X_1,...,X_n)}^{\otimes \eta_i}[(\partial_{X_i}P)]=0$ in 
$W^*(B,X)\otimes_{w^*h\eta_i\circ E_B}W^*(B,X)$ since this  space has been defined so that $.\#S^{(i)}$ is one-to-one.

Let us show by induction on $m$ that for $\eta'=(\eta_{i_1},...\eta_{i_m})$,$ev_{(X_1,...,X_n)}^{\otimes \eta'}(\partial^m_{(i_1,...,i_m)}(P))=0$ in $W^*(B,X)^{\otimes_{w^*h}\eta'\circ E_B}.$

Assuming by induction the length $k$, one uses lemma \ref{PartialEvaluation} with $n=X_{i_1}...X_{i_m}.$ From the induction step and the first formula in the lemma, one gets $ev_{(X_1,...,X_n)}^{\otimes \eta_i}(Q(a))=0$ so that by the reasoning above $ev_{(X_1,...,X_n)}^{\otimes \eta_{i_{m+1}}}[(\partial_{X_{i_{m+1}}}Q(a))]=0$ and the the second formula in the lemma implies $ev_{(X_1,...,X_n)}^{\otimes \eta(nX_{i_{m+1}})}(\partial^{m+1}_{(i_1,...,i_m,i_{m+1})}(P))=0$ in $W^*(B,X)^{\otimes_{w^*h}\eta(nX_{i_{m+1}})\circ E_B}.$
This completes the induction step.

This implies $ev_{(X_1,...,X_n)}^{\sigma (e) h-an}(\partial^{|n|}_{n(X)}P)=0$ and thus from lemma \ref{TaylorAnalytic} in seeing $sX_i=X_i+(s-1)X_i$, that for s close to 1, $ev_{(sX_1,...,sX_n)}^{\otimes \eta'}(\partial^k_{(i_1,...,i_k)}(P))=0$ in $W^*(B,X)^{\otimes_{w^*h}\eta'\circ E_B}.$
Then iterating, one gets the same for all $s\in [0,1]$ and thus especially $ev_{(0,...,0)}^{\otimes \eta'}(\partial^k_{(i_1,...,i_k)}(P))=0$ in $B^{\otimes_{w^*h}\eta'}\subset W^*(B,X)^{\otimes_{w^*h}\eta'\circ E_B}$ (from theorem \ref{SubmoduleEtaTHM} for the inclusion) This is exactly $P^{\sigma h}(P)=0$ in $B\langle X_1,...,X_n:\eta,R,B\rangle$ seen componentwise in the $\ell^1$ direct sum. This concludes.\end{proof}


\begin{thebibliography}{GMS07}


\bibitem[AP02]{AP} C. Anantharaman and C. Pop, {\em Relative tensor products and
infinite C$^*$-algebras,} J. Operator Theory {\bf 47} (2002), 389--412.

  

\bibitem[B79]{Bherg}
J.~Bergh, \emph{ On the Relation between the Two Complex Methods of Interpolation},  Indiana Univ. Math. J.  28 (1979),  775--778. 



\bibitem[B97]{B97}
D.~Blecher, \emph{A new approach to Hilbert {C}$^*$-
modules}, Math. Ann. \textbf{307}, 253--290 (1997). 
\bibitem[B97b]{B97b}
D.~Blecher, \emph{On self-dual $W^*$-modules}, in Fields Institute Communications  \textbf{13}, 65--80 (1997). 


\bibitem[BLM]{BLM}  D. P. Blecher and C.  Le Merdy, {\em Operator algebras and their modules--an 
operator space approach,} LMS Monographs, Oxford Univ. Press, Oxford, 2004.

\bibitem[BP91]{BlecherPaulsen}
D.~Blecher and V.~Paulsen, \emph{Tensor product of operator spaces}, JFA \textbf{99}
(1991), 262--292.


\bibitem[BS92]{BlecherSmith}
D.~Blecher and R.~Smith, \emph{The dual of the Haagerup tensor product}, . London Math. Soc. \textbf{45}
(1992), 126--144.







\bibitem[Ceb13]{Ceb13}
G.~C\'ebron, \emph{Free Convolution Operators and Free Hall Transform}, preprint arXiv:1304.1713v3.


\bibitem[Dab10]{dabrowski:SPDE}
Y.~Dabrowski, \emph{A free stochastic partial differential equation}, Preprint,
  arXiv.org:1008:4742, 2010.


\bibitem[EE88]{EffrosExel}E.G. Effros and  R. Exel, \emph{On multilinear double commutant theorems}, in  \emph{Operator algebras and applications,
Vol. 1}, London Mathematical Society, Lecture Note Series, Vol. 135, Cambridge University Press,
Cambridge, 1988, pp. 81--94.


\bibitem[EK]{EffrosKishimoto}
E.~Effros and A.~Kishimoto, \emph{Module maps and {H}ochschild-{J}ohnson cohomology}, Indiana Univ. Math. J.
\textbf{36} (1987), 257--276.

\bibitem[ER91]{EffrosRuanFacto}
E.~Effros and Z.~Ruan, \emph{Self-Duality for the Haagerup Tensor Product
and Hilbert Space Factorizations}, JFA \textbf{100}
(1991), 257--284.


\bibitem[ER00]{EffrosRuan}
E.~Effros and Z.~Ruan, \emph{Operator Spaces}, Clarendon Press Oxford, 2000.

\bibitem[ER03]{EffrosRuanDual}
E.~Effros and Z.~Ruan, \emph{Operator space tensor products and Hopf convolution algebras}, J. Operator Th.
\textbf{50} (2003), 131--156.




\bibitem[GMS06]{guionnet-edouard:combRM}
A.~Guionnet and E.~Maurel-Segala, \emph{Combinatorial aspects of matrix
  models}, ALEA Lat. Am. J. Probab. Math. Stat. \textbf{1} (2006), 241--279.
  \MR{2249657 (2007g:05087)}


\bibitem[GS12]{alice-shlyakhtenko-transport}
A.~Guionnet and D.~Shlyakhtenko, \emph{Monotone free transport}
arXiv:1204.2182 (2012)


\bibitem[K91]{Kouba}
O.~Kouba, \emph{ On the Interpolation
of Injective
or Projective Tensor Products of Banach
Spaces},   J. Funct. Anal. 96. 38--61
(1991). 



\bibitem[M95]{M95} B. Magajna, {\em The Haagerup norm on the tensor  product 
of operator modules,}  J. Funct. Anal. {\bf 129} (1995), 325--348. 

\bibitem[M97]{M97} B. Magajna, {\em Strong operator modules and the Haagerup tensor
product,}  Proc. London Math. Soc {\bf 74} (1997), 201--240.

\bibitem[M05]{M05} B. Magajna, {\em Duality and normal parts of operator modules,}  J. Funct. Anal. {\bf 219} (2005), 306--339.

\bibitem[MSW14]{MS14} T. Mai, R. Speicher and M. Weber {\em Absence of algebraic relations and of zero divisors under the assumption of full non-microstates free entropy dimension } arXiv:1502.06357


\bibitem[MP95]{MathesPaulsen} D. Mathes and V. Paulsen, {\em Operator Ideals and Operator Spaces,}  J. Funct. Anal. {\bf 103} (1995), 1763--1772.


\bibitem[OP97]{OP97}
T.~Oikhberg and G.~Pisier, \emph{The ``maximal" tensor product of operator spaces}, Proc. of the Edinburgh Math. Soc.
\textbf{42} (1999), 267--284.


\bibitem[P96]{PisierOH}
G.~Pisier, \emph{The Operator Hilbert Space OH, Complex Interpolation and Tensor Norms}, Memoirs of the AMS 1996, 122. 


\bibitem[P03]{PisierBook}
G.~Pisier, \emph{Introduction to Operator Space Theory}, Cambridge University Press, 2003.



 \bibitem[SS98]{SS98}   A.~Sinclair, R.~Smith \emph{Factorization of Completely Bounded Bilinear Operators and Injectivity}, J. Funct. Anal. {\bf 157} (1998), 62--87.

\bibitem[Shl99]{S99}
D.~Shlyakhtenko, \emph{A-Valued Semicircular Systems}, J. Funct. Anal. {\bf 166} (1999), 1--47.

\bibitem[Shl00]{Shl00}
\bysame, 
\newblock {Free entropy with respect to a completely positive map}.
\newblock {\em American Journal of Mathematics}, 122:45--81, 2000.


\bibitem[Shl04]{shlyakht:someEstimates}
\bysame, \emph{Some estimates for non-microstates free entropy
  dimension with applications to $q$-semicircular families}, Int. Math. Res.
  Notices \textbf{51} (2004), 2757--2772.

\bibitem[Shl09]{shlyakht:lowerEstimates}
\bysame, \emph{Lower estimates on microstates free entropy dimension}, Analysis
  and PDE \textbf{2} (2009), 119--146.
  
  \bibitem[Shl14]{S14}
\bysame, \emph{Free Entropy Dimension and Atoms}, arXiv:1408.0580.
  
  
\bibitem[Sp98]{Sp98}
R.~{{Speicher}}.
\newblock {Combinatorial theory of the free product with amalgamation and operator-valued free probability theory}.
\newblock {\em {Memoirs of the AMS}}, 627, 1998.















\bibitem[Voi94]{Vo2}
D.~{{Voiculescu}}.
\newblock  {The analogs of entropy and of Fisher's information
  measure in free probability theory, {II}}.
\newblock {\em   Invent. Math.  } 118(3):411--440, 1994.




\bibitem[Voi98]{Vo5}
\bysame, 
\newblock {The analogs of entropy and of Fisher's information measure in free
  probability theory, {V} : Non commutative Hilbert Transforms}.
\newblock {\em Inventiones mathematicae}, 132:189--227, 1998.

\bibitem[Voi00]{VoCG}
\bysame, 
\newblock {A note on Cyclic gradients}.
\newblock {\em Indiana Univ. Math. J.} 49(3):837--841, 2000.

\bibitem[Voi02]{VoS}
\bysame, 
\newblock {Free Entropy}.
\newblock {\em Bulletin of the London Mathematical Society}, 34(3):257--278,
  2002.


\end{thebibliography}
\end{document}